\documentclass[a4paper, 10pt, twoside, notitlepage]{amsart}

\usepackage[utf8]{inputenc}
\usepackage{color}
\usepackage{amsmath} 
\usepackage{amssymb} 
\usepackage{amsthm}
\usepackage{geometry}
\usepackage{graphicx}
\usepackage{esint}
\usepackage[colorlinks=true,linkcolor=blue]{hyperref}

\usepackage{amsmath, amssymb, amsfonts, amsthm, amsopn, cancel}
\usepackage{graphicx,tikz}
\usepackage[all]{xy}
\usepackage{amsmath}
\usepackage{bbm}
\usepackage{multirow, longtable, makecell, caption, array, enumitem}
\usepackage{amssymb}
\usepackage{mathtools}
\usepackage{bbold}
\usepackage{bookmark}
\usepackage{hyperref}
\usepackage{esint}
\allowdisplaybreaks

\calclayout

\usepackage{fancyhdr}
\fancypagestyle{firstpage}{
  \fancyhf{}
  
}

\theoremstyle{plain}
\numberwithin{equation}{section}
\newtheorem{thm}{Theorem}[section]
\newtheorem{prop}{Proposition}[section]
\newtheorem{lem}[prop]{Lemma}
\newtheorem{cor}[prop]{Corollary}

\newtheorem{defi}[prop]{Definition}
\newtheorem{rmk}[prop]{Remark}

\newtheorem{example}[prop]{Example}

\newcommand {\R} {\mathbb{R}}

\newcommand {\supp} {\text{supp}}
\newcommand {\diam} {\text{diam}}

\def\bC{\mathbb{C}}

\def\bN{\mathbb{N}}

\def\bR{\mathbb{R}}

\def\cB{\mathcal{B}}

\def\cK{\mathcal{K}}
\def\cL{\mathcal{L}}
\def\cM{\mathcal{M}}
\def\cN{\mathcal{N}}

\def\cS{\mathcal{S}}

\DeclareMathOperator{\vol}{vol}

\def\cB{\mathcal{B}}

\def\cK{\mathcal{K}}
\def\cL{\mathcal{L}}
\def\cM{\mathcal{M}}
\def\cN{\mathcal{N}}

\def\cS{\mathcal{S}}

\def\lc{\left(}
\def\rc{\right)}

\def\supp{\operatorname{supp}}

\def\*b{*_{\bullet}}

\def\S0{S^0_{0,0}}

\def\Bd'{B_{\delta'}}

\def\cBd'{\bar{B}_{\delta'}}

\newcommand{\diff}{\mathrm{d}}
\newcommand{\Div}{\mathrm{div}}

\newcommand{\Span}{\mathrm{span}}
\newcommand{\eps}{\varepsilon}

\begin{document}

\title{Instability estimates for the recovery of absorption in the diffusive regime of radiative transfer}
\author{Elena Dematt\`e} 
\address{Max-Planck Institute for Mathematics in the Sciences, Inselstraße 22, 04103 Leipzig, Germany}
\email{elena.dematte@mis.mpg.de}
\author{Alessandro Felisi}
\address{Institute for Applied Mathematics, University of Bonn, 53115 Bonn, Germany}
\email{afelisi@uni-bonn.de}
\author{Angkana Rüland}
\address{Institute for Applied Mathematics, University of Bonn, 53115 Bonn, Germany}
\email{rueland@uni-bonn.de}
\author{Juan J.~L.~Vel\'azquez}
\address{Institute for Applied Mathematics, University of Bonn, 53115 Bonn, Germany}
\email{velazquez@iam.uni-bonn.de}

\begin{abstract}
We revisit the instability properties of the recovery of the absorption coefficient for the radiative transfer equation in the diffusive regime. To this end, we develop a rather robust framework building on \cite{KRS21} which allows us to deal with nonlinear critical stability transition phenomena. In particular, this permits us to consider rather general geometries based on the identification of compression properties of the forward operator. Given the albedo operator as the measurement data, we show that in the regime of vanishing Knudsen number there is a transition from Hölder to logarithmic stability in the inverse problem for the radiative transfer equation. As a central ingredient, we rely on suitable a priori estimates for the radiative transfer equation which we deduce by building on the strategy from \cite{DV25}.
\end{abstract}

\maketitle
\tableofcontents

\section{Introduction}
\label{sec:introduction}
We study the stability properties of an inverse problem for the radiative transfer equation in the diffusive regime. We prove that there is a critical transition behaviour between Hölder and logarithmic stability estimates, even in the presence of rather general geometries. 

To this end, we consider the (stationary) radiative transfer equation: 
\begin{align}
\label{eq:rad_trans}
\begin{split}
v \cdot \nabla u(x,v) + \left( K_n \sigma_a(x) + \frac{1}{K_n} \sigma_{s}(x) \right) u(x,v) & = \frac{1}{K_n} \sigma_{s}(x) \langle u  \rangle \mbox{ for } (x,v) \in D \times \Omega,\\
u(x,v) & = f(x) \mbox{ on } \Gamma_{-}.
\end{split}
\end{align}
Here $u$ models the particle density, $\sigma_s$ models the scattering coefficient and $\sigma_a$ is the unknown absorption coefficient. The parameter $K_n>0$ models the Knudsen number. The set $D \subset \R^{d}$ denotes the spatial domain, $\Omega = S^{d-1}$ the velocity domain and $\Gamma_- = \{(x,v)\in\Gamma\colon\,   v\cdot n_x<0\}$ is the incoming boundary. Here $n_x$ denotes the outer unit normal vector at $x\in \partial D$. In what follows below, we will assume that the scattering coefficient $\sigma_s$ is constant and that the boundary datum $f$ is \textit{isotropic}, i.e., that it depends only on the spatial variable $x$ and not on the velocity variable $v$.

For our inverse problem we assume that the  measurements are given in terms of the angularly averaged albedo operator
\begin{align*}
\Lambda_{\sigma_a}: f(x) \mapsto \Lambda_{\sigma_a}f(x) := \frac{1}{K_n} \int\limits_{\Omega} v \cdot n_x u(x,v) d\mu(v), \ x \in \partial D,
\end{align*}
where $\mu$ denotes a given velocity distribution. 
We remark that, up to a constant, formally, the operator converges to the Dirichlet-to-Neumann operator associated to the Schr{\"o}dinger operator $(-\Delta+q)$ for some potential $q$ depending on $\sigma_a$ -- see, for instance, \cite[Theorem~1.1]{LLU19}.

We are interested in the stability properties of the inverse problem of reconstructing $\sigma_a$ from the knowledge of $\Lambda_{\sigma_a}$ in the regime in which the Knudsen number $K_n$ tends to zero.

We highlight that our problem setting is in a critical regime between two limiting extremes: On the one hand, the limit $K_n \rightarrow \infty$ corresponds to a pure transport regime, while, on the other hand, for $K_n \rightarrow 0$ the problem turns into the elliptic Calder\'on problem. One hence expects a transition between the associated stability moduli with Hölder stability dominating the regime of large Knudsen numbers, while logarithmic features enter in the setting of small Knudsen numbers. Such a transition had been earlier discussed in \cite{LLU19} in terms of upper bounds and in \cite{ZZ19} for lower bounds. The main novelty of our work is to provide a \emph{robust framework} for lower bound estimates that also in the setting of this critical problem allows one to deal with rather \emph{general geometries} instead of requiring extremely strong symmetry conditions as in earlier work. Moreover, we provide self-contained a priori estimates for the radiative transfer equation building on the strategy developed in \cite{DV25}.

\subsection{Main result}
\label{sub:main_result}

As one of our main objectives, we seek to robustify the findings from \cite{ZZ19} by adopting the perspective  from \cite{KRS21}, suitably adapted to our critical transition problem. To this end, we formulate an abstract framework which builds on the arguments in \cite{KRS21} and which is particularly well adapted to our transition regime (see Sec.~\ref{sec:instability}).

More precisely, as our main result, we will prove the following lower bound on the transition between the Hölder and logarithmic stability regimes.

\begin{thm}
\label{thm:main}
Let $s_1=\frac{9}{2}+\lfloor\frac{d}{2}\rfloor$, let $s>2d+s_1$, let $\gamma>\lfloor\frac{d}{2}\rfloor+3$, let $K\Subset D$, let $M>0$, and let
\begin{equation}
\label{eq:definition_S}
    \cS = \{\sigma_a\in L^{\infty}_+(D)\cap H^{\gamma}(D)\colon\ \supp(\sigma_a)\subset K,\ \|\sigma_a\|_{H^{\gamma}(D)}\leq M\}.
\end{equation}

Then, the following holds: If there exists some modulus of continuity $\omega:[0,1) \rightarrow [0,\infty)$ such that
\begin{align*}
\|\sigma_{a,1}- \sigma_{a,2}\|_{L^2(D)} \leq \omega( \|\Lambda_{\sigma_{a,1}} - \Lambda_{\sigma_{a,2}}\|_{H^{s} \rightarrow H^{-s}}),\quad \sigma_{a,j}\in\cS,
\end{align*}
then, for small $t>0$, the following lower bound is valid:
\begin{equation}
\label{eq:inst_main2}
    \omega(t) \geq C_1 \min\{|\log(C_2 t)|^{-2\gamma},K_n^{-\gamma/(s-s_1)}t^{8\gamma/(s-s_1)}\}\lc 1+|\log(C_3 t)| \rc^{-\gamma/d}.
\end{equation}
Equivalently, for every small $\varepsilon>0$, there exist absorption coefficients $\sigma_{a,1}, \sigma_{a,2} \in \mathcal{S}$ and associated measurement operators $\Lambda_{\sigma_{a,1}}, \Lambda_{\sigma_{a,2}}$ satisfying the following bounds
\begin{align}
\begin{split}
\label{eq:inst_main}
&\|\sigma_{a,1} - \sigma_{a,2}\|_{L^{2}(D)} \geq \varepsilon,\\ &\|\Lambda_{\sigma_{a,1}}- \Lambda_{\sigma_{a,2}}\|_{H^{s} \rightarrow H^{-s}} \leq C\max\{\exp(-c\varepsilon^{-1/d}),K_n^{1/4}\varepsilon^{(s-s_1)/4d}\}\cdot(\text{log factors}).
\end{split}
\end{align}
\end{thm}

\begin{center}
    \begin{figure}
        \centering
        \includegraphics[height=5cm]{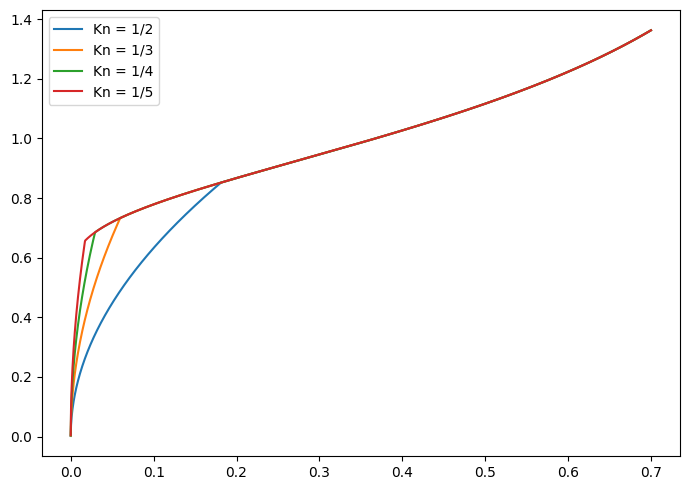}
        \caption{Plot of $t\mapsto \min\{|\log(C_2 t)|^{-2\gamma},K_n^{-\gamma/(s-s_1)}t^{8\gamma/(s-s_1)}\}$ for some values of the parameters and for variable $K_n$. For small $t$, the behaviour of the modulus of continuity is Hölder, while for large $t$ it transitions to logarithmic. The scaling of the critical threshold $t_*$ is of order $K_n^{\beta}$ for some $\beta>0$, up to log factors.}
        \label{fig:plot}
    \end{figure}
\end{center}
We highlight that the above result, in particular recovers the instability results from \cite{ZZ19}. 
Indeed, relying on the more robust framework from \cite{KRS21} in the form laid out in Sec.~\ref{sec:instability}, we are able to relax the strong geometric assumptions from \cite{ZZ19}. In particular, while in the setting of \cite{ZZ19} it is assumed that
\begin{itemize}
\item $\sigma_s(x) = \sigma_s$ is independent of $x$,
\item $D = B_1(0)$,
\item $\sigma_a \in \mathcal{S}:=\{\sigma_a \in L^{\infty}(D): \ \sigma_a \geq 0, \ \supp (\sigma_{a}) \subset B_{r_0}(0), \ \sigma_a \in C^{q}(B_{r_0}(0)) \}$ for some $r_0 \in (0,1)$, and $q>0$,
\end{itemize}
in our instability framework we can investigate rather \emph{general geometries} $D$ and \emph{compact sets} $K\Subset D$ (which are not necessarily balls). This is a consequence of relying on more robust mapping properties which do not require symmetry conditions and, in particular, avoid the determination of explicit singular values and bases. As in \cite{KRS21} we instead rely on compression properties which follow from the mapping properties of the forward operator.

\subsection{Outline of the proof}
\label{sub:outline_proof}
In order to prove our main result (Theorem~\ref{thm:main}), we seek to use the abstract instability framework carried out in \cite{KRS21}. We briefly outline the main ideas of the proof in what follows below.

Let $X=L^{\infty}_+(D)$, let $Y=\cB(H^s(\partial D),H^{-s}(\partial D))$, where $s$ is as in Theorem~\ref{thm:main}, and let $\cS\subset X$ be a compact set of absorption coefficients supported in $K\Subset D$. We consider the forward operator 
\begin{equation}
    \Lambda\colon X \longrightarrow Y
\end{equation}
associated to the inverse problem, mapping $\sigma_a\in L^{\infty}_+(D)$ to the corresponding albedo operator $\Lambda_{\sigma_a}$ as defined in \eqref{eq:albedo_definition}. A conditional stability estimate for the inverse problem, such as
\begin{equation}
    \|\sigma_{a,1}-\sigma_{a,2}\|_{L^{\infty}(D)} \leq \omega(\|\Lambda_{\sigma_{a,1}}-\Lambda_{\sigma_{a,2}}\|),\quad 
    \sigma_{a,1},\sigma_{a,2}\in\cS,
\end{equation}
implies the following condition on the covering numbers of $K$ and of $\Lambda(\cS)$:
\begin{equation}
    \cN(\cS,\omega(\varepsilon),d_X) \leq \cN(\Lambda(\cS),\varepsilon,d_Y);
\end{equation}
we refer the reader to Sec.~\ref{sec:instability} for the definition and properties of covering numbers. Therefore, in order to prove a lower bound for $\omega=\omega(\varepsilon)$ as in Theorem~\ref{thm:main}, it suffices to prove a lower bound for $\cN(\cS,\delta,d_X)$ and an upper bound for $\cN(\Lambda(\cS),\varepsilon,d_Y)$.

The lower bound for $\cN(\cS,\delta,d_X)$ comes from known entropy estimates for the embedding of Sobolev spaces into $L^{\infty}(D)$; see, for instance, \cite{ET96} and Proposition~\ref{prop:non-negativity}. A crucial step in our proof is then to provide an upper bound for $\cN(\Lambda(\cS),\varepsilon,d_Y)$. In order to do this, we adopt the comparison strategy from \cite{KRS21}: more precisely, we seek to derive an estimate of the following type (see Proposition~\ref{prop:estimates})
\begin{equation}
\label{eq:albedo_est}
    \|(\Lambda_{\sigma_a}-\Lambda_0)f\|_{H^{-1/2}(\partial D)} \leq C( \|\rho_{0,0}\|_{H^{s_0}(K)}+K_n \|f\|_{H^{s_1}(\partial D)} ),\quad \sigma_a\in\cS,\ f\in H^{s_1}(\partial D),
\end{equation}
where $\Lambda_0$ is the albedo operator relative to the zero absorption coefficient, $\rho_{0,0}$ is defined in Sec.~\ref{sub:diffusion}, $K_n$ is the Knudsen number and $s_0,s_1>0$. Roughly speaking, the structure of this estimate reflects the diffusion approximation of radiative transfer, namely the fact that a solution to the radiative transfer equation approaches a solution $\rho_{0,0}$ to an elliptic boundary value problem (BVP) as $K_n\rightarrow 0^+$; the second terms accounts for a remainder that vanishes as $K_n\rightarrow 0^+$.

Equation~\eqref{eq:albedo_est} implies the existence of an orthonormal basis $(f_k)_k$ of $H^s(\partial D)$ such that
\begin{align*}
    \|(\Lambda_{\sigma_a}-\Lambda_0)f_k\|_{H^{-1/2}(\partial D)} \leq C( \exp(-ck^{\mu})+K_n k^{-\nu} ),\quad \sigma_a\in\cS
\end{align*}
for some $C, c,\mu,\nu>0$, which in turn implies the desired upper bound on $\cN(\Lambda(\cS),\varepsilon,d_Y)$ -- see Sec.~\ref{sec:instability}. The exponential decay in $k$ of the first term is related to the exponential instability of unique continuation for the solution of an elliptic equation from $K\Subset D$ to the boundary $\partial D$, while the power decay of the second term comes from entropy estimates for embeddings between Sobolev spaces.

We are then left to prove \eqref{eq:albedo_est}. In order to do this, we consider a test function $g\in H^{1/2}(\partial D)$ and its harmonic extension $\tilde{g}\in H^1(D)$. The following formula holds (see Lemma~\ref{lem:albedo_formula}):
\begin{equation}
\label{eq:albedo_formula}
    \langle (\Lambda_{\sigma_a}-\Lambda_0)f,g \rangle = -\int_{D} \sigma_a \langle u_a \rangle \tilde{g} + \frac{1}{K_n} \int_{D} \langle v(u_a-u_0)\rangle\cdot\nabla{\tilde{g}}.
\end{equation}
Here, $\langle \cdot, \cdot \rangle$ denotes the duality pairing between the respective Sobolev spaces on the left hand side.
In this context, the functions $u_a$ and $u_0$ solve \eqref{eq:rad_trans} with absorption coefficients $\sigma_a$ and $0$, respectively, and boundary datum $f$. The idea to obtain \eqref{eq:albedo_est} is to plug in \eqref{eq:albedo_formula} the diffusion approximation expansion for both $u_a$ and $u_0$; the first terms in the expansion will be estimated by $\|\rho_{0,0}\|_{H^{s_0}(K)}$, while the remainders will be estimated by $K_n\|f\|_{H^{s_1}(\partial D)}$.

The diffusion approximation expansions to the first order for $u_a,u_0$ are given by
\begin{align*}
    &u_a \sim \psi_{a,0} + \psi_{a,1}K_n + \phi_{a,1}K_n,\\
    &u_0 \sim \psi_{0,0} + \psi_{0,1}K_n + \phi_{0,1}K_n,
\end{align*}
where $\psi_{a,0},\psi_{a,1},\psi_{0,0},\psi_{0,1}$ are referred to as interior solutions, while $\phi_{a,1},\phi_{0,1}$ are referred to as boundary layer solutions -- see, for instance, \cite{BLP79, BSS84}. The zeroth order boundary layers $\phi_{a,0},\phi_{0,0}$ vanish due to the isotropic data $f$. The zeroth order interior solutions $\psi_{a,0},\psi_{0,0}$ coincide with the isotropic solutions $\rho_{a,0},\rho_{0,0}$ to some elliptic BVP. In Sec.~\ref{sub:diffusion}, we will define precisely the first terms in the diffusion approximation, show rigorously how to estimate them and how to deal with the remainder.

Although the construction of interior and boundary layers solutions outlined above is the classical approach to diffusion approximation expansions, the estimates concerning the higher order boundary layers require some care, as shown, for instance, in \cite{WG14}. In order to avoid such technical issues, we adopt a different approach, already introduced in \cite{DV25}: we write $u=u_a-u_0$ as
\begin{align}
\label{eq:expansion_intro}
    u = \psi_0 + \psi_1 K_n + K_n^2 R,
\end{align}
where $\psi_0=\psi_{a,0}-\psi_{0,0}$, $\psi_1=\psi_{a,1}-\psi_{1,0}$, and estimate $R$ by studying the BVP it satisfies and leveraging estimates provided by the maximum principle for radiative transfer applied to specifically designed supersolutions. Using this strategy, we provide the following a priori estimates for the remainder $R$ from \eqref{eq:expansion_intro} and for the analogous error $R_a$ in the diffusion approximation for $u_a$.

\begin{thm}
\label{thm:elena_juan_estimates}
Let $K\Subset D$, let $M>0$ and let $\sigma_a\in \cK$, where $\cK$ is defined as in \eqref{eq:definition_K}. Then there exists $C>0$ depending on $D,\ K,\ M$ and $s_0=3+\lfloor \frac{d}{2}\rfloor,\ s_1=\frac{9}{2}+\lfloor\frac{d}{2}\rfloor$ such that
\begin{equation}
\label{eq:elenajuan_first_estimate}
    \|\langle R_a \rangle\|_{L^{\infty}(D)} \leq C( K_n^{-2}\|\rho_{0,0}\|_{H^{s_0}(K)}+ K_n^{-1}\|f\|_{H^{s_1}(\partial D)} ),
\end{equation}
and 
\begin{equation}
\label{eq:elenajuan_second_estimate}
    \|\langle vR \rangle\|_{L^{\infty}(D)} \leq C( K_n^{-1}\|\rho_{0,0}\|_{H^{s_0}(K)}+ \|f\|_{H^{s_1}(\partial D)} ),
\end{equation}
where $R$, $R_a$ and $\rho_{0,0}$ are defined in Sec.~\ref{sub:diffusion} by \eqref{eq:R}, \eqref{eq:Ra} and \eqref{eq:leading order rho}, respectively.
\end{thm}

 Both estimates are central in the estimates for our inverse problem. 
In particular, the proof of Theorem~\ref{thm:elena_juan_estimates} will constitute the main technical part of the estimates on the diffusion approximation, and will be provided in Sec.~\ref{sub:diffusion}.

\subsection{Comparison with the literature}
\label{sub:literature}

The study of the inverse problem for the radiative transfer equation is a classical theme with a large literature on it, see for instance \cite{B09} for an overview. Uniqueness questions were addressed in the seminal works \cite{CS96,CS99,SU03}. Building on these ideas, \cite{W99} and \cite{BJ08} investigated the associated stability properties for fixed Knudsen numbers. Relying on microlocal tools, a stability result in the critical regime of vanishing Knudsen numbers was derived in \cite{LLU19}. Here the authors deduced a transition between a Hölder (for fixed Knudsen numbers) and a logarithmic (for increasingly small Knudsen numbers) stability regime. 

A first indication of the optimality of this transition between a logarithmic and a Hölder regime was provided in \cite{CLW18} in the linearized setting and confirmed in \cite{ZZ19} in the full nonlinear regime (however with slightly different measurement operators than in \cite{LLU19}). A key observation in these instability results is the proximity (in form of a diffusion approximation) of the nonlinear equation to the Calder\'on problem in the setting of small Knudsen numbers.
Our work builds on these seminal contributions. It provides a robust framework for extending the available instability results to rather general geometries which do no longer require strong symmetry assumptions (such as spherical symmetry). Indeed, we work in an abstract framework similar as in \cite{KRS21} adapted to our critical transition setting. This allows us to avoid the explicit computation of singular value bases but instead only requires good PDE bounds for the forward operator and for a suitable ``comparison map''. We expect that this will also allow us to study variations and extensions of the above setting for this and related problems. Moreover, a key contribution of our work is also to provide precise quantitative estimates for the boundary behaviour of the remainder (see Theorem~\ref{thm:elena_juan_estimates} and Sec.~\ref{sub:diffusion}), which is fundamental to understand the contribution of first order terms in the diffusion expansion.

The diffusion approximation of the radiative transfer equation has been extensively studied both in the absence and in the presence of absorption, in the time-dependent case as well as in the stationary one. We refer to \cite{BLP79,GW17,W21} for the diffusion approximation of the radiative transfer equation in the abscence of absorption-emission processes. The one-dimensional boundary layer equation has been studied for instance in \cite{BSS84} in the absence of absorption, in \cite{G87} in the absence of scattering processes and in \cite{S87} when both radiation processes are present. Finally, in \cite{DV25} the diffusion approximation of the radiative transfer equation for the absorption-emission processes only has been studied through its formulation as a nonlocal integral equation via maximum principle arguments. This is the method that we will use also in this article. For further references concerning the study of the radiative transfer equation we refer to the PhD thesis \cite{D26}, which contains an exhaustive summary of the available literature.

\subsection{Outline of the article}
\label{sub:outline_paper}
The remainder of the article is structured as follows. In Sec.~\ref{sec:preliminaries}, we introduce appropriate function spaces and recall some results about solutions for radiative transfer equations, together with the definition and some properties of the albedo operator in Subsec.~\ref{sub:albedo}. In Sec.~\ref{sub:diffusion}, we provide a proof of the fine estimates for the diffusion approximation and for the remainder (see Theorem~\ref{thm:elena_juan_estimates}), and in Subsec.~\ref{sub:albedo_apriori_estimates} we show how these imply the required estimates for the albedo operator. In Sec.~\ref{sec:instability}, we prove our main result (Theorem~\ref{thm:main}) using the abstract framework from \cite{KRS21} and the estimates provided later in the paper.

\section{Preliminaries}
\label{sec:preliminaries}
In this section, we introduce appropriate function spaces and recall some results about solvability of radiative transfer equations.

\subsection{Notation and function spaces}
\label{sub:function}
Let $D\subset\R^{d}$ be a strictly convex, bounded domain with $\partial D\in C^{5+\lfloor \frac{d}{2}\rfloor}$,\footnote{This is the regularity of the boundary needed to derive the diffusion approximation estimates of Theorem~\ref{thm:elena_juan_estimates}; see Sec.~\ref{sub:diffusion}.} let $\Omega=S^{d-1}$ and let $\mu$ denote the normalized $(d-1)$-dimensional Hausdorff measure on $S^{d-1}$. In the sequel, we will also consider a fixed compact set $K\Subset D$ and $M>0$. We also introduce the spaces
\begin{align*}
    L^{\infty}_+(D) &=
    \{\sigma_a\in L^{\infty}(D)\colon\ 
    \sigma_a\geq 0\}
\end{align*}
and the subset
\begin{equation}
\label{eq:definition_K}
    \cK=\{\sigma_a\in L^{\infty}_+(D)\colon\ \supp(\sigma_a)\subset K,\ \|\sigma_a\|_{C^{2,1/2}(D)}\leq M\}.
\end{equation}
Notice that, by Sobolev embedding, the set $\cS$ defined in \eqref{eq:definition_S} is contained in $\cK$. We will often use the notation $\|\sigma_{a}\|_{\cK}$ to denote the $C^{2,1/2}$ norm of an element $\sigma_a\in \cK$. Given a function $u=u(x,v)$, we will use the notation $\langle u\rangle$ to denote the average in the velocity variable with respect to $\mu$. Unless otherwise noted, we use $\nabla$ to denote the gradient with respect to the space variables $x$.

We define the \textit{anisotropic Sobolev space}
\begin{align*}
    W^2(D\times\Omega) =
    \{u\in\mathcal{D}'(D\times\Omega)\colon\, u\in L^2(D\times\Omega),\ v\cdot\nabla{u}\in L^2(D\times \Omega)\}.
\end{align*}
This is a Hilbert space endowed with the inner product
\begin{align*}
    (u_1,u_2)_{W^2} =
    (u_1,u_2)_{L^2} + (v\cdot\nabla{u_1},v\cdot\nabla{u_2})_{L^2}.
\end{align*}
We define
\begin{align*}
    &\Gamma = \partial D\times \Omega,\\
    &\Gamma_{+} = \{(x,v)\in\Gamma\colon\,
    v\cdot n_x>0\},\\
    &\Gamma_{-} = \{(x,v)\in\Gamma\colon\,
    v\cdot n_x<0\},
\end{align*}
and endow them with the measure
\begin{align*}
    d{\xi}=|v\cdot n_x|d{x}d\mu(v),
\end{align*}
where $d{x}$ is the $(d-1)$-dimensional Hausdorff measure on $\partial D$.
\begin{thm}[Traces in $W^2$]
There exist unique bounded operators
\begin{align*}
    T_{\pm}\colon W^2(D\times\Omega)\rightarrow L^2_{\mathrm{loc}}(\Gamma_{\pm},d\xi)
\end{align*}
that extend the operators $T_{\pm}u=u|_{\Gamma_{\pm}}$ defined on the dense subspace $C^{\infty}(\overline{D}\times\Omega)\subset W^2(D\times\Omega)$.
\end{thm}
\begin{proof}
See \cite[Ch.~XXI, Sec.~2.2, Thm.~1]{dautraylions}.
\end{proof}
\noindent
We now define the following space
\begin{align*}
    W^2_*(D\times\Omega) &=
    \{u\in W^2(D\times\Omega)\colon\, u|_{\Gamma_-}\in L^2(\Gamma_{-},d\xi)\}.
\end{align*}
This is a Hilbert space endowed with the inner product
\begin{align*}
    (u_1,u_2)_{W^2_*} = (u_1,u_2)_{W^2} + (u_1|_{\Gamma_-},u_2|_{\Gamma_-})_{L^2}.
\end{align*}
\begin{thm}[Traces in $W_*^2$]
The image of $W^2_*(D\times\Omega)$ via the trace operator $T_+$ is contained in $L^2(\Gamma_{+},d\xi)$. Moreover, the operator
\begin{align*}
    T_+\colon W_*^2(D\times\Omega)\rightarrow L^2(\Gamma_{+},d\xi)
\end{align*}
is bounded.
\end{thm}
\begin{proof}
See \cite{bardos70,cessenat}.
\end{proof}

\begin{thm}[Poincaré-Friedrichs inequality]
\label{thm:poincare}
There exists a constant $C=C(D)>0$ such that the following holds for $u\in W^2_*(D\times\Omega)$:
\begin{align*}
    \|u\|_{L^2(D\times\Omega)} \leq C\big(\|v\cdot\nabla{u}\|_{L^2(D\times\Omega)}+
    \|u|_{\Gamma_-}\|_{L^2(\Gamma_{-},
    d\xi)}\big)
\end{align*}
In particular, the right hand side defines an equivalent norm on $W^2_*(D\times\Omega)$.
\end{thm}
\begin{proof}
See \cite[Lemma~3.1]{MRS00}.
\end{proof}

\subsection{Radiative transfer and the albedo operator}
\label{sub:bvp}
In this section, we recall some results about BVP for radiative transfer, and define the albedo operator.
\begin{thm}
\label{thm:radiative_transfer_stability}
Given $\sigma_a\in L^{\infty}_+(D)$ with $\|\sigma_a\|_{L^{\infty}(D)}\leq M$, $F\in L^2(D\times\Omega)$ and $G\in L^2(\Gamma_-,d{\xi})$, there exists a unique (strong) solution $u\in W^2_*(D\times\Omega)$ of
\begin{equation}
\label{eq:ua_def}
\begin{cases}
    v\cdot\nabla{u} + K_n\sigma_a u + \frac{1}{K_n}(u-\langle u\rangle) = F & \text{in $D\times\Omega$},\\
    u|_{\Gamma_-} = G.
\end{cases}
\end{equation}
Moreover, the following estimate holds for some $C=C(D,M)>0$ and for $K_n\in(0,1]$:
\begin{align*}
    \|u\|_{W_*^2(D\times\Omega)} \leq C(
    K_n^{-1}\|F\|_{L^2(D\times\Omega)} + K_n^{-1/2}\|G\|_{L^2(\Gamma_-,d{\xi})}
    ).
\end{align*}
\end{thm}
\begin{proof}
See \cite[Ch.~XXI, Sec.~4.1, Thm.~4]{dautraylions}, together with Remark 15 and Appendix 2 to Ch.~XXI in the same reference, for a proof of existence and uniqueness. For the stability estimates, we first multiply \eqref{eq:ua_def} by $u$ and integrate over $D\times\Omega$ to obtain
\begin{align*}
    \frac{1}{2}\int_{D\times\Omega} v\cdot\nabla(u^2) + K_n\int_{D\times\Omega} \sigma_a u^2 + \frac{1}{K_n}\int_{D\times\Omega} (u-\langle u\rangle)u &= \int_{D\times \Omega} Fu,
\end{align*}
and therefore
\begin{align*}
    \frac{1}{2}\|u\|_{L^2(\Gamma_+,\diff{\xi})}^2-
    \frac{1}{2}\|u\|_{L^2(\Gamma_-,\diff{\xi})}^2 + K_n\|\sqrt{\sigma_a}u\|_{L^2(D\times\Omega)}^2 +
    \frac{1}{K_n} \|u-\langle u\rangle\|_{L^2(D\times\Omega)}^2 &= (F,u)_{L^2(D\times\Omega)},
\end{align*}
from which we deduce that
\begin{align*}
    K_n\|\sqrt{\sigma_a}u\|_{L^2(D\times\Omega)}^2 + \frac{1}{K_n} \|u-\langle u\rangle\|_{L^2(D\times\Omega)}^2 \leq |(F,u)_{L^2(D\times\Omega)}| + \frac{1}{2}\|G\|_{L^2(\Gamma_-,\diff{\xi})}^2.
\end{align*}
We now use the equation for $u$ to estimate $\|v\cdot \nabla{u}\|_{L^2(D\times\Omega)}$:
\begin{align*}
    \|v\cdot \nabla{u}\|_{L^2(D\times\Omega)}^2 &\lesssim
    K_n^2\|\sigma_a u\|_{L^2(D\times\Omega)}^2 + \frac{1}{K_n^2}\|u-\langle u\rangle\|_{L^2(D\times\Omega)}^2 + \|F\|_{L^2(D\times\Omega)}^2 \\ &\leq
    \frac{\max(M,1)}{K_n}\lc K_n \|\sqrt{\sigma_a}u\|_{L^2(D\times\Omega)}^2 +
    \frac{1}{K_n}\|u-\langle u\rangle\|_{L^2(D\times\Omega)}^2 \rc + \|F\|_{L^2(D\times\Omega)}^2\\ 
    &\lesssim
    \frac{C_{\varepsilon}}{K_n^2} \|F\|_{L^2(D\times\Omega)}^2 + \varepsilon \|u\|_{L^2(D\times\Omega)}^2+ \frac{1}{K_n}\|G\|_{L^2(\Gamma_-,\diff{\xi})}^2.
\end{align*}
By Poincaré's inequality (Theorem~\ref{thm:poincare}), we conclude that
\begin{align*}
    \|u\|_{W^2_*(D\times\Omega)} &\lesssim
    \|v\cdot\nabla{u}\|_{L^2(D\times\Omega)} + \|G\|_{L^2(\Gamma_-,\diff{\xi})} \\ &\lesssim
    K_n^{-1} \|F\|_{L^2(D\times\Omega)} +
    \varepsilon \|u\|_{L^2(D\times\Omega)} +
    K_n^{-1/2}\|G\|_{L^2(\Gamma_-,\diff{\xi})}.
\end{align*}
Choosing $\varepsilon$ sufficiently small implies the claim.
\end{proof}

\subsection{Properties of the albedo operator}
\label{sub:albedo}

In this section, we deduce basic properties of the albedo operator. We begin by providing the strong form of the definition of the albedo operator which we will use in what follows below.
\begin{defi}
Let $s\geq 0$ be fixed. Given $\sigma_a\in L^{\infty}_+(D)$, we define the \textit{albedo operator} $\Lambda_{\sigma_a}\in\mathcal{B}(H^s(\partial D),H^{-s}(\partial D))$ as follows: For $f\in H^s(\partial D)$
\begin{equation}
\label{eq:albedo_definition}
    \Lambda_{\sigma_a}f = \frac{1}{K_n}
    \int_{\Omega} (v\cdot n_x) u_a |_{\Gamma}\,d{\mu}(v) = \frac{1}{K_n} \langle (v\cdot n_x) u_a|_{\Gamma} \rangle,
\end{equation}
where $u_a\in W^2_*(D\times\Omega)$ solves \eqref{eq:ua_def} with $F=0$ and $G=f$. We define the forward map
\begin{align*}
    \Lambda\colon L^{\infty}_+(D)\rightarrow
    \mathcal{B}(H^s(\partial D),H^{-s}(\partial D))
\end{align*}
as $\Lambda(\sigma_a)=\Lambda_{\sigma_a}$.
\end{defi}
We next provide a weak form of the albedo operator.
\begin{lem}
\label{lem:albedo_formula}
Let $\sigma_a\in L^{\infty}_+(D)$, let $f\in L^2(\partial D)$ and let $u_a\in W^2_*(D\times\Omega)$ be the solution of \eqref{eq:ua_def} with $F=0$ and $G=f$. Let $g\in H^{1/2}(\partial D)$ and let $\tilde{g}\in H^1(D)$ be such that $\tilde{g}|_{\partial D}=g$. Then the following holds:
\begin{equation}
    \langle \Lambda_{\sigma_a}f,g \rangle = -\int_{D} \sigma_a \langle u_a \rangle \tilde{g} + \frac{1}{K_n} \int_{D} \langle v u_a\rangle\cdot\nabla{\tilde{g}}.
\end{equation}
\end{lem}

\begin{proof}
Using the definition of $\Lambda_{\sigma_a}$ and the divergence theorem, we have that
\begin{align*}
    \langle \Lambda_{\sigma_a}f,g \rangle &= 
    \frac{1}{K_n} \int_{\partial D} \langle (v\cdot n_x)u_a \rangle g \\ &=
    \frac{1}{K_n} \int_{\partial D}
    \langle v u_a \rangle g \cdot n_x \\ &=
    \frac{1}{K_n} \int_{D} \langle v\cdot \nabla{u_a} \rangle\tilde{g} + 
    \frac{1}{K_n} \int_{D} \langle vu_a \rangle\cdot\nabla\tilde{g}.
\end{align*}
Using the equation satisfied by $u_a$ in the first term proves the lemma.
\end{proof}

For later use in the derivation of our instability estimates, we finally deduce the precise form of the Banach space adjoint operator $\Lambda_{\sigma_a}'$ for the albedo operator. Here we make use of the isotropic boundary data under consideration.

\begin{prop}
\label{prop:dual_operator}
Let $f,g\in H^{s}(\partial D)$. Let $u_a\in W_*^2(D\times\Omega)$ be the solution of \eqref{eq:ua_def} with $F=0$ and $G=f$, and let $w_a\in W_*^2(D\times\Omega)$ be the solution of
\begin{equation}
\label{eq:backward_rte}
\begin{cases}
    -v\cdot \nabla{w} + K_n \sigma_a w + \frac{1}{K_n}(w-\langle w\rangle)=0 & \text{in $D\times\Omega$}\\
    w|_{\Gamma_+}=g.
\end{cases}
\end{equation}
Then the following holds:
\begin{align*}
    \langle \Lambda_{\sigma_a}f,g \rangle = -\frac{1}{K_n}\int_{\partial D}
    f\langle (v\cdot n_x)w_a \rangle.
\end{align*}
In particular, we have that the adjoint operator $ \Lambda_{\sigma_a}'$ is given by
\begin{align*}
    \Lambda_{\sigma_a}'g = -\frac{1}{K_n}\int_{\Omega} (v\cdot n_x)w_a|_{\Gamma}.
\end{align*}
\end{prop}
\begin{proof}
Integrating \eqref{eq:ua_def} (with $F=0$) against $w_a$ on $D\times\Omega$, we get
\begin{align*}
    \int_{D\times\Omega} (v\cdot\nabla{u_a})w_a + K_n\int_{D\times\Omega} u_a w_a +
    \frac{1}{K_n} \int_{D\times\Omega} (u_a-\langle u_a\rangle)w_a = 0.
\end{align*}
The first term can be written as
\begin{align*}
    \int_{D\times\Omega} (v\cdot\nabla{u_a})w_a = \int_{\Gamma} (v\cdot n_x)u_a w_a - \int_{D\times\Omega} u_a(v\cdot \nabla w_a).
\end{align*}
The last term can be written as
\begin{align*}
    \frac{1}{K_n}\int_{D\times\Omega} (u_a-\langle u_a\rangle)w_a =
    \frac{1}{K_n}\int_{D\times\Omega} u_a(w_a-\langle w_a\rangle).
\end{align*}
We conclude that
\begin{align*}
    0 &= \int_{\Gamma} (v\cdot n_x)u_a w_a - \int_{D\times\Omega} u_a(v\cdot \nabla w_a) + K_n\int_{D\times\Omega} u_a w_a +
    \frac{1}{K_n}\int_{D\times\Omega} u_a(w_a-\langle w_a\rangle) \\ &= \int_{\Gamma} (v\cdot n_x)u_a w_a,
\end{align*}
where we have used the equation for $w_a$ in the last equality. Therefore, we have that
\begin{equation}
\label{eq:dual1}
    \int_{\Gamma_+} (v\cdot n_x)u_a g = \int_{\Gamma_+} (v\cdot n_x)u_a w_a = -\int_{\Gamma_-}(v\cdot n_x)u_a w_a =
    -\int_{\Gamma_-} (v\cdot n_x)f w_a.
\end{equation}
Using the fact that $fg$ is isotropic, we infer that
\begin{align*}
    \int_{\Gamma} (v\cdot n_x)fg = 0,
\end{align*}
and therefore
\begin{equation}
\label{eq:dual2}
    \int_{\Gamma_+} (v\cdot n_x)fg = -\int_{\Gamma_-} (v\cdot n_x)fg.
\end{equation}
Adding \eqref{eq:dual2} to \eqref{eq:dual1}, we deduce that
\begin{align*}
    \int_{\Gamma} (v\cdot n_x)u_a g = -\int_{\Gamma} (v\cdot n_x)f w_a,
\end{align*}
and we conclude.
\end{proof}

\section{Diffusion approximation}
\label{sub:diffusion}
In this section, we aim to prove Theorem \ref{thm:elena_juan_estimates}. To this end, we will establish several estimates which can be obtained using the diffusion approximation. First of all we define the terms involved in the diffusion approximation. 
As we anticipated in the introduction, we consider the expansions
\begin{align*}
    u_0&=\sum_{j=0}^\infty K_n^j \psi_{0,j}(x,v)=\psi_{0,0}(x,v)+K_n\psi_{0,1}(x,v)+K_n^2R_0(x,v);\\
    u_a&=\sum_{j=0}^\infty K_n^j \psi_{a,j}(x,v)=\psi_{a,0}(x,v)+K_n\psi_{a,1}(x,v)+K_n^2R_a(x,v);\\
    u&=u_a-u_0=\sum_{j=0}^\infty K_n^j \psi_{j}(x,v)=\psi_{0}(x,v)+K_n\psi_{1}(x,v)+K_n^2R(x,v).\\
\end{align*}
We plug now these expressions into \eqref{eq:rad_trans} and we compare the terms of the same order of magnitude. Therefore, we see that the zeroth order solutions $\psi_{a,0}$ and $\psi_{0,0}$ are isotropic and satisfy $\psi_{a,0}=\rho_{a,0}$ and $\psi_{0,0}=\rho_{0,0}$, where $\rho_{a,0},\rho_{0,0}\in H^{s_0}(D)$ for $s_0=3+\lfloor \frac{d}{2}\rfloor$ are the solution to the following elliptic BVP:
\begin{equation}\label{eq:leading order rho}
\begin{cases}
    -C_d\Delta{\rho_{a,0}}+\sigma_a\rho_{a,0} = 0 & \text{in $D$},\\
    \rho_{a,0}|_{\partial D} = f,
\end{cases}
\qquad
\begin{cases}
    -C_d\Delta{\rho_{0,0}} = 0 & \text{in $D$},\\
    \rho_{0,0}|_{\partial D} = f,
\end{cases}
\end{equation}
where $C_d>0$ is such that $\langle v\otimes v\rangle=C_d I$ and where $f\in H^{s_1}(\partial D)$ for $s_1=\frac{9}{2}+\lfloor \frac{d}{2}\rfloor$. Notice, in particular, that $\psi_0=\psi_{a,0}-\psi_{0,0}=\rho_{a,0}-\rho_{0,0}$ satisfies
\begin{equation}\label{eq: leading order}
\begin{cases}
    -C_d\Delta{\psi_0}+\sigma_a\psi_0 = -\sigma_a\rho_{0,0} & \text{in $D$},\\
    \psi_0|_{\partial D} = 0.
\end{cases}
\end{equation}
We remark that, according to the method of matched asymptotic expansions, the boundary conditions of the problems \eqref{eq:leading order rho} and \eqref{eq: leading order} are determined by the limit as $y\to\infty$ of the solutions to the boundary layer equations associated to the problems solved by the functions $u_0$, $u_a$ and $u$, respectively. Indeed, the one-dimensional boundary layer equations are given at any point $p\in\partial D$ by   
\begin{equation}\label{eq:b.l. leading order}
\begin{cases}
    -(v\cdot n_p) \partial_y \phi(y,v;p)=\langle \phi\rangle (y;p)-\phi(y,v;p) & y>0,\; v\in S^{d-1},\\
    \phi_{a,0}(0,v;p) = g(p)& v\cdot n_p<0,
\end{cases}
\end{equation}
where
\[g(p)=\begin{cases}
    f(p)&\text{  for the problem associated to  }u,\ u_a,\\0&\text{ for the problem associated to  }u.
\end{cases}\]
Since $g$ is isotropic, Theorem 4 in \cite{BSS84} implies that the unique bounded solution to \eqref{eq:b.l. leading order} is given exactly by $\phi=g$.

The inverse of the operator
\[\phi(x,v)\mapsto \phi(x,v)-\langle \phi\rangle(x)\] is unique up to isotropic functions, therefore the first order terms are given by
\begin{align*}
     \psi_{0,1}(x,v)=&-v\cdot \nabla \rho_{0,0}(x)+c_0(x),\\\psi_{a,1}(x,v)=&-v\cdot \nabla \rho_{a,0}(x)+c_a(x),\\\psi_{1}(x,v)=&-v\cdot \nabla \psi_0(x)+c(x),
\end{align*}
where $c_0$, $c_a$ and $c$ depend only on $x\in D$.  

First of all, we see that comparing the terms in the radiative transfer equation of order $K_n^2$ these functions solve in the interior
\begin{align*}
     -C_d \Delta c_0(x)=0, \quad -C_d \Delta c_a(x)=-\sigma_a c_a(x) \;\text{ and }\;  -C_d \Delta c(x)=-\sigma_a (c_0(x)+c(x)). 
\end{align*} 
The boundary conditions will be specified later and will be chosen in a specific way.
Specifically, we will choose the functions $c_0,\ c_a$ and $c$ so that the contribution due to the boundary conditions of the remainders $R_0,\ R_a$ and $R$ vanish in the interior of the domain. This is a key step for the proof of the estimates of Theorem \ref{thm:elena_juan_estimates}.

With the definition of the zeroth and first order terms we see that the remainders $R_0$, $R_a$ and $R$ solve (exactly) the following boundary value problems. The function $R_0$ is a solution to
\begin{equation}\label{eq:R0}
    \begin{cases}
        v\cdot \nabla R_0=\frac{1}{K_n}\left(\langle R_0\rangle-R_0\right)+\frac{1}{K_n}\left(\Div \left(v\otimes v \nabla \rho_{0,0}\right)-v\cdot \nabla c_0\right) & \text{in }D\times\Omega\\
        \left.R_0\right|_{\Gamma_-}=\frac{1}{K_n}\left.\left(v\cdot \nabla\rho_{0,0}-c_0\right)\right|_{\Gamma_-}.
    \end{cases}
\end{equation}
Notice that by construction $\langle\Div \left(v\otimes v \nabla \rho_{0,0}\right)-v\cdot \nabla c_0\rangle=0 $.

Similarly,
\begin{equation}\label{eq:Ra}
    \begin{cases}
        v\cdot \nabla R_a=\frac{1}{K_n}\left(\langle R_a\rangle-R_a\right)-K_n\sigma_a R_a\\\qquad\qquad\qquad+\frac{1}{K_n}\left(\Div \left(v\otimes v \nabla \rho_{a,0}\right)-v\cdot \nabla c_a-\sigma_a\rho_{a,0}\right) & \text{in }D\times\Omega\\
        \qquad\qquad\qquad+\sigma_a\left(v\cdot \nabla\rho_{a,0}-c_a\right)\\
        \left.R_a\right|_{\Gamma_-}=\frac{1}{K_n}\left.\left(v\cdot \nabla\rho_{a,0}-c_a\right)\right|_{\Gamma_-}
        
    \end{cases}
\end{equation}
and
\begin{equation}\label{eq:R}
    \begin{cases}
        v\cdot \nabla R=\frac{1}{K_n}\left(\langle R\rangle-R\right)-K_n\sigma_a R\\\qquad\qquad\qquad+\frac{1}{K_n}\left(\Div \left(v\otimes v \nabla \psi_0\right)-v\cdot \nabla c-\sigma_a\left(\psi_0+\rho_{0,0}\right)\right)\\
        \qquad\qquad\qquad+\sigma_a\left(v\cdot \nabla \left(\psi_0+\rho_{0,0}\right)-\left(c+c_0\right)\right)-K_n\sigma_a R_0 & \text{in }D\times\Omega\\
        \left.R\right|_{\Gamma_-}=\frac{1}{K_n}\left.\left(v\cdot \nabla\psi_0-c\right)\right|_{\Gamma_-}.
    \end{cases}
\end{equation}
According to Theorem \ref{thm:elena_juan_estimates} we aim to prove 
\[\left|\langle R_a\rangle(x)\right|\leq C(D,\sigma_a,K)\frac{\Arrowvert \rho_{0,0}\Arrowvert_{H^{s_0}(D)}}{K_n^2}+C(D,\sigma_a,K)\frac{\Arrowvert f\Arrowvert_{H^{s_1}(\partial D)}}{K_n}\]
as well as
\[\left|\langle v R\rangle(x)\right|\leq C(D,\sigma_a,K)\frac{\Arrowvert \rho_{0,0}\Arrowvert_{H^{s_0}(D)}}{K_n}+C(D,\sigma_a,K)\Arrowvert f\Arrowvert_{H^{s_1}(\partial D)}.\]
Notice that, since $R_a=R+R_0$, it is enough to consider only $R_0$ and $R$. 

Before proving rigorously these estimates, we first give some formal justification, why we expect such estimates to hold and how we expect the remainder to behave.
\subsection{Formal discussion: The diffusion approximation}\label{sub:diff}
As $K_n\to 0$ we expect by the diffusion approximation that the leading order of $R_0$ and of $R$ solve in the bulk the following elliptic problems: For the leading order $R_{0,0}$ of $R_0$ we expect that it is a solution to

\begin{equation}\label{eq:bulk R00}
    -C_d \Delta \langle R_{0,0}\rangle=\langle \Div\left(v\otimes v\nabla \Div\left(v\otimes v \nabla \rho_{0,0}\right)\right)\rangle=0
\end{equation}
where we used that $\langle \Div\left(v\otimes v\nabla \Div\left(v\otimes v \nabla g(x)\right)\right)\rangle=\overline{C}_d \Delta^2 g(x)$ and the properties of $\rho_{0,0}$.

Let $\tilde{R}_0$ denote the leading order of the remainder $R$. We expect it to approximate in the bulk the solution of 
\begin{equation}\label{eq:bulk tildeR0}
    -C_d \Delta \langle \tilde{R}_{0}\rangle=-\sigma_a \langle \tilde{R}_{0}\rangle-\sigma_a\langle R_{0,0}\rangle+\tilde{C}_d\Delta\left(\sigma_a\left(\psi_0+\rho_{0,0,}\right)\right)-C_d\Div\left(\sigma_a\nabla\left(\psi_0+\rho_{0,0}\right)\right),
\end{equation}
where $\tilde{C}_d=\frac{\overline{C}_d}{C_d}-C_d$.

At this point it is worth to explain how the elliptic equations \eqref{eq:bulk R00} and \eqref{eq:bulk tildeR0} are derived. Let us start considering $R_0=\sum_{j\geq0} K_n^j R_{0,j}$. Plugging this expansion into \eqref{eq:R0} and studying all terms of the same order of magnitude we see that
\[R_{0,0}= \langle R_{0,0}\rangle+\Div\left(v\otimes v\nabla\rho_{0,0}\right)-v\cdot \nabla c_0,\]
which is solvable since $\langle\Div\left(v\otimes v\nabla\rho_{0,0}\right)-v\cdot \nabla c_0\rangle=0$. The first order terms give
\[v\cdot\nabla R_{0,0}=\langle R_{0,1}\rangle-R_{0,1}\]
which can be also solved since $\langle v\cdot \nabla R_{0,0}\rangle=-C_d\Delta c_0=0$. Thus,
\[R_{0,1}= \langle R_{0,1}\rangle -v\cdot \nabla\langle R_{0,0}\rangle -v\cdot \nabla \Div\left(v\otimes v\nabla\rho_{0,0}\right)- \Div\left(v\otimes v\nabla c_0\right).\]
Finally, integrating the next order terms $v\cdot \nabla R_{0,1}=\langle R_{0,2}\rangle-R_{0,2}$ over $S^{d-1}$ we obtain \eqref{eq:bulk R00}.

Moreover, near the boundary $\partial D$ the behavior of $R_0$ is given at the leading order by the solution of the following (1-dimensional) boundary layer equation defined for any $p\in\partial D$
\begin{equation}\label{eq:b.l.R0}
    \begin{cases}
        -(v\cdot n_p)\partial_y \overline{R}_0(y,v;p)=\langle \overline{R}_0\rangle-\overline{R}_0+\Div\left.\left(v\otimes v\nabla\rho_{0,0}\right)\right|_{x=p}-v\cdot \left.\nabla c_0\right|_{x=p} &y>0,\;v\in S^{d-1},\\
        \overline{R}_0(0,v;p)=\frac{1}{K_n}\left.\left(v\cdot \nabla\rho_{0,0}-c_0\right)\right|_{x=p}&v\cdot n_p<0.
    \end{cases}
\end{equation}
Thus, choosing $\left.c_0\right|_{\partial D}$ to be such that $\lim\limits_{y\to\infty} \overline{R}_0(y,v;p)$ is bounded independently of $K_n$ we can expect $\langle R_0\rangle$, and thus also $R_0$, to be bounded in the bulk. Moreover, Theorem 4 in \cite{BSS84} and the linearity of the first equation in \eqref{eq:b.l.R0} will imply, as we will see later in a more rigorous way, that
\[\left|\langle \overline{R}_0\rangle (y,p)\right|\leq \frac{C\left(\Arrowvert f\Arrowvert\right)}{K_n}e^{-Ay}+C\left(\Arrowvert f\Arrowvert,\Arrowvert \rho_{0,0}\Arrowvert\right),\]
so that we expect
\begin{equation}\label{eq:desired est.}
    \left|\langle R_{0,0}\rangle (x)\right|\leq \frac{C\left(\Arrowvert f\Arrowvert\right)}{K_n}e^{-\frac{Ad(x)}{K_n}}+C\left(\Arrowvert f\Arrowvert,\Arrowvert \rho_{0,0}\Arrowvert\right),
\end{equation}
where $d(x)=\text{dist}(x,\partial D)$ and $\|\cdot\|$ denotes some norms which are to be specified in what follows below. Notice that approximating the leading order of $R_0$ using matched asymptotic expansions we can see that estimate \eqref{eq:desired est.} cannot be improved.

In a similar way we can argue for $R=\sum\limits_{j\geq 0} K_n^j\tilde{R}_{j}$ solving \eqref{eq:R}. The leading order satisfies 
\[\tilde{R}_0=\langle \tilde{R}_0\rangle+f_1,\]
where $f_1=\Div \left(v\otimes v \nabla \psi_0\right)-v\cdot \nabla c-\sigma_a\left(\psi_0+\rho_{0,0}\right)$ with $\langle f_1\rangle=0$. Since $\langle v\cdot\nabla f_1\rangle=-C_d\Delta c$
and $\langle f_2\rangle=\sigma_a(c+c_0)$, where $f_2=-\sigma_a\left(v\cdot \nabla \left(\psi_0+\rho_{0,0}\right)-\left(c+c_0\right)\right)$, we see that
\[\tilde{R}_1=\langle \tilde{R}_1\rangle+f_2-v\cdot \nabla\langle\tilde{R}_0\rangle-v\cdot\nabla f_1\]
is solvable.
Finally, we conclude that 
\[v\cdot \nabla\tilde{R}_1=\langle \tilde{R}_2\rangle-\tilde{R}_2-\sigma_a\tilde{R}_0-\sigma_a R_{0,0}.\]
Thus, integrating over $S^{d-1}$ we obtain \eqref{eq:bulk tildeR0}. Notice that we used also
\begin{align*}
    \langle \Div\left(v\otimes v\nabla f_1\right)\rangle=\tilde{C}_d\Delta \left(\sigma_a(\psi_0+\rho_{0,0})\right) \;\text{ and }\;\langle v\cdot\nabla f_2\rangle=C_d\Div\left(\sigma_a\nabla\left(\psi_0+\rho_{0,0}\right)\right).
\end{align*}
Near the boundary $\partial D$ the remainder $R$ behaves according to the boundary layer equation
\begin{equation}\label{eq:b.l.R}
    \begin{cases}
        -(v\cdot n_p)\partial_y \overline{R}(y,v;p)=\langle \overline{R}\rangle-\overline{R}+f_1(v,p) &y>0,\;v\in S^{d-1},\\
        \overline{R}(0,v;p)=\frac{1}{K_n}\left.\left(v\cdot \nabla\psi_0-c\right)\right|_{x=p}&v\cdot n_p<0.
    \end{cases}
\end{equation}
Once more we see that choosing in a proper way $\left.c\right|_{\partial D}$
\[\left|\langle \overline{R}\rangle\right|\leq \frac{C(\Arrowvert\psi_0\Arrowvert)}{K_n}\text{ and }\lim\limits_{y\to\infty}\overline{R}(y,v;p) \text{ is bounded,}\]
where $\Arrowvert \cdot\Arrowvert$ denotes some suitable norm that will be specified later.
It is important also to notice that, while the bounds on the boundary layer solution $\overline{R}$ seem not to depend on $f$, the elliptic equation \eqref{eq:bulk tildeR0} solved by the leading order in the bulk yields a bound of $R$ in terms of $\sigma_a R_0$, among others. This is the reason why we will have to consider carefully the behaviour of $R_0$, especially near the boundary $\partial D$, where we expect the solution to explode as $K_n \rightarrow 0$. 
This behaviour is substantiated by (formal) matched asymptotic expansions which yield the optimality of the estimate \eqref{eq:desired est.} (see also Proposition \ref{prop:R01}).

We remark also that because of the boundary layer equations we can expect neither $R_0$ nor $R$ to be uniformly bounded in $\overline{D}$. Due to the same reason, we can also not expect their $L^p$-norms to be bounded.

\subsection{Construction of the auxiliary functions $c_0,\ c_a$ and $c$}
\label{subsub:first}
In this subsection we determine the function $c$ (respectively $c_0,\ c_a$) such that the main contribution of the remainder $R$ (respectively $R_0,\ R_a$) is concentrated in a boundary layer of thickness $K_n$ near the boundary.

As we introduced above, we set $\psi_{0,1}=-v\cdot\nabla\rho_{0,0}+c_0$ and $\psi_1=-v\cdot \nabla \psi_0+c$, where $-\Delta c_0=0$ and $-C_d\Delta c=-\sigma_a(c+c_0)$ in the bulk. In this subsection we determine the boundary conditions that we should impose in order to guarantee that the solutions to the boundary layer equations \eqref{eq:b.l.R0} and \eqref{eq:b.l.R} converge to a constant independent on $K_n$ as $y\to\infty$. To this end, because of linearity, we consider the equation
\begin{equation*}
    \begin{cases}
        -(v\cdot n_p)\partial_y u(y,v;p)=\langle u\rangle-u &y>0,\;v\in S^{d-1},\\
        u(0,v;p)=\frac{1}{K_n}\left(v\cdot \nabla\phi(p)-a(p)\right)&v\cdot n_p<0
    \end{cases}
\end{equation*}
for some $\nabla \phi\in C^0(D)$ and some suitable function $a(p)$ to be determined below. Solving the ODE we obtain 
\begin{align*}
    u(y,v;p)=&\frac{1}{K_n}\mathbb{1}_{\{v\cdot n_p<0\}}e^{-\frac{y}{|v\cdot n_p|}}\left(v\cdot\nabla \phi(p)-a(p)\right)+\mathbb{1}_{\{v\cdot n_p<0\}}\left(\int_0^y\frac{e^{-\frac{y-s}{|v\cdot n_p|}}}{|v\cdot n_p|}\langle u\rangle(s;p)ds\right)\\
    &+\mathbb{1}_{\{v\cdot n_p>0\}}\left(\int_y^\infty\frac{e^{\frac{y-s}{|v\cdot n_p|}}}{|v\cdot n_p|}\langle u\rangle(s;p)ds\right).
\end{align*}
Integrating this expression over $S^{d-1}$ and using the symmetry properties of the function $v\mapsto\mathbb{1}_{\{v\cdot n_p<0\}}e^{-\frac{y}{|v\cdot n_p|}}(v\cdot\tau) (\tau\cdot \nabla \phi(p)) $, where $\tau\perp n_p$, we conclude that $\langle u\rangle$ solves the following nonlocal equation
\begin{equation}\label{eq:nonlocal}
\begin{split}
    \langle u\rangle (y;p)-\int_0^\infty E(y-\xi)\langle u\rangle&(\xi;p) d\xi=-\frac{1}{K_n}\frac{1}{4\pi}\int_{\{v\cdot n_p<0\}} \left[(v\cdot n_p)\partial_n\phi(p)+a(p)\right]e^{-\frac{y}{|v\cdot n_p|}}dv\\ =& \frac{1}{2K_n}\left[\frac{\partial_n\phi(p)}{2}\left(e^{-y}+ye^{-y}-y^2E_1(y)\right)-a(p)\left(e^{-y}-yE_1(y)\right)\right],
\end{split}
\end{equation}
where $E(x)=\frac{E_1(x)}{2}=\frac{1}{2}\int_{|x|}^\infty \frac{e^{-t}}{t}dt$ is the normalized exponential integral. The nonlocal equation \eqref{eq:nonlocal} has been obtained using a suitable change of variables as done in \cite{DV25}. Moreover, we used also that
\[\int_0^1 e^{-\frac{y}{s}}\begin{pmatrix}
    1\\s
\end{pmatrix}ds=\begin{pmatrix}
   e^{-y}-yE_1(y)\\\frac{1}{2}\left(e^{-y}+ye^{-y}-y^2E_1(y)\right)
\end{pmatrix}. \] Let us denote by $\mathcal{L}$ the linear operator acting on $L^\infty(\R_+)$ defined as
\begin{equation}\label{eq:def.L}
    \mathcal{L}(w)(y)=w(y)-\int_0^\infty E(y-\xi)w(\xi) \ d\xi.
\end{equation}
Since the sources are positive and have exponential decay, the results of Theorem 3.2, Theorem 3.3 and Lemma 3.4 in \cite{DV25} imply the existence of unique non-negative continuous functions $w_1(y)$, $w_2(y)$ solving
\[\mathcal{L}(w_1)(y)=\frac{1}{2}\left(e^{-y}+ye^{-y}-y^2E_1(y)\right) \qquad\text{and}\qquad\mathcal{L}(w_2)(y)=e^{-y}-yE_1(y).\]
Moreover, there exist positive constants $W_1, \ W_2>0$ and $C_1,\ C_2>0$ such that
\[\left|w_1(y)-W_1\right|\leq C_1 e^{-\frac{y}{2}}\qquad\text{and}\qquad\left|w_2(y)-W_2\right|\leq C_2 e^{-\frac{y}{2}}.\]
Thus, defining for any $p\in\partial D$
\[a(p)= \frac{W_1}{W_2} \partial_n \phi(p)\]
we can conclude that $\langle u\rangle(y;p)$ converges to zero as $y\to\infty$ for every $p\in\partial D$ with exponential rate
\[\left|\langle u\rangle(y;p)\right|\leq \left(C_1+\frac{W_1}{W_2}C_2\right) \frac{|\partial_n \phi(p)|}{2K_n} e^{-\frac{y}{2}}.\]
Hence, we define $c_0$ and $c$ to be the solutions of the following Dirichlet problems
\begin{equation}\label{eq:def.c}
\begin{cases}
    -C_d\Delta c_0=0& \text{ in }D,\\
    \left.c_0\right|_{\partial D}=\frac{W_1}{W_2}\left.\partial_n \rho_{0,0}\right|_{\partial D},
\end{cases}
\qquad\text{and}\qquad\begin{cases}
    -C_d\Delta c=-\sigma_a(c+c_0)& \text{ in }D,\\
    \left.c\right|_{\partial D}=\frac{W_1}{W_2}\left.\partial_n \psi_0\right|_{\partial D}.
\end{cases}
\end{equation}In the following we will carry on the proof for Theorem \ref{thm:elena_juan_estimates} for the spacial dimension $d=3$. Later, in Subsec.~\ref{sec:higher_dim} we will generalize the result for any spacial dimension $d\geq 3$. Thus, we consider now $s_0=4$ and $s_1= 5+\frac{1}{2}$ as well as $f\in H^{s_1}(\partial D)$.

\subsection{A rough estimate for the remainders}\label{subsub:remainder}
In this subsection we compute some first (rougher) estimates for the remainders $R$ and $R_0$, which are obtained mainly using in a straightforward way the triangle inequality. However, the resulting estimate for $\langle vR\rangle$ turns out, as we will see, not to be the optimal one as stated in Theorem \ref{thm:elena_juan_estimates}. We will show the following.
\begin{lem}\label{lemma:rough estimate}
    Let $K_n\in(0,1]$. The solution $R\in L^\infty(D\times S^{2})$ of \eqref{eq:R} satisfies the following bound
    \begin{equation}
        \left|\langle R\rangle(x)\right|\leq C_1(D,\Arrowvert \sigma_a\Arrowvert_{\mathcal{K}}) \frac{\|\rho_{0,0}\|_{H^4(K)}}{K_n}+C_2\left(D,\Arrowvert \sigma_a\Arrowvert_{\mathcal{K}}\right)\| f\|_{H^{s_1}(\partial D)}+C_3(D) \Arrowvert\sigma_a R_0\Arrowvert_\infty,
    \end{equation}
     where $C_1, \ C_2, \ C_3$ are independent of $K_n$, $C_1$ is independent of $f$ and $C_3$ is a numerical constant independent on the data of the problem.
\end{lem}
We will prove this lemma applying the maximum principle in a similar way as it has been done in Theorem 4.2.~in \cite{DV25}. Before giving the proof of Lemma \ref{lemma:rough estimate}, we need to derive a nonlocal equation solved by $\langle R\rangle$.

First of all, let us recall that $f_1=\Div \left(v\otimes v \nabla \psi_0\right)-v\cdot \nabla c-\sigma_a\left(\psi_0+\rho_{0,0}\right)$ and $\langle f_2\rangle=-\sigma_a(c+c_0)$. Let also $f_3=\sigma_a R_0$. Then, by linearity, $R=R^0+R^1+R^2+R^3$, where $R^0$ solves
\begin{equation}\label{eq:R^0}
\begin{cases}
      v\cdot \nabla R^0=\frac{1}{K_n}\left(\langle R^0\rangle-R^0\right)-K_n\sigma_a R^0& \text{in }D\times\Omega\\
        \left.R^0\right|_{\Gamma_-}=\frac{1}{K_n}\left.\left(v\cdot \nabla\psi_0-c\right)\right|_{\Gamma_-}
\end{cases}
\end{equation}
and $R^i$ solves for $i\in\{1,2,3\}$
\begin{equation}\label{eq:R^i}
\begin{cases}
      v\cdot \nabla R^i=\frac{1}{K_n}\left(\langle R^i\rangle-R^i\right)-K_n\sigma_a R+K_n^{i-2}f_i(x,v)& \text{in }D\times\Omega\\
        \left.R^i\right|_{\Gamma_-}=0.
\end{cases}
\end{equation}
For any $x\in D$ and $v\in S^{d-1}$, let us define by $y(x,v)\in\partial D$ the unique intersection point of the half-line starting in $x$ and moving in direction $-v$ with the boundary, i.e. $y(x,v)=\partial D\cap \left\{x-tv:t\geq0\right\}$. Moreover, we also define $s(x,v)=\left|x-y(x,v)\right|$.
Hence, solving by characteristics \eqref{eq:R^0} and \eqref{eq:R^i} yield
\begin{equation}\label{eq:charac.R^0}
    \begin{split}
        R^0(x,v)=&\frac{1}{K_n}\left(v\cdot \nabla\psi_0(y(x,v))-\frac{W_1}{W_2}\partial_n \psi_0(y(x,v))\right)\exp\left(-\int_0^{s(x,v)}\frac{1+K_n^2\sigma_a(x-rv)}{K_n}\ dr\right)\\
        &+\int_0^{s(x,v)} \frac{1}{K_n}\exp{\left(-\int_0^r\frac{1+K_n^2\sigma_a(x-\tau v)}{K_n}\ d\tau\right)}\langle R^0\rangle(x-rv) \ dr
    \end{split}
\end{equation}
and for $i\in\{1,2,3\}$
\begin{equation}\label{eq:charac.R^i}
    \begin{split}
        R^i(x,v)=&\int_0^{s(x,v)} \frac{1}{K_n}\exp{\left(-\int_0^r\frac{1+K_n^2\sigma_a(x-\tau v)}{K_n}\ d\tau\right)}\langle R^0\rangle(x-rv) \ dr\\
        &+\int_0^{s(x,v)} K_n^{i-2}\exp{\left(-\int_0^r\frac{1+K_n^2\sigma_a(x-\tau v)}{K_n}\ d\tau\right)}f_i(x-rv,v) \ dr.
    \end{split}
\end{equation}
We remark that the considerations above, in particular equations \eqref{eq:R^0}, \eqref{eq:R^i}, \eqref{eq:charac.R^0} and \eqref{eq:charac.R^i} hold for any spacial dimension $d\geq 3$.
Integrating now over $S^{2}$ and using the change of variables $(0,\infty)\times S^2\ni(r,v)\mapsto \eta=x-rv\in D$ we obtain, similarly as in \cite{DV25}, the following nonlocal equations solved by the averages in $v\in S^2$ of $R^i$
\begin{equation}\label{eq:nonlocal.R^0}
    \begin{split}
        \langle& R^0\rangle(x)-\int_D \frac{\exp{\left(-\int_{[x,\eta]}\frac{1+K_n^2\sigma_a(\xi)}{K_n}\ d\xi\right)}}{4\pi K_n |x-\eta|^2}\langle R^0\rangle(\eta) \ d\eta\\=&\frac{1}{4\pi K_n}\int_{S^2}\left(v\cdot \nabla\psi_0(y(x,v))-\frac{W_1}{W_2}\partial_n \psi_0(y(x,v))\right)\exp\left(-\int_0^{s(x,v)}\frac{1+K_n^2\sigma_a(x-rv)}{K_n}\ dr\right) \ dv,
    \end{split}
\end{equation}
where $\int_{[x,\eta]} f(\xi) \ d\xi=\int_0^{|x-y|} f\left(x-t \frac{x-y}{|x-y|}\right) \ dt$ denotes the integral along the line connecting $x$ to $\eta$.
Moreover, for $i\in \{1,2,3\}$ we have
\begin{equation}\label{eq:nonlocal.R^i}
    \begin{split}
        \langle R^i\rangle(x)&-\int_D \frac{\exp{\left(-\int_{[x,\eta]}\frac{1+K_n^2\sigma_a(\xi)}{K_n}\ d\xi\right)}}{4\pi K_n |x-\eta|^2}\langle R^i\rangle(\eta) \ d\eta\\=&\frac{K_n^{i-2}}{4\pi}\int_{S^2}\int_0^{s(x,v)}\exp\left(-\int_0^{r}\frac{1+K_n^2\sigma_a(x-\tau v)}{K_n}\ d\tau \right)f_i(x-rv,v)\ dr\ dv.
    \end{split}
\end{equation}
In order to simplify the notation we define the kernel 
\[E^{K_n,\sigma_a}(x,\eta)=\frac{\exp{\left(-\int_{[x,\eta]}\frac{1+K_n^2\sigma_a(\xi)}{K_n}\ d\xi\right)}}{4\pi K_n |x-\eta|^2}\]
and the nonlocal operator
\[\mathcal{L}_D^{K_n,\sigma_a}(u)(x)=u(x)-\int_D E^{K_n,\sigma_a}(x,\eta) u(\eta) \ d\eta.\]
Notice that
\begin{equation}\label{eq:kern}
    \begin{split}
    \int_{\bR^3}E^{K_n,\sigma_a}(x,\eta)\,\diff{\eta}
    & =
    \int_{S^2} \int_0^{s(x,v)} r^2 E^{K_n,\sigma_a}(x,x-rv)\ dr \ dv\\
    & \leq \int_{S^2} \int_0^{s(x,v)} r^2E^{K_n,0}(x,x-rv)\ dr \ dv\\
    & \leq \int_0^\infty\frac{\exp\left(-\frac{r}{K_n}\right)}{K_n}=1,
\end{split}
\end{equation}
where we used $\sigma_a\geq0$. Thus, if $\phi\geq 0$ then $\mathcal{L}_D^{K_n,\sigma_a}(\phi)\geq \mathcal{L}_D^{K_n,0}(\phi). $ We remark that $\mathcal{L}_D^{K_n,0}$ is the nonlocal operator considered in \cite{DV25}. As in \cite{DV25}, the Banach fixed-point theorem guarantees the existence of unique solutions $\langle R^i\rangle\in L^\infty(D)$ for all $i\in\{0,1,2,3\}$. In addition to that, due to the regularity assumptions on $\sigma_a$ and on $f$ the solution $\langle R^i\rangle$ is continuous. Moreover, for any $\sigma_a\geq 0$ the operator $\mathcal{L}_D^{K_n,\sigma_a}$ satisfies the following global maximum principle.
\begin{prop}\label{prop:max}
   Let $u\in C(\overline D)$. Assume that $u$ satisfies one of the following assumptions
   \begin{enumerate}
       \item $\mathcal{L}_D^{K_n,\sigma_a}(u)(x)\geq 0$ for all $x\in D$,
       \item $\mathcal{L}_D^{K_n,\sigma_a}(u)(x)\geq 0$ for all $x\in O\subset D$ open and $u\geq 0$ in $D\setminus O$.
   \end{enumerate} Then $u\geq 0$ in $D$. 
\end{prop}
\begin{proof}
    Assume that there exists $y\in \overline D$ such that $\min\limits_{\overline D} u=u(y)<0$. Then using the property \eqref{eq:kern} the following contradiction arises
    \[0\leq \mathcal{L}_D^{K_n,\sigma_a}(u)(y)\leq u(y)\int_{D^c} E^{K_n,\sigma_a}(y,\eta)d\eta+\int_D E^{K_n,\sigma_a}(y,\eta)\left(u(y)-u(\eta)\right)d\eta<0.\]
\end{proof}

The property $\mathcal{L}_D^{K_n,\sigma_a}[\phi]\geq \mathcal{L}_D^{K_n,0}[\phi]$ for $\phi\geq 0$ implies that the supersolutions constructed in \cite{DV25} are supersolutions also for the operator $\mathcal{L}_D^{K_n,\sigma_a}$ for $\sigma_a\not\equiv 0$. In particular, the function given in Lemma 4.1 of \cite{DV25} given by
\begin{equation}\label{eq:phi1}\phi_{K_n }^1(x)=C_D-|x|^2,\end{equation}
where $C_D=2\max\limits_{\overline D} |x|^2+2 \left(\diam(D)\right)^2+4\diam(D)+4$, satisfies
\[\mathcal{L}_D^{K_n,\sigma_a}\left[\phi_{K_n}^1\right]\geq2 K_n^2.\]
Moreover, consider the function 
\begin{equation}\label{eq:phi2}
    \phi_{K_n,A}^2(x)=C_1(D)\phi_{K_n}^1(x)+C_2(D)\left[\left(1-\frac{\gamma}{1+\left(\frac{d(x)}{AK_n}\right)^2}\right)\wedge\left(1-\frac{\gamma}{1+\left(\frac{\mu R}{AK_n}\right)^2}\right)\right]
\end{equation}
defined in Theorem 4.2 of \cite{DV25} for $A\geq 1$, suitable $\gamma,\mu\in(0,1)$ and suitable constants $C_1(D)$ and $C_2(D)$ which depend only on the domain $D$ whose minimal radius of curvature is denoted by $R=\min\limits_{\partial D}R(x)$. Then, also in this case we have
\[\mathcal{L}_D^{K_n,\sigma_a}\left[\phi_{K_n,A}^2\right]\geq e^{-\frac{d(x)}{AK_n}}.\]
We are now ready to prove Lemma \ref{lemma:rough estimate}.
\begin{proof}[Proof of Lemma \ref{lemma:rough estimate}]
Let $f\in H^{5+\frac{1}{2}}(\partial D)$. Then there exists an extension $\tilde{f}\in H^6(D)$ such that $\tilde{f}=f$ in $\partial D$ in the trace sense. Thus, elliptic Sobolev regularity implies that $\rho_{0,0}\in H^6(D)$. By the Sobolev embedding we conclude also that $\rho_{0,0}\in C^{4,\frac{1}{2}}$. Thus, by Schauder theory we obtain that $\psi_0\in C^{6,\frac{1}{2}}(D)$ and it satisfies
\begin{equation}\label{eq:norm.est.1}
    \Arrowvert D^4\psi_0\Arrowvert_\infty\leq \Arrowvert \sigma_a \rho_{0,0}\Arrowvert_{C^{2,\frac{1}{2}}(D)}\leq \Arrowvert\sigma_a\Arrowvert_{\mathcal{K}}\Arrowvert \rho_{0,0}\Arrowvert_{H^4(K)}
\end{equation}
as well as
\begin{equation}\label{eq:norm.est.2}\Arrowvert D^2\psi_0\Arrowvert_\infty\leq \Arrowvert \sigma_a \rho_{0,0}\Arrowvert_{C^{0,\frac{1}{2}}(D)}\leq \Arrowvert\sigma_a\Arrowvert_{\mathcal{K}}\Arrowvert \rho_{0,0}\Arrowvert_{H^2(K)}.\end{equation}
Moreover, $c,c_0\in C^{3,\frac{1}{2}}(D)$ and satisfy
\begin{equation}\label{eq:norm.est.3}\Arrowvert c\Arrowvert_{C^{3,\frac{1}{2}}(D)}\leq\Arrowvert\sigma_a\Arrowvert_{\mathcal{K}}\Arrowvert \rho_{0,0}\Arrowvert_{H^4(K)}+\Arrowvert\sigma_a\Arrowvert_{\mathcal{K}}\Arrowvert f\Arrowvert_{H^{5+\frac{1}{2}}(\partial D)},\end{equation}
\begin{equation}\label{eq:norm.est.4}\Arrowvert c\Arrowvert_{C^{1,\frac{1}{2}}(D)}\leq  \Arrowvert\sigma_a\Arrowvert_{\mathcal{K}}\Arrowvert \rho_{0,0}\Arrowvert_{H^2(K)}+\Arrowvert\sigma_a\Arrowvert_{\mathcal{K}}\Arrowvert f\Arrowvert_{H^{5+\frac{1}{2}}(\partial D)},\end{equation}
and
\begin{equation}\label{eq:norm.est.5}\Arrowvert c_0\Arrowvert_{C^{3,\frac{1}{2}}(D)}\leq\Arrowvert f\Arrowvert_{H^{5+\frac{1}{2}}(\partial D)}.\end{equation}

Let us consider first $|\langle R^0\rangle|$. Since $\left.\psi_0\right|_{\partial D}=0$ we can estimate
\begin{equation*}
    \begin{split}
       &\bigg|\mathcal{L}_D^{K_n,\sigma_a} [\langle R^0\rangle](x)\bigg|\\
       & \leq \frac{1}{4\pi K_n}\left|\int_{S^2}\left(v\cdot n_{y(x,v)}-\frac{W_1}{W_2}\right)\partial_n \psi_0(y(x,v))\exp\left(-\int_0^{s(x,v)}\frac{1+K_n^2\sigma_a(x-rv)}{K_n}\ dr\right) \ dv\right|\\
       &\leq \left(1+\frac{W_1}{W_2}\right)\frac{\Arrowvert\sigma_a\Arrowvert_{\mathcal{K}}\Arrowvert\rho_{0,0}\Arrowvert_{H^2(K)}}{K_n}e^{-\frac{d(x)}{K_n}}.
    \end{split}
\end{equation*}
Therefore, an application of the maximum principle stated in Proposition \ref{prop:max} with the supersolution $A\phi_{K_n,1}^2$ (for a suitable constant $A>0$ independent of $K_n$) implies
\[\left|\langle R^0\rangle(x)\right|\leq C(D)\Arrowvert\sigma_a\Arrowvert_{\mathcal{K}}\frac{\Arrowvert\rho_{0,0}\Arrowvert_{H^2(K)}}{K_n}.\]
    Next we consider $\langle R^1+R^2\rangle$. To this end we consider the Taylor expansions
\begin{equation}\label{eq:taylor.f1}
    f_1(x-rv,v)=f_1(x,v)-rv\cdot\nabla f_1(x,v)+\frac{r^2}{2}v^T \nabla^2_x f_1(x,v) v+\mathcal{O}\left(|rv|^2\Arrowvert D^2f_1\Arrowvert_\infty\right)
\end{equation}
and \begin{equation}\label{eq:taylor.f2}
    f_2(x-rv,v)=f_2(x,v)-rv\cdot\nabla f_2(x,v)+\mathcal{O}\left(|rv|\Arrowvert Df_2\Arrowvert_\infty\right),
\end{equation}
    where by definition 
    \begin{align*}
    &\Arrowvert D^2f_1\Arrowvert_\infty\leq C(\Arrowvert\sigma_a\Arrowvert_{\mathcal{K}})(\| \rho_{0,0}\|_{H^4(K)}+\Arrowvert f\Arrowvert_{H^{5+\frac{1}{2}}(\partial D)}),\\
    &\Arrowvert Df_2\Arrowvert_\infty\leq C(\Arrowvert\sigma_a\Arrowvert_{\mathcal{K}})(\| \rho_{0,0}\|_{H^4(K)}+\Arrowvert f\Arrowvert_{H^{5+\frac{1}{2}}(\partial D)}).
    \end{align*}

Moreover, we will use that $\langle f_1\rangle=0$ and that $\langle v\cdot\nabla f_1\rangle=-C_d\Delta c=-\sigma_a(c+c_0)=\langle f_2\rangle$. The next results are also helpful for the estimate of $\mathcal{L}_D^{K_n,\sigma_a}\left[\langle R^1+R^2\rangle\right]$.
\begin{align*}
  \int_0^{s(x,v)} \exp{\left(-\int_0^{r}\frac{1+K_n^2\sigma_a(x-\tau v)}{K_n}\ d\tau\right) }  \ dr=& \frac{K_n}{1+K_n^2\sigma_a} \\&+\mathcal{O}\left(K_ne^{-\frac{d(x)}{K_n}}+K_n^3\Arrowvert \nabla\sigma_a\Arrowvert_\infty \left(1-e^{-\frac{d(x)}{K_n}}\right)\right),\\
   \int_0^{s(x,v)} r\exp{\left(-\int_0^{r}\frac{1+K_n^2\sigma_a(x-\tau v)}{K_n}\ d\tau\right) }  \ dr=& \frac{K_n^2}{\left(1+K_n^2\sigma_a\right)^2} \\&+\mathcal{O}\left(K_n^2e^{-\frac{d(x)}{2K_n}}+K_n^4\Arrowvert \nabla\sigma_a\Arrowvert_\infty \left(1-e^{-\frac{d(x)}{2K_n}}\right)\right),\\
    \int_0^{s(x,v)} r^2\exp{\left(-\int_0^{r}\frac{1+K_n^2\sigma_a(x-\tau v)}{K_n}\ d\tau\right) }  \ dr=& \frac{2K_n^3}{\left(1+K_n^2\sigma_a\right)^3} \\&+\mathcal{O}\left(K_n^3e^{-\frac{d(x)}{2K_n}}+K_n^5\Arrowvert \nabla\sigma_a\Arrowvert_\infty \left(1-e^{-\frac{d(x)}{2K_n}}\right)\right).
\end{align*}
Thus,
\begin{equation*}
    \begin{split}
        \bigg|&\mathcal{L}_D^{K_n,\sigma_a}\left[\langle R^1+R^2\rangle\right](x)\bigg|\leq \frac{K_n^3}{\left(1+K_n^2\sigma_a\right)^2}\Arrowvert\sigma_a^{2}(c+c_0)\Arrowvert_\infty\\&+C\left(\Arrowvert \sigma_a\Arrowvert_{\mathcal{K}}\left(\Arrowvert\rho_{0,0}\Arrowvert_{H^2(K)}+\Arrowvert f\Arrowvert_{H^{5+\frac{1}{2}}(\partial D)}\right)\right)\left(e^{-\frac{d(x)}{2K_n}}+\frac{K_n^2}{\left(1+K_n^2\sigma_a\right)^2}+K_n^2\Arrowvert \nabla\sigma_a\Arrowvert_\infty \left(1-e^{-\frac{d(x)}{2K_n}}\right)\right).
    \end{split}
\end{equation*}
An application of the maximum principle (with $A\phi^1_{K_n}$ for a suitable constant $A>0$) stated in Proposition \ref{prop:max} implies that
\[\left|\langle R^1+R^2\rangle(x)\right|\leq C\left(D,\Arrowvert\sigma_a\Arrowvert_{\mathcal{K}}\right)\left(\Arrowvert\rho_{0,0}\Arrowvert_{H^4(K)}+\Arrowvert f\Arrowvert_{H^{5+\frac{1}{2}}(\partial D)}\right).\]
Finally, we estimate the term $\left|\langle R^3\rangle\right|$ in the following way 
\begin{equation*}
    \begin{split}
        \bigg|&\mathcal{L}_D^{K_n,\sigma_a}\left[\langle R^3\rangle\right](x)\bigg|\leq\\& \left|\frac{K_n}{4\pi}\int_{S^2}\int_0^{s(x,v)}\exp\left(-\int_0^{r}\frac{1+K_n^2\sigma_a(x-\tau v)}{K_n}\ d\tau \right)\sigma_a(x-rv)R_0(x-rv,v)\ dr\ dv\right|\\
        &\leq K_n\Arrowvert\sigma_a R_0\Arrowvert_\infty \left|\frac{1}{4\pi}\int_{S^2}\int_0^{s(x,v)}e^{-\frac{r}{K_n}}\ dr\ dv\right|\\
    &\leq K_n^2 \Arrowvert\sigma_a R_0\Arrowvert_\infty \left(1-e^{-\frac{d(x)}{K_n}}\right).
    \end{split}
\end{equation*}
An application of the maximum principle stated in Proposition \ref{prop:max} with the supersolution $\phi_{K_n}^1$ as in \eqref{eq:phi1} gives
\begin{equation}\label{eq:est.R3}
    \left|\langle R^3\rangle(x)\right|\leq C(D)  \Arrowvert\sigma_a R_0\Arrowvert_\infty.
\end{equation}
This concludes the proof of Lemma \ref{lemma:rough estimate}.\end{proof}
From Lemma \ref{lemma:rough estimate} it is clear that we need also to estimate $\Arrowvert\sigma_a R_0\Arrowvert_\infty$. As a first ansatz, we try to repeat the same approach of Lemma \ref{lemma:rough estimate} for $R_0$. On the one hand, we see that the problem \eqref{eq:R0} can be solved by characteristics, so that
\begin{equation}\label{eq:charac.R0}
    \begin{split}
        R_0(x,v)=&\frac{1}{K_n}\left(v\cdot \nabla\rho_{0,0}(y(x,v))-\frac{W_1}{W_2}\partial_n \rho_{0,0}(y(x,v))\right)e^{-\frac{s(x,v)}{K_n}}\\
        &+\int_0^{s(x,v)} \frac{1}{K_n}e^{-\frac{r}{K_n}}\langle R_0\rangle(x-rv) \ dr\\
        &+\int_0^{s(x,v)} \frac{1}{K_n}e^{-\frac{r}{K_n}}\left(\Div\left(v\otimes v\nabla\rho_{0,0}(x-rv)\right)-v\cdot\nabla c_0(x-rv)\right) \ dr
    \end{split}
\end{equation}
and, similarly as for $\langle R\rangle$,
\begin{equation}\label{eq:nonlocal.R0}
    \begin{split}
        \mathcal{L}_D^{K_n,0}\left[\langle R_0\rangle\right](x)=&\frac{1}{4\pi K_n}\int_{S^2}\left(v\cdot \nabla\rho_{0,0}(y(x,v))-\frac{W_1}{W_2}\partial_n \rho_{0,0}(y(x,v))\right)e^{-\frac{s(x,v)}{K_n}} \ dv+\\
        &+\frac{1}{4\pi K_n}\int_{S^2}\int_0^{s(x,v)}e^{-\frac{r}{K_n}}\left(\Div\left(v\otimes v\nabla\rho_{0,0}(x-rv)\right)-v\cdot\nabla c_0(x-rv)\right) \ dr\ dv.
    \end{split}
\end{equation}
Thus, similarly as we proved Lemma \ref{lemma:rough estimate} we can prove the following result:
\begin{lem}\label{lemma:rough estimateR0}
    The solution $R_0\in L^\infty(D\times S^{2})$ of \eqref{eq:R0} satisfies the following bound
    \begin{equation}
        \left|\langle R_0\rangle(x)\right|\leq C_1\frac{\Arrowvert f\Arrowvert_{H^{5+\frac{1}{2}}(\partial D)}}{K_n}+ C_2(D)\Arrowvert f\Arrowvert_{H^{5+\frac{1}{2}}(\partial D)},
    \end{equation}
    where $C_1, \ C_2$ are independent of $K_n$.
\end{lem}
\begin{proof}
By linearity we split $R_0=R_0^1+R_0^2$, where \begin{equation}\label{eq:R01}
    \begin{cases}
        v\cdot \nabla R_0^1=\frac{1}{K_n}\left(\langle R_0^1\rangle-R_0^1\right) & \text{in }D\times\Omega\\
        \left.R_0^1\right|_{\Gamma_-}=\frac{1}{K_n}\left.\left(v\cdot \nabla\rho_{0,0}-\frac{W_1}{W_2}\partial_n \rho_{0,0}\right)\right|_{\Gamma_-}.
    \end{cases}
\end{equation}
and 
\[\begin{cases}
        v\cdot \nabla R_0^2=\frac{1}{K_n}\left(\langle R_0^2\rangle-R_0^2\right)+\frac{1}{K_n}\left( \Div\left(v\otimes v\nabla\rho_{0,0}\right)-v\cdot\nabla c_0\right) & \text{in }D\times\Omega\\
        \left.R_0^2\right|_{\Gamma_-}=0.
    \end{cases}\]
    Thus, $R_0^1$ and $R_0^2$ are solutions to
    \[\begin{split}
        \mathcal{L}_D^{K_n,0}\left[\langle R_0^1\rangle\right](x)=&\frac{1}{4\pi K_n}\int_{S^2}\left(v\cdot \nabla\rho_{0,0}(y(x,v))-\frac{W_1}{W_2}\partial_n \rho_{0,0}(y(x,v))\right)e^{-\frac{s(x,v)}{K_n}} \ dv\\\mathcal{L}_D^{K_n,0}\left[\langle R_0^2\rangle\right](x)=
        &\frac{1}{4\pi K_n}\int_{S^2}\int_0^{s(x,v)}e^{-\frac{r}{K_n}}\left(\Div\left(v\otimes v\nabla\rho_{0,0}(x-rv)\right)-v\cdot\nabla c_0(x-rv)\right) \ dr\ dv.
    \end{split}\]
    We proceed in the same way as we did for Lemma \ref{lemma:rough estimate} using the maximum principle of Proposition \ref{prop:max}, the supersolutions $\phi_{K_n}^1$ and $\phi_{K_n}^2$ defined in \eqref{eq:phi1} and \eqref{eq:phi2}, respectively, as well as the properties of $\rho_{0,0}$ and $c_0$, whose norms are bounded depending on $f$. Moreover, $\langle \Div\left(v\otimes v\nabla\rho_{0,0}\right)-v\cdot\nabla c_0\rangle=0$ as well as $\langle v\cdot \nabla\Div\left(v\otimes v\nabla\rho_{0,0}\right)-\Div\left(v\otimes v\nabla c_0\right)\rangle=0$. Notice that the regularity of $f$ implies that $\rho_{0,0}$ has four bounded derivatives as well as that $c_0$ has three bounded derivatives. Hence,
    \begin{equation*}       
           \left| \mathcal{L}_D^{K_n,0}\left[\langle R_0^1\rangle\right](x)\right|\leq C\frac{\Arrowvert f\Arrowvert_{H^{5+\frac{1}{2}}(\partial D)}}{K_n}e^{-\frac{d(x)}{K_n}}\end{equation*} and \[\left| \mathcal{L}_D^{K_n,0}\left[\langle R_0^2\rangle\right](x)\right|\leq C \Arrowvert f\Arrowvert_{H^{5+\frac{1}{2}}(\partial D)} \left(e^{-\frac{d(x)}{2K_n}}+K_n^2\right),\]
    which implies
   \begin{equation*}
       \left|\langle R_0^1\rangle(x)\right|\leq C_1\frac{\Arrowvert f\Arrowvert_{H^{5+\frac{1}{2}}(\partial D)}}{K_n}\qquad\text{ and }\qquad\left|\langle R_0^2\rangle(x)\right|\leq C_2(D)\Arrowvert f\Arrowvert_{H^{5+\frac{1}{2}}(\partial D)} .
   \end{equation*} Thus, Lemma \ref{lemma:rough estimateR0} is proved.
\end{proof}
Notice that the estimate of Lemma \ref{lemma:rough estimateR0} and equation \eqref{eq:charac.R0} imply that
\begin{equation}\label{eq:rough_R0}
    |R_0(x, v)|\leq C\frac{\Arrowvert f\Arrowvert_{H^{5+\frac{1}{2}}(\partial D)}}{K_n}e^{-\frac{d(x)}{K_n}}+C\frac{\Arrowvert f\Arrowvert_{H^{5+\frac{1}{2}}(\partial D)}}{K_n}+ C\Arrowvert f\Arrowvert_{H^{5+\frac{1}{2}}(\partial D)}\qquad \text{ for all }v\in S^2.
\end{equation}
This estimate is unfortunately not enough in order to estimate $|\langle vR\rangle|$ as in Proposition \ref{prop:estimates}. Hence, we need to improve the estimates for $R_0$ near the boundary $\partial D$.
\subsection{A refined estimate for the remainder $R_0$}\label{subsub:R0}
In this subsection we will improve on the rough bound \eqref{eq:rough_R0} and will show that $|\langle R_0\rangle|$ behaves like $\frac{\Arrowvert f\Arrowvert_{H^{5+\frac{1}{2}}(\partial D)}}{K_n}$ near the boundary $\partial D$, while it is bounded by $\Arrowvert f\Arrowvert_{H^{5+\frac{1}{2}}(\partial D)}$ at distances of order $1$ from the boundary. 

To this end, we see that it is enough to consider $R_0^1$ as in \eqref{eq:R01}, where $R_0=R_0^1+R_0^2$ solves \eqref{eq:R0}. Indeed, in the proof of Lemma \ref{lemma:rough estimateR0} we showed $|\langle R_0^2\rangle|\leq C_2(D)\Arrowvert f\Arrowvert_{H^{5+\frac{1}{2}}(\partial D)}$ and hence $|R_0^2(x, v)|\leq  C\Arrowvert f\Arrowvert_{H^{5+\frac{1}{2}}(\partial D)}$.

Since we want to study the behaviour of $R_0^1$ near the boundary $\partial D$, we need to consider also the boundary layer equation solved by $R_0^1$. Therefore, we define for $p\in\partial D$ the variable $\xi_1=-\frac{x-p}{K_n}\cdot n_p$, where $n_p\in S^2$ is the outer normal, and we assume that near the boundary $R_0^1$ depends only on the direction $-n_p$. Denoting by $y=\xi_1$, we study the one-dimensional problem 
\begin{equation*}\label{eq:b.l.R01}
    \begin{cases}
        -(v\cdot n_p)\partial_y \overline{R}_0^1(y,v;p)=\langle \overline{R}_0^1\rangle-\overline{R}_0^1 &y>0,\;v\in S^{d-1},\\
        \overline{R}_0^1(0,v;p)=\frac{1}{K_n}\left(v\cdot \nabla\rho_{0,0}(p)-\frac{W_1}{W_2}\partial_n \rho_{0,0}(p)\right)&v\cdot n_p<0.
    \end{cases}
\end{equation*}
As we have seen in Subsec.~\ref{subsub:first}, $\langle \overline{R}_0^1\rangle$ solves 
\[\mathcal{L}\left[\langle \overline{R}_0^1\rangle\right](y;p)=\frac{1}{K_n}\frac{1}{4\pi}\int_{\{v\cdot n_p<0\}} \left[(v\cdot n_p)\partial_n\rho_{0,0}(p)-\frac{W_1}{W_2}\partial_n\rho_{0,0}\right]e^{-\frac{y}{|v\cdot n_p|}}dv,\]
where $\mathcal{L}$ is the one-dimensional nonlinear operator defined in \eqref{eq:def.L}.
Moreover, the results in \cite{DV25} imply that $\langle \overline{R}_0^1\rangle(y;p)$ is Lipschitz-continuous with respect to $p$ as well as
\begin{equation}\label{eq:est.<R01>}
    \left|\langle \overline{R}_0^1\rangle(y;p)\right|\leq \left(C_1+\frac{W_1}{W_2}C_2\right) \frac{\Arrowvert f\Arrowvert_{H^{5+\frac{1}{2}}(\partial D)}}{K_n} e^{-\frac{y}{2}}.
\end{equation}
Let $0<\mu< 1$ small enough such that the projection $\pi:D\cap\{x:\ d(x)<\mu\}\to \partial D$ is well-defined, where $\pi(x)$ is the unique point of the boundary such that $d(x)=|x-\pi(x)|$. Moreover, the segment $x-\pi(x)$ is orthogonal to the boundary at the point $\pi(x)$. Let us also fix $\phi\in C^\infty(D)$ such that $0\leq \phi\leq 1$ and 
\[\phi=\begin{cases}
    0 &\text{ if }d(x)>\mu,\\1& \text{ if }d(x)<\frac{\mu}{2}.
\end{cases}\]
We consider the following function
\begin{equation}\label{eq:def.w}
    \overline{w}(x):= \langle \overline{R}_0^1\rangle \left(-\frac{x-\pi(x)}{K_n}\cdot n_{\pi(x)};\pi(x)\right)\phi(x).
\end{equation}
Notice that the function is well-defined in the domain $D$.
We aim to show the following Lemma.
\begin{lem}\label{lem:estimate R01}
The solution $R_0^1\in L^\infty(D\times S^{2})$ of \eqref{eq:R01} satisfies the following bound
    \begin{equation}
        \left|\mathcal{L}_D^{K_n,0}\left[\langle R_0^1\rangle-\overline{w}\right](x)\right|\leq C(D)\Arrowvert f \Arrowvert_{H^{5+\frac{1}{2}}(\partial D)}(e^{-\frac{Ad(x)}{K_n}}+K_n^2),
    \end{equation}
    where $C(D), A=A(D)$ are independent of $K_n$ and of $f$.
\end{lem}
Then, as shown in the end of this subsection, the maximum principle formulated in Proposition \ref{prop:max} and the properties of the boundary layer solution imply the following estimate.
\begin{prop}\label{prop:R01}
    The solution $R_0^1\in L^\infty(D\times S^{2})$ of \eqref{eq:R01} satisfies the following bound
    \begin{equation}
        \left|\langle R_0^1\rangle(x)\right|\leq \frac{C_1(D)\Arrowvert f \Arrowvert_{H^{5+\frac{1}{2}}(\partial D)}}{K_n}e^{-\frac{d(x)}{2K_n}}+C_2\left(D\right)\Arrowvert f \Arrowvert_{H^{5+\frac{1}{2}}(\partial D)},
    \end{equation}
    where $C_1(D), \ C_2(D)$ are independent of $K_n$ and of $f$.
\end{prop}
We now focus on the proof of Lemma \ref{lem:estimate R01} and then, at the end of the section, we will provide the proof of Proposition \ref{prop:R01}. 
In order to show Lemma \ref{lem:estimate R01} we need first to collect some properties of the solution to the (more general) problem
\begin{equation}\label{eq:u1dim}
    \mathcal{L}[u](y;p)=S_p(y):=\frac{1}{4\pi}\int_{\{v\cdot n_p<0\}} g(v,p)e^{-\frac{y}{|v\cdot n_p|}},
\end{equation}
where $\Arrowvert g\Arrowvert_\infty\leq 1$ { such that $u$ satisfies} $|u(y;p)|\leq Ce^{-\frac{y}{2}}$ uniformly with respect to $p\in \partial D$.
\begin{prop}[Collection of the results in \cite{DV25}]\label{prop:collection}
    The solution $u\in L^\infty(D\times S^{2})$ of \eqref{eq:u1dim} satisfies the following properties:
    \begin{enumerate}
        \item[(i)] $u$ is bounded and unique;
        \item[(ii)] $u$ converges with exponential rate to zero as $y\to\infty$, i.e. $|u(y;p)|\leq C e^{-\frac{y}{2}}$ uniformly with respect of $p\in \partial D$;
        \item[(iii)] if $p\to g(v,p)$ is Lipschitz-continuous with respect to $p$, then 
        \begin{align*}|u(y,p)-u(y,q)|\leq C(D) |p-q|
        \end{align*} uniformly with respect to $y$.
    \end{enumerate}
    Moreover, if the the assumption of (iii) holds, we have also
    \begin{equation}\label{eq:prop.collection1}
        \left|u(y;p)-u(y;q)\right|\leq C(D) |p-q|e^{-\frac{y}{2}},
    \end{equation}
    and finally for $\alpha\in(0,1)$ small enough and strcitly positive
    \begin{equation}\label{eq:prop.collection2}
        \left|u(y_1;p)-u(y_2;p)\right|\leq C(\alpha)\left|y_1-y_2\right|^{1-\alpha}e^{-\frac{\min(y_1,y_2)}{4}}.
    \end{equation}
    In particular,
     \begin{equation}\label{eq:prop.collection3}
        \left|u(y_1;p)-u(y_2;p)\right|\leq C(\alpha)e^{-\frac{\min(y_1,y_2)}{4}}\begin{cases}
            \left|y_1^{1-\alpha}-y_2^{1-\alpha}\right|&\text{ if }\ y_1,\ y_2<1\\
            |y_1-y_2|&\text{ if }\ y_1\geq1\ \text{or}\ y_2\geq1.
        \end{cases}
    \end{equation}
\end{prop}
\begin{proof}
    The statements (i)-(iii) are the results of Theorem 3.2, Theorem 3.3, Lemma 3.4, Lemma 3.5, Corollary 3.2 and Lemma 3.6 in \cite{DV25}.

    The estimate \eqref{eq:prop.collection1} is a consequence of Lemma 3.4 and Lemma 3.6 in \cite{DV25}. Indeed, by assumption $|S_p(y)-S_q(y)|\leq C(D)|p-q| e^{-y}$ (cf. Lemma 3.6 in \cite{DV25}). Therefore, $w(y)=\frac{u(y;p)-u(y;q)}{C(D)|p-q|}$ satisfies also Lemma 3.4 in \cite{DV25}. Hence, there exists $w_\infty\in \R$ such that
    \[\left|w(y)-w_\infty\right|\leq C e^{-\frac{y}{2}}.\]
    Since $u$ converges to zero, $w_\infty=0$. This implies \eqref{eq:prop.collection1}.

    As this also requires rather lengthy computations and arguments, we postpone the proof of \eqref{eq:prop.collection2} and of \eqref{eq:prop.collection3} to Subsec.~\ref{subsec:proof_proposition} at the end of Sec.~\ref{sub:diffusion}.
\end{proof}

    We turn now back to the proof of Lemma \ref{lem:estimate R01}.
    \begin{proof}[Proof of Lemma \ref{lem:estimate R01}.] 
    We argue in several steps.

    \emph{Step 1: Set-up and a first estimate.}
    Let $0<\mu< 1$ sufficiently small as assumed in the construction of $\overline{w}$, cf. \eqref{eq:def.w}. Notice that $\mu$ is independent of $K_n$.
    For $x\in D$ with $d(x)<\mu$ we denote by $n_x:=n_{\pi(x)}$ the outer normal at the projected point $\pi(x)\in\partial D$. Moreover, $\Pi_x$ is the halfspace containing $D$ defined by $\Pi_x=\{y\in D: (y-\pi(x))\cdot n_x<0\}$. 
    Notice that $\partial\Pi_x\cap\partial D=\{\pi(x)\}$.
    We compute first
    \begin{equation}\label{eq:I1+2+3}
        \begin{split}
            \bigg|&\cL_D^{K_n,0}\left[\langle R_0^1\rangle-\overline{w}\right](x)\bigg|\\\leq&  \left|\cL_D^{K_n,0}\left[\langle R_0^1\rangle\right](x)-\overline{w}(x)+\phi(x)\int_DE^{K_n,0}(x,\eta)\langle \overline{R}_0^1\rangle\left(-\frac{\eta-\pi(x)}{K_n}\cdot n_x;\pi(x)\right)\ d\eta\right|\\
            &+\left|\int_DE^{K_n,0}(x,\eta)\overline{w}(\eta)\ d\eta-\phi(x)\int_DE^{K_n,0}(x,\eta)\langle \overline{R}_0^1\rangle\left(-\frac{\eta-\pi(x)}{K_n}\cdot n_x;\pi(x)\right)\ d\eta\right|\\ \end{split}
    \end{equation}
    \begin{equation*}\begin{split}
            \leq&  \left|\cL_D^{K_n,0}\left[\langle R_0^1\rangle\right](x)-\phi(x)\cL\left[\overline{R}_0^1\right]\left(-\frac{x-\pi(x)}{K_n}\cdot n_x;\pi(x)\right)\right|\\
            &+\left|\phi(x)\int_{\Pi_x\setminus D}E^{K_n,0}(x,\eta)\langle \overline{R}_0^1\rangle\left(-\frac{\eta-\pi(x)}{K_n}\cdot n_x;\pi(x)\right)\ d\eta\right|\\
            &+\left|\int_DE^{K_n,0}(x,\eta)\left[\overline{w}(\eta)-\phi(x)\langle \overline{R}_0^1\rangle\left(-\frac{\eta-\pi(x)}{K_n}\cdot n_x;\pi(x)\right)\right]\ d\eta\right|\\
             =&I_1+I_2+I_3.
        \end{split}
    \end{equation*}
We next consider the three contributions $I_1, I_2, I_3$ separately.

\emph{Step 2: Estimate for $I_2$}.
        First of all, notice that $I_2=0$ if $d(x)>\mu$. Otherwise, we notice that $I_2$ can be estimated as it has been done for region $A_1$ in Lemma 4.3 in \cite{DV25}. Therefore, using the estimate \eqref{eq:est.<R01>} for $\langle\overline{R}_0^1\rangle$ we compute
        \begin{equation}\label{eq:I_2}
        \begin{split}
            I_2&\leq C \frac{\Arrowvert f\Arrowvert_{H^{5+\frac{1}{2}}(\partial D)}}{K_n}\left|\int_{\Pi_x\setminus D}E^{K_n,0}(x,\eta)\ d\eta\right|\leq  C \frac{\Arrowvert f\Arrowvert_{H^{5+\frac{1}{2}}(\partial D)}}{K_n} C(D) K_n e^{-\frac{d(x)}{4K_n}}\\
            &=C(D)\Arrowvert f\Arrowvert_{H^{5+\frac{1}{2}}(\partial D)} e^{-\frac{d(x)}{4K_n}}.
            \end{split}
        \end{equation}

        \emph{Step 3: Estimate for $I_1$.}
        We next turn to the estimate for $I_1$. First of all, by construction
        \[\begin{split}
            I_1=&\frac{1}{K_n}\frac{1}{4\pi}\left|\int_{S^2}\left[v\cdot \nabla\rho_{0,0}(y(x,v))-\frac{W_1}{W_2}\partial_n \rho_{0,0}(y(x,v))\right]e^{-\frac{|x-y(x,v)|}{K_n}} \ dv\right.\\&\qquad\qquad-\left.\int_{\{v\cdot n_x<0\}} \left[(v\cdot n_x)\partial_n\rho_{0,0}(\pi(x))-\frac{W_1}{W_2}\partial_n\rho_{0,0}(\pi(x))\right]e^{-\frac{|x-\pi(x)|}{K_n|v\cdot n_x|}}dv\right|.
        \end{split}\]
        
        Let us first consider $d(x)\geq \frac{\mu}{2}$. Since $|x-y(x,v)|\geq d(x)$ as well as $|x-\pi(x)|=d(x)$ we can estimate
        \[I_1\leq C \frac{\Arrowvert f\Arrowvert_{H^{5+\frac{1}{2}}(\partial D)}}{\mu} e^{-\frac{d(x)}{2K_n}}.\]
        
        We turn to the case $d(x)<\frac{\mu}{2}$. In order to obtain the desired estimate, we will split this term into three integrals over three different subsets of $S^2$. Specifically, we divide $S^2=A_1\cup A_2 \cup A_3$ in three disjoint sets with the following properties: $A_1\subset \{v\in S^2: v\cdot n_x>0\}$, $A_3\subset \{v\in S^2: v\cdot n_x<0\}$ and $A_2$ has small measure. Moreover, we will perform some geometric estimates making repeatedly use of the local approximation of the boundary by the interior sphere or paraboloid with radius $R$, where $R$ is the minimal radius of curvature of $\partial D$. In the following we also set $\mu<R$.

        We define
        \begin{align}
            A_1=&\left\{v\in S^2:\ \sqrt{d(x)}<v\cdot n_x\leq 1\right\},\label{eq:A1}\\
            A_2=&\left\{v\in S^2:\ -\frac{4}{\sqrt{R}}\sqrt{d(x)}\leq v\cdot n_x\leq \sqrt{d(x)}\right\},\label{eq:A2}\\
            A_3=&\left\{v\in S^2:\ -1\leq v\cdot n_x<-\frac{4}{\sqrt{R}}\sqrt{d(x)}\right\}.\label{eq:A3}
        \end{align}
        Notice that $|A_2|=2\pi \sqrt{d(x)}\left(1+\frac{4}{\sqrt{R}}\right)$. In order to estimate $I_1$ due to the integration over $A_1$ and $A_2$ we need the following geometric estimates. For $x\in D$, $d(x)<\frac{\mu}{2}$ and $v\in S^2$ we define by $\overline{y}(x,v)$ the intersection of the line $\{x-tv:t\geq 0\}$ with the sphere of radius $R$ approximating $\partial D$ at $\pi(x)$ in the interior, see also Fig. \ref{fig:geometry}. Thus,
        \[|x-y(x,v)|\geq |x-\overline{y}(x,v)|,\]
        which is what we aim to estimate now. Without loss of generality we can assume $\pi(x)=0$, $n_x=-e_1$ and $x=(d(x),0,0)^T$. Then, $\overline{y}$ solves for $\lambda\geq 0$
        \begin{equation}\label{eq:interior.sphere}
            \begin{cases}
                (R-\bar y_1)^2+\bar y_2^2+\bar y_3^2=R^2\\
                \bar y=d(x)e_1-\lambda v.
            \end{cases}
        \end{equation}
        We are interested in a bound for $\lambda$, since $|x-\bar{y}|=\lambda$. 

\emph{Step 3a: Estimate in $A_1$.}
        Let us first consider the case in which $v\in A_1$. Then, $v_1<0$ and $|v_1|\geq \sqrt{d(x)}$. Hence,
        \[0=(R-d(x)-\lambda|v_1|)^2+\lambda^2(v_2^2+v_3^2)-R^2=\lambda^2-2\lambda |v_1|(R-d(x))-d(x)(2R-d(x)).\]
        Assuming that $0<\mu< 1$ is sufficiently small so that $d<\frac{R}{2}$ we conclude, since $\lambda>0$, that
        \[\lambda=|v_1|(R-d(x))+\sqrt{|v_1|^2(R-d(x))^2+d(x)(2R-d(x))}\geq 2|v_1|(R-d(x))\geq R |v\cdot n_x|.\]
        Hence, since $\{v\cdot n_x<0\}\cap A_1=\emptyset$
        \begin{equation}
            \begin{split}
                I_1\bigg|_{A_1}\leq& C\frac{\Arrowvert f\Arrowvert_{H^{5+\frac{1}{2}}(\partial D)}}{K_n}\int_{A_1} e^{-\frac{|x-y(x,v)|}{K_n}}dv\leq C\frac{\Arrowvert f\Arrowvert_{H^{5+\frac{1}{2}}(\partial D)}}{K_n}\int_{A_1} e^{-\frac{R|v\cdot n_x|}{K_n}}dv\\
                \leq &C\frac{\Arrowvert f\Arrowvert_{H^{5+\frac{1}{2}}(\partial D)}}{K_n}\int_{\sqrt{d(x)}}^1 e^{-\frac{Rs}{K_n}}ds\leq C \frac{\Arrowvert f\Arrowvert_{H^{5+\frac{1}{2}}(\partial D)}}{R}e^{-\frac{\sqrt{d(x)}R}{K_n}}\leq C(D)\Arrowvert f\Arrowvert_{H^{5+\frac{1}{2}}(\partial D)}e^{-\frac{d(x)R}{K_n}},
            \end{split}
        \end{equation}
        where at the end we used that $\sqrt{d(x)}\geq d(x)$ since $d(x)< 1$.

\emph{Step 3b: Estimate in $A_2$.}
        Next, we consider the case in which $v\in A_2$. A geometric argument shows that $\inf\limits_{v\in A_2}|x-\bar y(x,v)|\geq \inf\limits_{\tilde{v}\in A_2,\ \tilde{v}\cdot n_x<0} |x-\bar y(x,\tilde{v})|$. Thus, we consider now \eqref{eq:interior.sphere} for $0<v_1\leq \frac{4}{\sqrt{R}}\sqrt{d(x)}$. Thus,  
         \[\lambda^2+2\lambda v_1(R-d(x))-d(x)(2R-d(x))=0.\]
         This implies again
         \[\lambda=-v_1(R-d(x))+\sqrt{v_1^2(R-d(x))^2+d(x)(2R-d(x))}\geq C(R)\sqrt{d(x)},\]
         where we used that for $C(R)=\frac{\sqrt{R}}{8}$ the following holds
         \[\begin{split}
             \left(v_1(R-d(x))+C\sqrt{d(x)}\right)^2\leq& v_1^2(R-d(x))^2+d(x)\left(\frac{8C}{\sqrt{R}}(R-d(x))+C^2\right)\\\leq& v_1^2(R-d(x))^2+d(x)(2R-d(x)).
         \end{split}\]
         Notice that the first inequality is because of $0<v_1\leq\frac{4}{\sqrt{R}}\sqrt{d(x)}$ while the second one is due to the definition of $C(R)$. Hence,
         \begin{equation}
            \begin{split}
                I_1\bigg|_{A_2}\leq& C\frac{\Arrowvert f\Arrowvert_{H^{5+\frac{1}{2}}(\partial D)}}{K_n}\left(\int_{A_2} e^{-\frac{|x-y(x,v)|}{K_n}}dv+\int_{\left\{-\frac{4}{R}\sqrt{d(x)}\leq v\cdot n_x<0\right\}} e^{-\frac{|x-\pi(x)|}{K_n|v\cdot n_x|}}dv\right)\\\leq& C\frac{\Arrowvert f\Arrowvert_{H^{5+\frac{1}{2}}(\partial D)}}{K_n}|A_2|\left(2 e^{-\frac{C(R)\sqrt{d(x)}}{K_n}}\right)\leq C(D)\Arrowvert f\Arrowvert_{H^{5+\frac{1}{2}}(\partial D)}\frac{\sqrt{d(x)}}{K_n} e^{-\frac{C(R)\sqrt{d(x)}}{K_n}}\\\leq& C(D)\Arrowvert f\Arrowvert_{H^{5+\frac{1}{2}}(\partial D)}e^{-\frac{C(R)d(x)}{K_n}},
            \end{split}
        \end{equation}
         where at the end we used once again that $\sqrt{d(x)}>d(x)$.

\emph{Step 3c: Estimate in $A_3$.}
         When $v\in A_3$ we proceed in a similar way. We denote by $\tilde{y}(x,v)$ the intersection of $\{x-tv:\ t\geq 0\}$ with the paraboloid approximating from the interior the boundary $\partial D$ at $\pi(x)$, we refer also to Fig. \ref{fig:geometry}. For $v=\left(\sin\theta,\ \cos\theta\sin\varphi, \ \cos\theta\cos\varphi\right)\in A_3$ we can consider without loss of generality $\cos\varphi=0$, i.e. $v=(\sin\theta,\ \cos\theta, \ 0)$, and $\theta\in\left(\theta_1,\frac{\pi}{2}\right)$ where $\sin\theta_1=\frac{4}{\sqrt{R}}\sqrt{d(x)}$. We also denote in the following by $p(x,v)$ the intersection of $\{x-tv:\ t\geq 0\}$ with the boundary of $\Pi_x$. Thus, $|x-p(x,v)|=\frac{|x-\pi(x)|}{|v\cdot n_x|}$. Assuming again without loss of generality $x=(d(x),0,0)$, $\pi(x)=0$ and $n_x=-e_1$, we have
         \[\tilde{y}(x,v)=(\tilde{y}_1,\tilde{y}_2,0) \qquad\text{ and }\qquad p(x,v)=\sigma e_2.\]
         \begin{center}
             \begin{figure}[h]
                 \centering
                 \includegraphics[height=8cm]{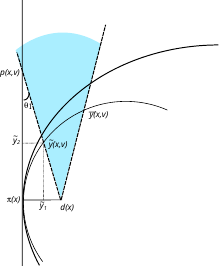}
                 \caption{Illustration of the geometric constructions used in order to estimate the integral $I_1$ over the three regions $A_1,\ A_2$ and $A_3$. The point $x$ is in the interior of the domain $D$ with $d(x)<\frac{\mu}{2}$. The boundary $\partial D$ is approximated in a neighborhood of $\pi(x)$ from the interior by a paraboloid. The line $\{x-tv:\ t\geq0\}$ for $v\cdot n_x=-\sin(\theta_1)$ intersects the approximated boundary in $\tilde{y}(x,v)$ and the boundary $\partial \Pi_x$ in $p(x,v)$. The point $\overline{y}(x,v)$ is the intersection of the approximated boundary with the line $\{x-tv:\ t\geq 0\}$ for $v\cdot n_x=\sqrt{d(x)}$. The portion of the boundary $\partial D$ in the blue region contains all boundary points $y(x,v)$ for $v\in A_2$.
                 }
                 \label{fig:geometry}
             \end{figure}
         \end{center}
         Notice that by definition $\sigma(x,\theta)=\frac{d(x)}{\tan \theta}$. Thus, since $\theta\in\left(\theta_1,\frac{\pi}{2}\right)$ we estimate
         \[\sigma\leq \frac{d(x)}{\tan\theta_1}\leq \frac{\sqrt{Rd(x)}}{4}\cos\theta_1<\frac{\sqrt{Rd(x)}}{4}.\]
         Moreover, we have
         \[|p(x,v)-y(x,v)|\leq |p(x,v)-\tilde{y}(x,v)|=\frac{\sigma-\tilde{y}_2}{\cos\theta}.\]
         By construction we see that \[\tilde{y}_1=|p(x,v)-\tilde{y}(x,v)|\sin\theta=(\sigma-\tilde{y}_2)\tan\theta.\]
         Since $\tilde{y}(x,v)$ lies on the paraboloid approximating the boundary, we impose $\tilde{y}_1=\frac{2}{R}\tilde{y}_2^2$. Hence,
         \[\tilde{y}_2^2+\frac{R}{2}\tilde{y}_2\tan\theta-\frac{R}{2}\sigma \tan\theta=0.\]
         Using that $\tilde{y}_2\geq0$ and the definition of $\sigma$ we compute
         \[\tilde{y}_2=\frac{R}{4}\tan\theta\left(\sqrt{1+\frac{8d(x)}{R\tan^2\theta}}-1\right)\geq \frac{d(x)}{\tan\theta}-\frac{2(d(x))^2}{R \tan^3\theta},\]
         where we used $\sqrt{1+s}-1\geq \frac{1}{2}s-\frac{s^2}{8}$ for all $s>0$. Thus, we estimate on the one hand using that $\tan\theta\geq\sin\theta\geq \frac{4\sqrt{d(x)}}{\sqrt{R}}$
         \[|p(x,v)-y(x,v)|\leq \frac{\sigma-\tilde{y}_2}{\cos\theta}=\frac{d(x)}{\sin\theta}-\frac{\tilde{y}_2}{\cos\theta}\leq \frac{2d(x)}{\sin\theta}\frac{d(x)}{R\tan^2\theta}\leq  \frac{2(d(x))^2}{R\sin^3\theta}\leq \frac{d(x)}{8\sin\theta}. \]
         On the other hand, we also have
         \[|x-y(x,v)|\geq|x-\tilde{y}(x,v)|=|p(x,v)-x|-|p(x,v)-\tilde{y}(x,v)|\geq \frac{7}{8}\frac{d(x)}{\sin\theta}.\]
         Changing to spherical coordinates, this implies that
         \[\begin{split}
             \frac{1}{4\pi K_n}\int_{A_3}& \left|e^{-\frac{|x-y(x,v)|}{K_n}}-e^{-\frac{|x-p(x,v)|}{K_n}}\right|\ dv\leq\frac{1}{4\pi K_n}\int_{A_3} \frac{|p(x,v)-y(x,v)|}{K_n}e^{-\frac{|x-y(x,v)|}{K_n}}\ dv\\
             \leq&\frac{1}{2}\int_{\theta_1}^{\frac{\pi}{2}}\frac{2 (d(x))^2}{K_n^2R\sin^3\theta}e^{-\frac{7d(x)}{8K_n\sin\theta}}\cos\theta\ d\theta\leq C(D)\int_{\theta_1}^{\frac{\pi}{2}}\frac{d(x)}{K_n\sin^2\theta}e^{-\frac{7d(x)}{16K_n\sin\theta}}\cos\theta\ d\theta\\\leq& C(D)e^{-\frac{7d(x)}{16K_n}},
         \end{split}\]
since $\frac{d}{d\theta}e^{-A/\sin(\theta)}=\frac{A\cos(\theta)}{\sin^2(\theta)}e^{-A/\sin(\theta)}$ and $d(x)<\sqrt{d(x)}$.   
Finally, using that $\sin\theta=|v\cdot n_x|$ we have \[|y(x,v)-\pi(x)|\leq |p(x,v)-y(x,v)|+|p(x,v)-\pi(x)|\leq \frac{9}{8}\frac{d(x)}{|v\cdot n_x|}.\]
Notice furthermore that by symmetry \[\int_{A_3} v\cdot \nabla g(\pi(x)) e^{-\frac{d(x)}{|v\cdot n_x|}}\ dv=\int_{A_3} v\cdot n_x \partial_n g(\pi(x)) e^{-\frac{d(x)}{|v\cdot n_x|}}\ dv.\]
These observations imply
\[\begin{split}
             \frac{1}{4\pi K_n}\int_{A_3} 
             &\left|v\cdot \nabla\left[\rho_{0,0}(y(x,v))-\rho_{0,0}(\pi(x))\right]-\frac{W_1}{W_2}\partial_n \left[\rho_{0,0}(y(x,v))-\rho_{0,0}(\pi(x))\right]\right|e^{-\frac{d(x)}{K_n|v\cdot n_x|}}\ dv\\
             &\leq  C \frac{\Arrowvert\nabla^2 \rho_{0,0}\Arrowvert_\infty}{K_n}\int_{A_3} |y(x,v)-\pi(x)|e^{-\frac{d(x)}{K_n|v\cdot n_x|}}\ dv\\
             &\leq C \Arrowvert f\Arrowvert_{H^{5+\frac{1}{2}}(\partial D)}\int_{A_3} \frac{d(x)}{K_n|v\cdot n_x|}e^{-\frac{d(x)}{K_n|v\cdot n_x|}}\ dv \\
             &\leq  C \Arrowvert f\Arrowvert_{H^{5+\frac{1}{2}}(\partial D)}e^{-\frac{d(x)}{2K_n}}.
         \end{split}\]
    Hence, using the triangle inequality we conclude that
    \[I_1\bigg|_{A_3}\leq C(D) \Arrowvert f\Arrowvert_{H^{5+\frac{1}{2}}(\partial D)}e^{-\frac{C(D)d(x)}{K_n}}. \]
    Thus, summarising 
    \begin{equation}\label{eq:I_1}
        I_1\leq C(D) \Arrowvert f\Arrowvert_{H^{5+\frac{1}{2}}(\partial D)}e^{-\frac{Ad(x)}{K_n}},
    \end{equation}
    for $0<\mu< 1$ fixed and sufficiently small and $0<A(D)<1$.

    \emph{Step 4: Estimate for $I_3$.}
Finally, the integral term $I_3$ has to be estimated in a careful way. Notice that if $K_n\geq K_{n_0}$, the triangle inequality implies
\[I_3\leq \frac{C \Arrowvert f\Arrowvert_{H^{5+\frac{1}{2}}(\partial D)}}{K_n}\int_{D} E^{K_n,0}(x,\eta)\ d\eta\leq C(D, K_{n_0})\frac{ \Arrowvert f\Arrowvert_{H^{5+\frac{1}{2}}(\partial D)}}{K_n}e^{-\frac{d(x)}{K_n}}.\] Thus, it is enough to prove the desired estimate for $I_3$ for $K_n<K_{n_0}<1$ for some $K_{n_0}$ fixed and small enough.

First of all, if $d(x)>\mu$ we conclude 
\begin{equation}\label{eq:I_3:1}
   \begin{split}
        I_3 &\leq \int_DE^{K_n,0}(x,\eta)\left|\overline{w}(\eta)\right|\ d\eta\leq \frac{C \Arrowvert f\Arrowvert_{H^{5+\frac{1}{2}}(\partial D)}}{K_n}\int_{D} E^{K_n,0}(x,\eta) e^{-\frac{d(\eta)}{K_n}}\\
        &\leq \frac{C \Arrowvert f\Arrowvert_{H^{5+\frac{1}{2}}(\partial D)}}{K_n}\left(\int_{\{d(\eta)<\mu/2\}}\frac{e^{-\frac{|x-\eta|}{2K_n}}e^{-\frac{\mu}{2K_n}}}{4\pi K_n|x-\eta|^2}+\int_{\{d(\eta)>\mu/2\}}\frac{e^{-\frac{|x-\eta|}{K_n}}e^{-\frac{\mu}{2K_n}}}{4\pi K_n|x-\eta|^2}\right)\\
        & \leq C(\mu) \Arrowvert f\Arrowvert_{H^{5+\frac{1}{2}}(\partial D)} K_n^2.
   \end{split}
\end{equation}
In a very similar way, using that by construction $|(\eta-\pi(x))\cdot n_x|\geq d(\eta)$ we can conclude that also if $d(x)\geq \frac{\mu}{2}$ we can estimate $I_3\leq C(\mu)\Arrowvert f\Arrowvert_{H^{5+\frac{1}{2}}(\partial D)} K_n^2$. Moreover, since for $0<\mu< 1$ sufficiently small the distance function is Lipschitz continuous with Lipschitz constant $1$ in $\{z:d(z)<\mu\}$, we can always estimate there $e^{-\frac{d(\eta)}{K_n}}\leq e^{-\frac{d(x)}{2K_n}}e^{\frac{|x-\eta|}{2K_n}}$ for all $d(x), d(\eta)<\mu$. This can be used in order to show that it is also enough to consider only the case in which $d(x)<K_n^{1-\delta_0}<\frac{\mu}{2}$ for $0<\delta_0<1$ and $K_n<K_{n_0}$ small enough. Indeed, for $K_n^{1-\delta_0}\leq d(x)<\frac{\mu}{2}$ we estimate similarly as above in \eqref{eq:I_3:1}
\begin{equation}\label{eq:I_3:2}
   \begin{split}
        I_3\leq&  \frac{C \Arrowvert f\Arrowvert_{H^{5+\frac{1}{2}}(\partial D)}}{K_n}\int_{\{d(\eta)<\mu\}} E^{K_n,0}(x,\eta) e^{-\frac{d(\eta)}{K_n}}\leq \frac{C \Arrowvert f\Arrowvert_{H^{5+\frac{1}{2}}(\partial D)} e^{-\frac{d(x)}{2K_n}}}{K_n}\int_{\{d(\eta)<\mu\}}\frac{e^{-\frac{|x-\eta|}{2K_n}}}{4\pi K_n|x-\eta|^2}\\\leq& C(D) \Arrowvert f\Arrowvert_{H^{5+\frac{1}{2}}(\partial D)}e^{-\frac{d(x)}{4K_n}} \frac{e^{-\frac{K_n^{-\delta_0}}{4}}}{K_n}\leq C(D,\delta_0)\Arrowvert f\Arrowvert_{H^{5+\frac{1}{2}}(\partial D)}e^{-\frac{d(x)}{4K_n}}.
   \end{split}
\end{equation}
Finally, since 
\[\begin{split}
         \int_{\{d(\eta)\geq \mu/2\}}E^{K_n,0}(x,\eta)\left|\overline{w}(\eta)\right|\ d\eta\leq& \frac{C \Arrowvert f\Arrowvert_{H^{5+\frac{1}{2}}(\partial D)}}{K_n}\int_{D} E^{K_n,0}(x,\eta) e^{-\frac{\mu}{2K_n}}\leq C(\mu) \Arrowvert f\Arrowvert_{H^{5+\frac{1}{2}}(\partial D)} K_n^2,
   \end{split}\]
   we see that it is enough to estimate for $d(x)<K_n^{1-\delta_0}<\frac{\mu}{2}$ the following
 \begin{equation}\label{eq:I_3:3}
   \begin{split}
         \int_{\{d(\eta)<\mu/2\}}&E^{K_n,0}(x,\eta)\left|\overline{w}(\eta)-\langle \overline{R}_0^1\rangle\left(-\frac{\eta-\pi(x)}{K_n}\cdot n_x;\pi(x)\right)\ \right|\ d\eta\\
         \leq \int_{\{d(\eta)<\mu/2\}}&E^{K_n,0}(x,\eta)\left|\overline{w}(\eta)-\langle \overline{R}_0^1\rangle\left(-\frac{\eta-\pi(x)}{K_n}\cdot n_x;\pi(\eta)\right)\ \right|\ d\eta\\+\int_{\{d(\eta)<\mu/2\}}&E^{K_n,0}(x,\eta)\left|\langle \overline{R}_0^1\rangle\left(-\frac{\eta-\pi(x)}{K_n}\cdot n_x;\pi(\eta)\right)-\langle \overline{R}_0^1\rangle\left(-\frac{\eta-\pi(x)}{K_n}\cdot n_x;\pi(x)\right)\ \right|\ d\eta\\=&J_1+J_2.
   \end{split}
\end{equation}  
In the following we will make use of some useful geometric estimates which are a consequence of basic differential geometry arguments. For $d(x),d(y)<\mu/2$ and $|\pi(x)-\pi(y)|<2\mu$ one can prove that
\begin{equation}\label{eq:geom.est.}
    |\pi(x)-\pi(y)|\leq C(R)|x-y|;\ |(y-\pi(y))\cdot n-(y-\pi(x))\cdot n|\leq C(R)|x-y|^2 \text{ and } |n_x-n_y|\leq C(R) |x-y|,
\end{equation}
where $n\in\{n_x,n_y\}$.

\emph{Step 4a: Estimate of $J_2$.}
Let us first consider $J_2$. In this case, using \eqref{eq:prop.collection1} and \eqref{eq:geom.est.} we estimate
\begin{equation}\label{eq:J2}
   \begin{split}
        J_2\leq&  C\Arrowvert f \Arrowvert_{H^{5+\frac{1}{2}}(\partial D)}\int_{\{d(\eta)<\mu/2\}}E^{K_n,0}(x,\eta)e^{-\frac{(\eta-\pi(x))\cdot n_x}{K_n}} \frac{|\pi(x)-\pi(\eta)|}{K_n}d\eta\\\leq&  C(D)\Arrowvert f \Arrowvert_{H^{5+\frac{1}{2}}(\partial D)}\int_{\{d(\eta)<\mu/2\}}E^{K_n,0}(x,\eta) e^{-\frac{(\eta-\pi(x))\cdot n_x}{K_n}}\frac{|x-\eta|}{K_n}d\eta\\\leq& C(D)\Arrowvert f \Arrowvert_{H^{5+\frac{1}{2}}(\partial D)}e^{-\frac{d(x)}{2K_n}}\int_D \frac{e^{-\frac{|x-\eta|}{2K_n}}}{4\pi K_n |x-\eta|^2}\frac{|x-\eta|}{K_n}d\eta\leq C(D)\Arrowvert f \Arrowvert_{H^{5+\frac{1}{2}}(\partial D)}e^{-\frac{d(x)}{2K_n}},       
    \end{split}
\end{equation}
where we used that $(\eta-\pi(x))\cdot n_x\geq d(\eta)$.

\emph{Step 4b: Estimate of $J_1$.}
We move now to the last estimate, i.e. $J_1$. In order to obtain the desired estimate we divide the domain $D\cap\left\{d(y)<\frac{\mu}{2}\right\}$ in the following four regions (cf. Fig. \ref{fig:regions})
\[D\cap\left\{d(y)<\frac{\mu}{2}\right\}\subset A_1\cup A_2 \cup A_3 \cup A_4\] defined for $\delta>\max\{\delta_0,\alpha\}$ by
\[A_1=\left\{K_n\leq d(y)<\frac{\mu}{2}\right\},\ A_2=\left\{d(y)<K_n: \ |x-y|>K_n^{1-\delta}\right\}, \ A_3=\left\{d(y)<K_n: |x-y|\leq 4 K_n \ \right\}\]
 and \[A_4=\begin{cases}
     \left\{d(y)<K_n:\ K_n<|x-y|<K_n^{1-\delta}\right\} & \text{if }d(x)\geq 2K_n,\\
     \left\{d(y)<4K_n: \ 4K_n<|x-y|<K_n^{1-\delta}\right\} & \text{if }d(x)< 2K_n.
 \end{cases}\]
 We remark that the subdivision is not disjoint. Indeed, $A_3\cap A_4=A_3$ for $d(x)\in[2K_n,5K_n)$ as well as $A_3=\emptyset$ if $d(x)\geq5K_n$.
 \begin{center}
 \begin{figure}[h]
     \includegraphics[width=0.5\linewidth]{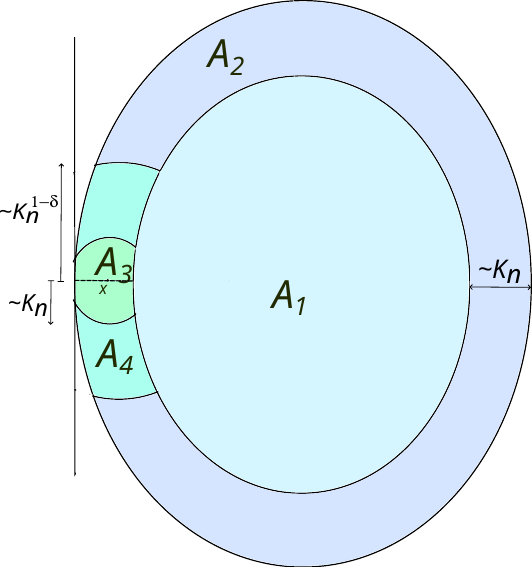}
     \caption{Subdivision of the domain $D$ in the regions $A_1,\ A_2,\ A_3$ and $A_4$. In this case $d(x)<2K_n$.
     }
     \label{fig:regions}
 \end{figure}
 \end{center}
 Since $(\eta-\pi(x))\cdot n_x\geq d(\eta)\geq K_n$ for all $\eta \in A_1$ we conclude, using \eqref{eq:prop.collection3}, that for all $\eta\in A_1$
  \begin{equation}\label{eq:J1:A1}
   \begin{split}
       & \bigg|\overline{w}(\eta)-\langle \overline{R}_0^1\rangle\left(-\frac{\eta-\pi(x)}{K_n}\cdot n_x;\pi(\eta)\right)\ \bigg|
        \leq C \Arrowvert f\Arrowvert_{H^{5+\frac{1}{2}}(\partial D)}e^{-\frac{d(\eta)}{K_n}}\frac{|(\eta-\pi(\eta))\cdot n_\eta-(\eta-\pi(x))\cdot n_x|}{K_n^2}\\
        & \qquad \leq C\Arrowvert f\Arrowvert_{H^{5+\frac{1}{2}}(\partial D)} e^{-\frac{d(\eta)}{K_n}}\frac{|(\pi(\eta)-\pi(x))\cdot n_x|+|(\eta-\pi(\eta))\cdot (n_\eta-n_x)|}{K_n^2}\\
        & \qquad\leq C(R)\Arrowvert f\Arrowvert_{H^{5+\frac{1}{2}}(\partial D)}e^{-\frac{d(\eta)}{K_n}}\left(\frac{|x-\eta|^2}{K_n^2}+\frac{|x-\eta|d(\eta)}{K_n^2}\right)\\
        &  \qquad\leq{ C(R)\Arrowvert f\Arrowvert_{H^{5+\frac{1}{2}}(\partial D)}e^{-\frac{d(x)}{4K_n}}e^{\frac{|x-\eta|}{2K_n}}\left(\frac{|x-\eta|^2}{K_n^2}+\frac{|x-\eta|}{K_n}\right)},
   \end{split}
\end{equation} 
where we used also \eqref{eq:geom.est.}. Hence,
 \begin{equation}\label{eq:J1:A1_2}
   \begin{split}
        \int_{A_1}&E^{K_n,0}(x,\eta)\left|\overline{w}(\eta)-\langle \overline{R}_0^1\rangle\left(-\frac{\eta-\pi(x)}{K_n}\cdot n_x;\pi(\eta)\right)\ \right|\ d\eta\\
        \leq & C(D) \Arrowvert f\Arrowvert_{H^{5+\frac{1}{2}}(\partial D)}e^{-\frac{d(x)}{4K_n}}\int_{A_1}\frac{e^{-\frac{|x-\eta|}{2K_n}}}{4\pi K_n |x-\eta|^2}\left(\frac{|x-\eta|^2}{K_n^2}+\frac{|x-\eta|}{K_n}\right)d\eta\leq C(D) \Arrowvert f\Arrowvert_{H^{5+\frac{1}{2}}(\partial D)}e^{-\frac{d(x)}{4K_n}}.
   \end{split}
\end{equation}  
For $\eta\in A_2$ we use \eqref{eq:prop.collection2} obtaining in a similar way as in \eqref{eq:J1:A1}
\begin{equation}\label{eq:J1:A2}
   \begin{split}
        \bigg|\overline{w}(\eta)-&\langle \overline{R}_0^1\rangle\left(-\frac{\eta-\pi(x)}{K_n}\cdot n_x;\pi(\eta)\right)\ \bigg|\leq C(R)\Arrowvert f\Arrowvert_{H^{5+\frac{1}{2}}(\partial D)}e^{-\frac{d(x)}{4K_n}}e^{\frac{|x-\eta|}{2K_n}}\left(\frac{|x-\eta|^{2-2\alpha}}{K_n^{2-\alpha}}+\frac{|x-\eta|^{1-\alpha}}{K_n}\right).
   \end{split}
\end{equation} 
Thus, changing to spherical coordinates we conclude
 \begin{equation}\label{eq:J1:A2_2}
   \begin{split}
        \int_{A_2}&E^{K_n,0}(x,\eta)\left|\overline{w}(\eta)-\langle \overline{R}_0^1\rangle\left(-\frac{\eta-\pi(x)}{K_n}\cdot n_x;\pi(\eta)\right)\ \right|\ d\eta\\
        \leq & C(D) \Arrowvert f\Arrowvert_{H^{5+\frac{1}{2}}(\partial D)}e^{-\frac{d(x)}{4K_n}}\int_{A_2}\frac{e^{-\frac{|x-\eta|}{2K_n}}}{4\pi K_n |x-\eta|^2}\left(\frac{|x-\eta|^{2-2\alpha}}{K_n^{2-\alpha}}+\frac{|x-\eta|^{1-\alpha}}{K_n}\right)d\eta\\\leq& C(D,\alpha) \Arrowvert f\Arrowvert_{H^{5+\frac{1}{2}}(\partial D)}\frac{e^{-\frac{d(x)}{4K_n}}}{K_n^\alpha} \int_{K_n^{-\delta}}^\infty e^{-\frac{r}{4}} dr\\
        \leq&  C(D,\alpha) \Arrowvert f\Arrowvert_{H^{5+\frac{1}{2}}(\partial D)}e^{-\frac{d(x)}{4K_n}}\frac{e^{-(4K_n)^{-\delta}}}{K_n^\alpha}\leq C(D,\alpha,\delta) \Arrowvert f\Arrowvert_{H^{5+\frac{1}{2}}(\partial D)}e^{-\frac{d(x)}{4K_n}},
   \end{split}
\end{equation} 
where we used several times that $a^\kappa e^{-a}\leq c(\kappa)e^{-a/2}$. We remark that \eqref{eq:J1:A2} holds for all $\eta \in D$ with $d(\eta)<\mu/2$.

Let us now consider $\eta\in A_3\cup A_4$. Without loss of generality we assume that $\pi(x)=0$ and $n_x=-e_1$. Moreover, by construction $\eta_1\geq d(\eta)\geq (\eta-\pi(\eta))_1$. Thus, using the result of Lemma \ref{lem:u.derivative} we can estimate
\begin{equation}\label{eq:J1:A34}
   \begin{split}
        \bigg|\overline{w}(\eta)-&\langle \overline{R}_0^1\rangle\left(-\frac{\eta-\pi(x)}{K_n}\cdot n_x;\pi(\eta)\right)\ \bigg|\leq C\Arrowvert f\Arrowvert_{H^{5+\frac{1}{2}}(\partial D)}e^{-\frac{d(\eta)}{K_n}}\int_{\frac{d(\eta)}{K_n}}^{\frac{\eta_1}{K_n}}\frac{1}{t^\alpha}dt\\\leq &C\Arrowvert f\Arrowvert_{H^{5+\frac{1}{2}}(\partial D)}e^{-\frac{d(x)}{2K_n}}e^{\frac{|x-\eta|}{2K_n}}\int_{\frac{(\eta-\pi(\eta))_1}{K_n}}^{\frac{\eta_1}{K_n}}\frac{1}{t^\alpha}dt\\\leq & C(\alpha)\Arrowvert f\Arrowvert_{H^{5+\frac{1}{2}}(\partial D)}e^{-\frac{d(x)}{2K_n}}e^{\frac{|x-\eta|}{2K_n}}\left(\left(\frac{\eta_1}{K_n}\right)^{1-\alpha}\left|1-\left(1-\frac{\pi(\eta)_1}{\eta_1}\right)^{1-\alpha}\right|\right).
   \end{split}
\end{equation}
Thus, using that in $[0,1/2]$ the function $f(x)=1-(1-x)^{1-\alpha}$ is Lipschitz continuous together with \eqref{eq:geom.est.} we obtain
\[\left(\left(\frac{\eta_1}{K_n}\right)^{1-\alpha}\left|1-\left(1-\frac{\pi(\eta)_1}{\eta_1}\right)^{1-\alpha}\right|\right)\leq \frac{\pi(\eta)_1}{K_n}\frac{K_n^\alpha}{\eta_1^\alpha}\leq C(D)\frac{|x-\eta|^2}{K_n}\frac{K_n^\alpha}{\eta_1^\alpha}\]
if $\pi(\eta)_1< \frac{1}{2}\eta_1$. If $\pi(\eta)_1\leq \eta_1$ we conclude as well
\[\left(\left(\frac{\eta_1}{K_n}\right)^{1-\alpha}\left|1-\left(1-\frac{\pi(\eta)_1}{\eta_1}\right)^{1-\alpha}\right|\right)\leq \frac{K_n^{1-\alpha}}{\eta_1^{1-\alpha}}\leq \frac{\pi(\eta)_1}{2K_n}\frac{K_n^\alpha}{\eta_1^\alpha}\leq C(D)\frac{|x-\eta|^2}{K_n}\frac{K_n^\alpha}{\eta_1^\alpha}.\]
Hence,
\begin{equation}\label{eq:A34}
    \bigg|\overline{w}(\eta)-\langle \overline{R}_0^1\rangle\left(-\frac{\eta-\pi(x)}{K_n}\cdot n_x;\pi(\eta)\right)\ \bigg|\leq C(D,\alpha)\Arrowvert f\Arrowvert_{H^{5+\frac{1}{2}}(\partial D)}e^{-\frac{d(x)}{2K_n}}e^{\frac{|x-\eta|}{2K_n}}\frac{|x-\eta|^2}{K_n^2}\frac{K_n^\alpha}{\eta_1^\alpha}. 
\end{equation}

We turn now to the estimate for the region $A_3$. Since for $K_n>0$ small enough $A_3\subset D\cap ([0,2K_n]\times[-4K_n,4K_n]^2)$ we can easily estimate
 \begin{equation}\label{eq:J1:A3}
   \begin{split}
        \int_{A_3}&E^{K_n,0}(x,\eta)\left|\overline{w}(\eta)-\langle \overline{R}_0^1\rangle\left(-\frac{\eta-\pi(x)}{K_n}\cdot n_x;\pi(\eta)\right)\ \right|\ d\eta\\
        \leq & C(D,\alpha) \Arrowvert f\Arrowvert_{H^{5+\frac{1}{2}}(\partial D)}e^{-\frac{d(x)}{2K_n}}\int_{[0,2K_n]\times[-4K_n,4K_n]^2}\frac{e^{-\frac{|x-\eta|}{2K_n}}}{4\pi K_n |x-\eta|^2}\frac{|x-\eta|^{2}}{K_n^2}\frac{K_n^\alpha}{\eta_1^\alpha}d\eta\\\leq & C(D,\alpha) \Arrowvert f\Arrowvert_{H^{5+\frac{1}{2}}(\partial D)}e^{-\frac{d(x)}{2K_n}}\int_{[0,2K_n]\times[-4K_n,4K_n]^2}\frac{1}{4\pi K_n^3 }\frac{K_n^\alpha}{\eta_1^\alpha}d\eta\\\leq& C(D,\alpha) \Arrowvert f\Arrowvert_{H^{5+\frac{1}{2}}(\partial D)}e^{-\frac{d(x)}{2K_n}}\int_0^{2K_n}\frac{1}{K_n^{1-\alpha}\eta_1^\alpha}d\eta_1\leq C(D,\alpha) \Arrowvert f\Arrowvert_{H^{5+\frac{1}{2}}(\partial D)}e^{-\frac{d(x)}{2K_n}}.
   \end{split}
\end{equation} 
It remains to estimate the integral term over the region $A_4$. This is another cumbersome estimate. Let us first consider the case in which $d(x)\geq2K_n$. Then, for $K_n>0$ small enough a simple geometric argument implies that $A_4\subset \left[0,\frac{3}{2}K_n\right]\times \left[-\frac{\sqrt{2}}{2}K_n^{1-\delta},\frac{\sqrt{2}}{2}K_n^{1-\delta}\right]^2$. In the following we consider $\eta=x-rn$ in spherical coordinates, where $n_1=\cos(\theta)$ for $\theta\in\left[0,\frac{\pi}{2}\right]$. Let us define $\theta_1,\ \theta_2\in\left[0,\frac{\pi}{2}\right]$ by $\tan(\theta_1)=\frac{K_n^{1-\delta}}{\sqrt{2}d(x)}$ and $\tan(\theta_2)=\frac{K_n^{1-\delta}}{\sqrt{2}\left(d(x)-\frac{3}{2}K_n\right)}$ 
. Since for $K_n$ small enough and for $\delta>\delta_0$ we have $\frac{\sqrt{2}}{2}K_n^{1-\delta}\leq \left(d(x)^2+\frac{1}{2}K_n^{2-2\delta}\right)^{\frac{1}{2}}\leq K_n^{1-\delta}$, we see that $\frac{d(x)}{K_n}K_n^\delta\leq\cos(\theta_1)\leq \frac{2d(x)}{K_n}K_n^\delta$. Also, another simple geometric argument implies that $0\leq \theta_2-\theta_1\leq 3K_n^\delta$ for $K_n$ small enough. Denoting by $s(x,n)$ the distance of $x$ to the boundary $\partial D$ in direction $-n$ we estimate $s(x,n)\leq \frac{d(x)}{\cos(\theta)}$. Moreover, $\eta_1=d(x)-r\cos(\theta)$ as well as $r\geq \frac{d(x)}{\cos(\theta)}-\frac{3}{2}\frac{K_n}{\cos(\theta)}$. We now divide the region $\eta=x-rn\in \left[0,\frac{3}{2}K_n\right]\times \left[-\frac{\sqrt{2}}{2}K_n^{1-\delta},\frac{\sqrt{2}}{2}K_n^{1-\delta}\right]^2$ in the regions in which $\theta\in[0,\frac{\pi}{4}]$ (denoted by $A_{4}^1$), $\theta\in(\frac{\pi}{4},\theta_1)$ (denoted by $A_{4}^2$) and $\theta\in[\theta_1,\theta_2]$ (denoted by $A_{4}^3$) and we estimate the integral term changing to spherical coordinates. See also Fig. \ref{fig:another subdivision} for a picture of the subdivision of $A_4$.

Using \eqref{eq:A34} and that $e^{-a}a^2\leq Ce^{-\frac{a}{2}}$ for $a\geq 0$, we estimate first
\begin{equation}\label{eq:J1:A41}
   \begin{split}
        \int_{A_4^1}&E^{K_n,0}(x,\eta)\left|\overline{w}(\eta)-\langle \overline{R}_0^1\rangle\left(-\frac{\eta-\pi(x)}{K_n}\cdot n_x;\pi(\eta)\right)\ \right|\ d\eta\\
        \leq & C(D,\alpha) \Arrowvert f\Arrowvert_{H^{5+\frac{1}{2}}(\partial D)}e^{-\frac{d(x)}{2K_n}}\int_{A_4^1}\frac{e^{-\frac{|x-\eta|}{2K_n}}}{4\pi K_n |x-\eta|^2}\frac{|x-\eta|^{2}}{K_n^2}\frac{K_n^\alpha}{\eta_1^\alpha}d\eta\\
        \leq & C(D,\alpha) \Arrowvert f\Arrowvert_{H^{5+\frac{1}{2}}(\partial D)}e^{-\frac{d(x)}{2K_n}}\frac{1}{2}\int_{0}^{\frac{\pi}{4}}\sin(\theta)\int_{\frac{d(x)}{\cos(\theta)}-\frac{3}{2}\frac{K_n}{\cos(\theta)}}^{\frac{d(x)}{\cos(\theta)}}\frac{e^{-\frac{r}{4K_n}}}{K_n}\frac{K_n^\alpha}{\left(\frac{d(x)}{\cos(\theta)}-r\right)^\alpha\cos(\theta)^\alpha}dr\ d\theta\\\end{split}\end{equation}\begin{equation*}\begin{split}
        \leq & C(D,\alpha) \Arrowvert f\Arrowvert_{H^{5+\frac{1}{2}}(\partial D)}e^{-\frac{d(x)}{2K_n}}\frac{1}{2}\int_{0}^{\frac{\pi}{4}}\sin(\theta)\int_{\frac{d(x)}{\cos(\theta)}-\frac{3\sqrt{2}}{4}K_n}^{\frac{d(x)}{\cos(\theta)}}\frac{e^{-\frac{r}{4K_n}}}{K_n}\frac{K_n^\alpha}{\left(\frac{d(x)}{\cos(\theta)}-r\right)^\alpha\cos(\theta)^\alpha}dr\ d\theta\\
        \leq &C(D,\alpha)\frac{1}{(1-\alpha)^2}\left(\frac{3\sqrt{2}}{4}\right)^{1-\alpha} \Arrowvert f\Arrowvert_{H^{5+\frac{1}{2}}(\partial D)}e^{-\frac{d(x)}{2K_n}},
   \end{split}
\end{equation*}  where we used that $\cos(\theta)\geq \frac{\sqrt{2}}{2}$ for the set under consideration.

In a similar way changing also the order of integration we estimate
\begin{equation}\label{eq:J1:A42}
   \begin{split}
        \int_{A_4^2}&E^{K_n,0}(x,\eta)\left|\overline{w}(\eta)-\langle \overline{R}_0^1\rangle\left(-\frac{\eta-\pi(x)}{K_n}\cdot n_x;\pi(\eta)\right)\ \right|\ d\eta\\
        \leq & C(D,\alpha) \Arrowvert f\Arrowvert_{H^{5+\frac{1}{2}}(\partial D)}e^{-\frac{d(x)}{2K_n}}\frac{1}{2}\int_{\frac{\pi}{4}}^{\theta_1}\sin(\theta)\int_{\frac{d(x)}{\cos(\theta)}-\frac{3}{2}\frac{K_n}{\cos(\theta)}}^{\frac{d(x)}{\cos(\theta)}}\frac{e^{-\frac{r}{4K_n}}}{K_n}\frac{K_n^\alpha}{\left(\frac{d(x)}{\cos(\theta)}-r\right)^\alpha\cos(\theta)^\alpha}dr\ d\theta\\
        \leq & C(D,\alpha) \Arrowvert f\Arrowvert_{H^{5+\frac{1}{2}}(\partial D)}e^{-\frac{d(x)}{2K_n}}\frac{1}{2}\int_{\frac{\pi}{4}}^{\theta_1}\frac{\sin(\theta)}{\cos(\theta)}\int_{\frac{d(x)}{K_n}-\frac{3}{2}}^{\frac{d(x)}{K_n}}e^{-\frac{r}{4\cos(\theta)}}\frac{1}{\left(\frac{d(x)}{K_n}-r\right)^\alpha}dr\ d\theta\\
        \leq  &C(D,\alpha) \Arrowvert f\Arrowvert_{H^{5+\frac{1}{2}}(\partial D)}e^{-\frac{d(x)}{2K_n}}\frac{1}{2}\int_{\frac{d(x)K_n^\delta}{K_n}}^{\frac{\sqrt{2}}{2}}\frac{1}{z}\int_{\frac{d(x)}{K_n}-\frac{3}{2}}^{\frac{d(x)}{K_n}}e^{-\frac{r}{4z}}\frac{1}{\left(\frac{d(x)}{K_n}-r\right)^\alpha}dr\ dz\\
        = & C(D,\alpha) \Arrowvert f\Arrowvert_{H^{5+\frac{1}{2}}(\partial D)}e^{-\frac{d(x)}{2K_n}}\frac{1}{2}\int_{\frac{d(x)}{K_n}-\frac{3}{2}}^{\frac{d(x)}{K_n}}\frac{1}{\left(\frac{d(x)}{K_n}-r\right)^\alpha}\int_{\frac{r\sqrt{2}}{4}}^{\frac{r K_n}{4d(x)K_n^\delta}}\frac{e^{-y}}{y}dy\ dr\\
         \leq & C(D,\alpha) \frac{E\left(\frac{\sqrt{2}}{8}\right)}{1-\alpha}\left(\frac{3}{2}\right)^{1-\alpha} \Arrowvert f\Arrowvert_{H^{5+\frac{1}{2}}(\partial D)}e^{-\frac{d(x)}{2K_n}},
   \end{split}
\end{equation}
where we first changed the coordinates according to $r\mapsto \frac{r}{K_n}\cos(\theta)$, then according to $z=\cos(\theta)$ and finally we changed the order of integration performing another change of coordinates replacing $ z$ by $\frac{r}{4y}$. In the end we also used that $r\geq \frac{1}{2}$.

For the last region, i.e. $A_4^3$, we use estimate \eqref{eq:J1:A2} together with the well-known inequality $e^{-a}a^{2-2\alpha}\leq e^{-\frac{a}{2}}$ for $a\geq 0$ and we conclude
\begin{equation}\label{eq:J1:A43}
   \begin{split}
        \int_{A_4^3}&E^{K_n,0}(x,\eta)\left|\overline{w}(\eta)-\langle \overline{R}_0^1\rangle\left(-\frac{\eta-\pi(x)}{K_n}\cdot n_x;\pi(\eta)\right)\ \right|\ d\eta\\
        \leq & C(D) \Arrowvert f\Arrowvert_{H^{5+\frac{1}{2}}(\partial D)}e^{-\frac{d(x)}{4K_n}}\int_{A_4^3}\frac{e^{-\frac{|x-\eta|}{2K_n}}}{4\pi K_n |x-\eta|^2}\frac{|x-\eta|^{2-2\alpha}}{K_n^{2-\alpha}}d\eta\\\leq& C(D,\alpha) \Arrowvert f\Arrowvert_{H^{5+\frac{1}{2}}(\partial D)}\frac{e^{-\frac{d(x)}{4K_n}}}{K_n^\alpha} \int_{\theta_1}^{\theta_2}\sin(\theta)\int_0^\infty e^{-\frac{r}{4}} dr d\theta\\
        \leq&  C(D,\alpha) \Arrowvert f\Arrowvert_{H^{5+\frac{1}{2}}(\partial D)}e^{-\frac{d(x)}{4K_n}}\frac{\theta_2-\theta_1}{K_n^\alpha}\leq C(D,\delta) \Arrowvert f\Arrowvert_{H^{5+\frac{1}{2}}(\partial D)}e^{-\frac{d(x)}{4K_n}},
   \end{split}
\end{equation} 
where we used that $\theta_2-\theta_1\leq K_n^\delta\leq K_n^\alpha$ for $\delta\geq\alpha$. 
\begin{center}
    \begin{figure}[h]
        \centering
        \includegraphics[width=0.8\linewidth]{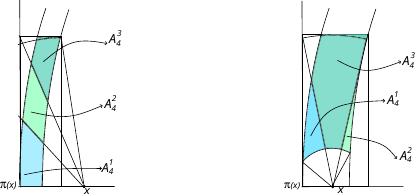}
        \caption{Schematic illustration of the subdivision of region $A_4$. On the left it is represented the case in which $d(x)\geq 2K_n$, while on the right the case in which $d(x)<2K_n$. Notice that only the upper region (i.e. the case in which $\eta_2>0$) is visible.}
        \label{fig:another subdivision}
    \end{figure}
\end{center}
It remains to consider the case in which $d(x)<2K_n$. Again, a simple geometric argument shows that for $K_n$ small enough we have $A_4\subset B_{4K_n}^{c}(x)\cap \left([0, (4+\frac{1}{2})K_n]\times\left[-\frac{\sqrt{2}}{2}K_n^{1-\delta},\frac{\sqrt{2}}{2}K_n^{1-\delta}\right]^2 \right)$. We will divide this region again in three different regions and we will consider $\eta=x-rn$ in spherical coordinates with $n_1=\cos(\theta)$. Hence, for $\theta\in(0,\frac{\pi}{2})$ we have again $s(x,n)\leq \frac{d(x)}{\cos(\theta)}$ and $y_1=d(x)-r\cos(\theta)$, while if $\theta>\frac{\pi}{2}$ we estimate $r\leq \frac{(4+\frac{1}{2})K_n-d(x)}{\cos(\pi-\theta)}$ and $y_1=d(x)+r\cos(\pi-\theta)$. Moreover, $r\geq 4K_n$. We also define the following four angles. Let $\cos(\theta_0)=\frac{d(x)}{4K_n}$ and $\cos(\theta_3)=\frac{(4+\frac{1}{2})K_n-d(x)}{4K_n}$. Moreover we define $\tan(\theta_1)=\frac{K_n^{1-\delta}}{\sqrt{2}d(x)}$ and $\tan(\theta_2)=\frac{K_n^{1-\delta}}{\sqrt{2}\left((4+\frac{1}{2})K_n-d(x)\right)}$. In a similar way as we estimated before we see that for $K_n$ small enough and for $\delta>\delta_0$ it is true that $\frac{d(x)}{K_n}K_n^\delta\leq \cos(\theta_1)\leq 2\frac{d(x)}{K_n}K_n^\delta$ and that $0\leq(\pi-\theta_2)-\theta_1 \leq 3K_n^\delta$. We divide then $\eta=x-rn\in  B_{4K_n}^{c}(x)\cap \left( [0, (4+\frac{1}{2})K_n]\times\left[-\frac{\sqrt{2}}{2}K_n^{1-\delta},\frac{\sqrt{2}}{2}K_n^{1-\delta}\right]^2 \right)$ in the three regions $A_4^1,\ A_4^2,\ A_4^3$ defined similarly as before by $\theta\in[\theta_0,\theta_1]$, $\theta\in[\pi-\theta_2,\pi-\theta_3]$ and $\theta\in(\theta_1,\pi-\theta_2)$, respectively (cf. Fig. \ref{fig:another subdivision}).
Clearly, the integral term over $A_3^3$ can be estimated by \eqref{eq:J1:A43}. For the regions $A_4^1$ and $A_4^2$ we proceed in a similar way as we did in \eqref{eq:J1:A42} and we obtain on the one hand
\begin{equation}\label{eq:J1:A411}
   \begin{split}
        \int_{A_4^1}&E^{K_n,0}(x,\eta)\left|\overline{w}(\eta)-\langle \overline{R}_0^1\rangle\left(-\frac{\eta-\pi(x)}{K_n}\cdot n_x;\pi(\eta)\right)\ \right|\ d\eta\\
        \leq & C(D,\alpha) \Arrowvert f\Arrowvert_{H^{5+\frac{1}{2}}(\partial D)}e^{-\frac{d(x)}{2K_n}}\frac{1}{2}\int_{\theta_0}^{\theta_1}\frac{\sin(\theta)}{\cos(\theta)}\int_{4\cos(\theta)}^{\frac{d(x)}{K_n}}e^{-\frac{r}{4\cos(\theta)}}\frac{1}{\left(\frac{d(x)}{K_n}-r\right)^\alpha}dr\ d\theta\\
        \leq & C(D,\alpha) \Arrowvert f\Arrowvert_{H^{5+\frac{1}{2}}(\partial D)}e^{-\frac{d(x)}{2K_n}}\frac{1}{2}\int_{\frac{d(x)K_n^\delta}{K_n}}^{\frac{d(x)}{4K_n}}\frac{1}{z}\int_{4z}^{\frac{d(x)}{K_n}}e^{-\frac{r}{4z}}\frac{1}{\left(\frac{d(x)}{K_n}-r\right)^\alpha}dr\ dz\\\end{split}\end{equation}\begin{equation*}\begin{split}
        = & C(D,\alpha) \Arrowvert f\Arrowvert_{H^{5+\frac{1}{2}}(\partial D)}e^{-\frac{d(x)}{2K_n}}\frac{1}{2}\int_{4\frac{d(x)}{K_n}K_n^\delta}^{\frac{d(x)}{K_n}}\frac{1}{\left(\frac{d(x)}{K_n}-r\right)^\alpha}\int_{1}^{\frac{r K_n}{d(x)K_n^\delta}}\frac{e^{-y}}{y}dy\ dr\\
         \leq & C(D,\alpha) \frac{E\left(1\right)}{1-\alpha}\left(\frac{d(x)}{K_n}(1-4K_n^\delta)\right)^{1-\alpha} \Arrowvert f\Arrowvert_{H^{5+\frac{1}{2}}(\partial D)}e^{-\frac{d(x)}{2K_n}}\leq C(D,\alpha)\Arrowvert f\Arrowvert_{H^{5+\frac{1}{2}}(\partial D)}e^{-\frac{d(x)}{4K_n}}.
   \end{split}
\end{equation*}
On the other hand, using that $\sin(\theta)=\sin(\pi-\theta)$ we estimate also
\begin{equation}\label{eq:J1:A421}
   \begin{split}
        \int_{A_4^2}&E^{K_n,0}(x,\eta)\left|\overline{w}(\eta)-\langle \overline{R}_0^1\rangle\left(-\frac{\eta-\pi(x)}{K_n}\cdot n_x;\pi(\eta)\right)\ \right|\ d\eta\\
        \leq & C(D,\alpha) \Arrowvert f\Arrowvert_{H^{5+\frac{1}{2}}(\partial D)}e^{-\frac{d(x)}{2K_n}}\frac{1}{2}\int_{\theta_3}^{\theta_2}\frac{\sin(\theta)}{\cos(\theta)}\int_{4\cos(\theta)}^{4+\frac{1}{2}-\frac{d(x)}{K_n}}e^{-\frac{r}{4\cos(\theta)}}\frac{1}{\left(\frac{d(x)}{K_n}+r\right)^\alpha}dr\ d\theta\\
        = & C(D,\alpha) \Arrowvert f\Arrowvert_{H^{5+\frac{1}{2}}(\partial D)}e^{-\frac{d(x)}{2K_n}}\frac{1}{2}\int_{\arccos{(\theta_2)}}^{1+\frac{1}{8}-\frac{d(x)}{4K_n}}\frac{1}{z}\int_{4z}^{4+\frac{1}{2}-\frac{d(x)}{K_n}}e^{-\frac{r}{4z}}\frac{1}{\left(\frac{d(x)}{K_n}+r\right)^\alpha}dr\ dz\\
        = & C(D,\alpha) \Arrowvert f\Arrowvert_{H^{5+\frac{1}{2}}(\partial D)}e^{-\frac{d(x)}{2K_n}}\frac{1}{2}\int_{4\arccos{(\theta_2)}}^{4+\frac{1}{2}-\frac{d(x)}{K_n}}\frac{1}{\left(\frac{d(x)}{K_n}+r\right)^\alpha}\int_{1}^{\frac{r}{4\arccos(\theta_2)}}\frac{e^{-y}}{y}dy\ dr\\
         \leq & C(D,\alpha) \frac{E\left(1\right)}{1-\alpha}\left(\frac{9}{2}\right)^{1-\alpha} \Arrowvert f\Arrowvert_{H^{5+\frac{1}{2}}(\partial D)}e^{-\frac{d(x)}{2K_n}}.
   \end{split}
\end{equation}
Thus, estimates \eqref{eq:J1:A41}-\eqref{eq:J1:A421} imply 
\begin{equation}\label{eq:J1:A4}
   \begin{split}
        \int_{A_4}&E^{K_n,0}(x,\eta)\left|\overline{w}(\eta)-\langle \overline{R}_0^1\rangle\left(-\frac{\eta-\pi(x)}{K_n}\cdot n_x;\pi(\eta)\right)\ \right|\ d\eta\\
        \leq &C(D,\alpha) \Arrowvert f\Arrowvert_{H^{5+\frac{1}{2}}(\partial D)}e^{-\frac{d(x)}{4K_n}}
   \end{split}
\end{equation}
for every $d(x)<K_n^{1-\delta_0}$. Finally, putting together estimates \eqref{eq:J1:A1_2}, \eqref{eq:J1:A2_2}, \eqref{eq:J1:A3} and \eqref{eq:J1:A4} we can conclude that for any fixed $0<\alpha<\delta< 1$, $0<\delta_0<\delta< 1$ and $0<\mu<1$ sufficiently small and independent of $K_n$ there exists a constant $C(D)>0$ such that 
\begin{equation*}\label{eq:J1}
    J_1\leq  C(D) \Arrowvert f\Arrowvert_{H^{5+\frac{1}{2}}(\partial D)}e^{-\frac{d(x)}{4K_n}}
\end{equation*}
so that
\begin{equation}\label{eq:I_3}
    I_3\leq  C(D) \Arrowvert f\Arrowvert_{H^{5+\frac{1}{2}}(\partial D)}\left(e^{-\frac{d(x)}{4K_n}}+K_n^2\right).
\end{equation}
Estimates \eqref{eq:I_1}, \eqref{eq:I_2} and \eqref{eq:I_3} imply Lemma \ref{lem:estimate R01}.
    \end{proof}
    \begin{rmk}
        We emphasize that although the geometric constructions used in order to estimate \eqref{eq:I1+2+3} and in particular $I_1,\ I_2$ and $I_3$ have been done for $D\subset \R^3$, these arguments can be repeated also for higher dimensions, i.e. for $D\subset\R^d$ with $d>3$. Indeed, the subdivision of $S^2=A_1\cup A_2 \cup A_3$ as defined in \eqref{eq:A1}, \eqref{eq:A2} and \eqref{eq:A3} can be generalized also for $S^{d-1}$ as well as the domain $D$ can be divided into $D\subset A_1\cup A_2\cup A_3\cup A_4$ as we did in order to estimate $I_3$ (cf. Figure \ref{fig:regions}). Moreover, all the geometric arguments in the proof of Lemma \ref{lem:estimate R01} have been achieved using that for any $x\in D$ with $d(x)\ll 1$ there exists an isometry $\phi$ such that $\phi(\pi(x))=0$ and $n_{\phi(x)}=-e_1$. In this way we could assume without of generality that $x=d(x)e_1$, $\pi(x)=0$ and $n_x=-e_1$. This holds for any dimension $d\geq 3$.
    \end{rmk}
We show now how Lemma \ref{lem:estimate R01} implies Proposition \ref{prop:R01}.
\begin{proof}[Proof of Proposition \ref{prop:R01}]
    Lemma \ref{lem:estimate R01} and an application of the maximum principle of Proposition \ref{prop:max} with the supersolution $\psi=C(D)\Arrowvert f\Arrowvert_{H^{5+\frac{1}{2}}(\partial D)} \left(\phi^2_{K_n,A}+\frac{1}{2}\phi^1_{K_n}\right)$, where $A$ and $C(D)$ are as in Lemma \ref{lem:estimate R01}, implies that
    \[\left|\langle R_0^1\rangle(x)-\overline{w}(x)\right|\leq \psi(x).\]
    Since $|\psi(x)|\leq C_2(D)\Arrowvert f\Arrowvert_{H^{5+\frac{1}{2}}(\partial D)} $ and by definition $\left|\overline{w}(x)\right|\leq C_1(D)\frac{\Arrowvert f\Arrowvert_{H^{5+\frac{1}{2}}(\partial D)}}{K_n}e^{-\frac{d(x)}{2K_n}}$ we conclude
    \[\left|\langle R_0^1\rangle(x)\right|\leq\left|\overline{w}(x)\right|+ |\psi(x)|\leq C_1(D)\frac{\Arrowvert f\Arrowvert_{H^{5+\frac{1}{2}}(\partial D)}}{K_n}e^{-\frac{d(x)}{2K_n}}+C_2(D)\Arrowvert f\Arrowvert_{H^{5+\frac{1}{2}}(\partial D)}.\]
\end{proof}
\subsection{Proof of Theorem \ref{thm:elena_juan_estimates}}
With the previous results in hand, we provide the proof of Theorem \ref{thm:elena_juan_estimates}. Let us first consider the case $d=3$.
\begin{proof}[Proof of Theorem \ref{thm:elena_juan_estimates} for $d=3$]
The rougher estimates of Lemma \ref{lemma:rough estimate} and Lemma \ref{lemma:rough estimateR0} and the estimate \eqref{eq:rough_R0} imply the existence of two constants $C_1(D,\Arrowvert \sigma_a\Arrowvert_{\mathcal{K}})>0$ and $C_2(D,\Arrowvert \sigma_a\Arrowvert_{\mathcal{K}})$ such that for $K_n\in(0,1]$:
\[\left|\langle R_a\rangle(x)\right|\leq\left|\langle R_0\rangle(x)\right|+\left|\langle R\rangle(x)\right|\leq C_1\frac{\Arrowvert\rho_{0,0}\Arrowvert_{H^4(K)}}{K_n}+C_2\frac{\Arrowvert f\Arrowvert_{H^{5+\frac{1}{2}}(\partial D)}}{K_n}, \]
which implies \eqref{eq:elenajuan_first_estimate} since $K_n\leq 1$  and since the constants $C_1,\ C_2$ are monotone in $\Arrowvert \sigma_a\Arrowvert_{\mathcal{K}}$ and $\Arrowvert\sigma_a\Arrowvert_{\cK}\leq M$.

From Proposition \ref{prop:R01} using that $R_0=R_0^1+R_0^2$ we see that there exist two constants $C_1=C_1(D)>0$ and $C_2=C_2(D)>0$ which depend only on the domain $D$ such that
    \[\left| \langle R_0\rangle(x)\right|\leq C_1\frac{\Arrowvert f\Arrowvert_{H^{5+\frac{1}{2}}(\partial D)}}{K_n}e^{-\frac{d(x)}{2K_n}}+C_2\Arrowvert f\Arrowvert_{H^{5+\frac{1}{2}}(\partial D)}.\]
    Thus, since for every $x-rv\in D$, $r\geq 0$ and $v\in \Omega$ the triangle inequality implies that $d(x)\leq d(x-rv)+r$ we can estimate using equation \eqref{eq:charac.R0}
    \[\begin{split}
        \sup\limits_{v\in\Omega}\left| R_0(x,v)\right|\leq & C \frac{\Arrowvert \nabla \rho_{0,0}\Arrowvert_{\infty}}{K_n}e^{-\frac{d(x)}{K_n}}+\sup\limits_{v\in\Omega} \int_0^{s(x,v)}\frac{1}{K_n}e^{-\frac{r}{K_n}}\left|\langle R_0\rangle(x-rv)\right| \ dr\\
        &+\sup\limits_{v\in\Omega} \int_0^{s(x,v)} \frac{1}{K_n}e^{-\frac{r}{K_n}}\left|\Div\left(v\otimes v\nabla\rho_{0,0}(x-rv)\right)-v\cdot\nabla c_0(x-rv)\right| \ dr\\
        \leq &  C \frac{\Arrowvert \nabla \rho_{0,0}\Arrowvert_{\infty}}{K_n}e^{-\frac{d(x)}{K_n}}+C(D) \frac{\Arrowvert f\Arrowvert_{H^{5+\frac{1}{2}}(\partial D)}}{K_n}e^{-\frac{d(x)}{2K_n}}\int_0^{s(x,v)}\frac{1}{K_n}e^{-\frac{r}{2K_n}} \ dr\\
        &+C(D)\Arrowvert f\Arrowvert_{H^{5+\frac{1}{2}}(\partial D)}\int_0^{s(x,v)}\frac{1}{K_n}e^{-\frac{r}{K_n}}\ dr \\
        &+C\left(\Arrowvert \nabla^2 \rho_{0,0}\Arrowvert_{\infty}+\Arrowvert\nabla c_0\Arrowvert_\infty\right)\int_0^{s(x,v)}\frac{1}{K_n}e^{-\frac{r}{K_n}}\ dr\\
        \leq& C_1(D) \frac{\Arrowvert f\Arrowvert_{H^{5+\frac{1}{2}}(\partial D)}}{K_n}e^{-\frac{d(x)}{2K_n}}+C_2(D)\Arrowvert f\Arrowvert_{H^{5+\frac{1}{2}}(\partial D)},
    \end{split}\]
    where $C_1,\ C_2$ depend only on $D$ and where we used also that $\rho_{0,0}$ and $c_0$ solve \eqref{eq:leading order rho} and \eqref{eq:def.c}, which implies their regularity so that by Sobolev embedding and Schauder theory
    \[\Arrowvert \rho_{0,0}\Arrowvert_{C^{2,\frac{1}{2}}(D)}\leq \Arrowvert f\Arrowvert_{H^{5+\frac{1}{2}}(\partial D)} \qquad \text{ and }\qquad \Arrowvert c_0\Arrowvert_{C^{1,\frac{1}{2}}}(D)\leq \Arrowvert f\Arrowvert_{H^{5+\frac{1}{2}}(\partial D)}.\]

    Let us define by $\kappa_0=\text{dist}(K,\partial D)>0$. Since $\sigma_a(x)R_0(x,v)\equiv 0$ if $x\in D\setminus K$ and hence if $d(x)<\kappa_0$, we can conclude that
    \[\left|\sigma_a(x)R_0(x,v)\right|\leq C(\kappa_0,D) \Arrowvert \sigma_a\Arrowvert_\infty \Arrowvert f\Arrowvert_{H^{5+\frac{1}{2}}(\partial D)},\]
    where we used that $e^{-\frac{d(x)}{K_n}}\mathbb{1}_{K}\leq e^{-\frac{\kappa_0}{K_n}}\mathbb{1}_{K}\leq\mathbb{1}_{K} \frac{K_n}{\kappa_0}. $ We can therefore refine the estimate for $\langle R\rangle$ since
    \[\left|\langle R^3\rangle(x)\right|\leq C(D) \Arrowvert \sigma_aR_0\Arrowvert_\infty\leq C(D,K,\Arrowvert \sigma_a\Arrowvert_{\mathcal{K}})\Arrowvert f\Arrowvert_{H^{5+\frac{1}{2}}(\partial D)}, \]
    where $R^3 $ solves \eqref{eq:R^i} and satisfies \eqref{eq:est.R3} as proved in Lemma \ref{lemma:rough estimate}. Thus, we compute using the results in Lemma \ref{lemma:rough estimate}
    \begin{equation}\label{eq:ref.est.R}
     \left|\langle R\rangle(x)\right|\leq   C_1(D,\Arrowvert \sigma_a\Arrowvert_{\mathcal{K}})\frac{\Arrowvert\rho_{0,0}\Arrowvert_{H^4(K)}}{K_n}+C_2(D,K,\Arrowvert \sigma_a\Arrowvert_{\mathcal{K}})\Arrowvert f\Arrowvert_{H^{5+\frac{1}{2}}(\partial D)},
    \end{equation}
    where $C_1, \ C_2$ are independent of $K_n\in(0,1]$.

Next we can estimate $\left|R(x,v)\right|$ using the formulas \eqref{eq:charac.R^0} and \eqref{eq:charac.R^i} obtained solving \eqref{eq:R^0} and \eqref{eq:R^i} by characteristics. First of all we recall that by definition
\[\Arrowvert f_1\Arrowvert_\infty \leq \Arrowvert\psi_0\Arrowvert_{C^2(D)}+\Arrowvert c\Arrowvert_{C^1(D)}+\Arrowvert \sigma_a (\psi_0+\rho_{0,0})\Arrowvert_{C^0(D)}\leq C(\Arrowvert \sigma_a\Arrowvert_{\mathcal{K}}) \left(\Arrowvert \rho_{0,0}\Arrowvert_{H^2(K)}+\Arrowvert f\Arrowvert_{H^{5+\frac{1}{2}}(\partial D)}\right)
\]
 and 
 \[\Arrowvert f_2\Arrowvert_\infty \leq \Arrowvert\sigma_a(\psi_0+\rho_{0,0})\Arrowvert_{C^1(D)}+\Arrowvert \sigma_a (c+c_0)\Arrowvert_{C^0(D)}\leq C(\Arrowvert \sigma_a\Arrowvert_{\mathcal{K}}) \left(\Arrowvert \rho_{0,0}\Arrowvert_{H^2(K)}+\Arrowvert f\Arrowvert_{H^{5+\frac{1}{2}}(\partial D)}\right).
\]
Moreover, $\Arrowvert \psi_0\Arrowvert_{C^1(D)}\leq C(\Arrowvert \sigma_a\Arrowvert_{\mathcal{K}}a)\Arrowvert \rho_{0,0}\Arrowvert_{H^2(K)}$. Thus,
\begin{equation}\label{eq:estimate|R|}
 \begin{split}
      \left|R(x,v)\right|\leq& \left|R^0(x,v)\right|+\left|R^1(x,v)+R^2(x,v)\right|+\left|R^3(x,v)\right|\\\leq& C(D,\Arrowvert \sigma_a\Arrowvert_{\mathcal{K}}) \frac{\Arrowvert \rho_{0,0}\Arrowvert_{H^2(K)}}{K_n}+C(D,\Arrowvert \sigma_a\Arrowvert_{\mathcal{K}})\left(\Arrowvert \rho_{0,0}\Arrowvert_{H^4(K)}+\Arrowvert f\Arrowvert_{H^{5+\frac{1}{2}}(\partial D)}\right)\\
      &+C(D, K,\Arrowvert \sigma_a\Arrowvert_{\mathcal{K}})\Arrowvert f\Arrowvert_{H^{5+\frac{1}{2}}(\partial D)}\\
      \leq& C_1(D,\Arrowvert \sigma_a\Arrowvert_{\mathcal{K}})\frac{\Arrowvert \rho_{0,0}\Arrowvert_{H^4(K)}}{K_n}+C_2(D, K,\Arrowvert \sigma_a\Arrowvert_{\mathcal{K}})\Arrowvert f\Arrowvert_{H^{5+\frac{1}{2}}(\partial D)},
 \end{split}
\end{equation}
for all $K_n\in(0,1]$. Finally, since once again the constants $C_1(D,\Arrowvert\sigma_a\Arrowvert_\cK),\ C_2(D,\Arrowvert\sigma_a\Arrowvert_\cK,K)$ are monotone with respect to $\Arrowvert\sigma_a\Arrowvert_\cK$ and since $\Arrowvert\sigma_a\Arrowvert_{\cK}\leq M$, we can conclude the proof of Theorem \ref{thm:elena_juan_estimates} as follows
\[\Arrowvert\langle vR\rangle\Arrowvert_\infty\leq \sup\limits_{(x,v)\in D\times\Omega}\left|R(x,v)\right|\leq C_1(D,M)\frac{\Arrowvert \rho_{0,0}\Arrowvert_{H^4(K)}}{K_n}+C_2(D, K, M)\Arrowvert f\Arrowvert_{H^{5+\frac{1}{2}}(\partial D)}.\]
\end{proof}

\subsection{Extension to higher dimensions}
\label{sec:higher_dim}
We extend now the proof of Theorem \ref{thm:elena_juan_estimates} to any spacial dimension $d\geq 3$. Notice that, as we already remarked, until Subsec.~\ref{subsub:remainder} the constructions and the results are obtained for general space dimensions $d\geq 3$. Indeed, the main differences appear in the explicit nonlocal equations obtained solving the transport equations by characteristics. We remark that the definition of the auxiliary functions $f_1,\ f_2,\ f_3$ defining the remainder $R=R^0+R^1+R^2+R^3$ as in \eqref{eq:R^0} and \eqref{eq:R^i} are independent of the space dimension, as well as these latter transport equations and their solution formulas \eqref{eq:charac.R^0} and \eqref{eq:charac.R^i}. The first main difference appears in the nonlocal equations \eqref{eq:nonlocal.R^0} and \eqref{eq:nonlocal.R^i}. Integrating \eqref{eq:charac.R^0} and \eqref{eq:charac.R^i} over $S^{d-1}$ and changing to spherical coordinates we obtain nonlocal operators of the following form
\[\cL_{D,d}^{K_n,\sigma_a}(u)(x)=u(x)-\int_D E_{d}^{K_n,\sigma_a}(x,\eta)u(\eta)d\eta,\]
where \[E_d^{K_n,\sigma_a}(x,\eta)=\frac{\exp{\left(-\int_{[x,\eta]}\frac{1+K_n^2\sigma_a(\xi)}{K_n}\ d\xi\right)}}{c_d K_n |x-\eta|^{d-1}}\]
for $c_d=|S^{d-1}|$. Moreover, the maximum principle stated in Proposition \ref{prop:max} applies also for $d\geq 3$, since as in \eqref{eq:kern}
\begin{align*}
\int\limits_{\R^d} E_d^{K_n,\sigma_a}(x,\eta) d \eta   &=
    \int_{S^{d-1}} \int_0^{s(x,v)} r^{d-1} E_d^{K_n,\sigma_a}(x,x-rv)\ dr \ dv\\
    & \leq \int_0^\infty\frac{\exp\left(-\frac{r}{K_n}\right)}{K_n}=1.
    \end{align*}
It is also important to notice that the supersolutions $\phi_{K_n}^1$ in \eqref{eq:phi1} and $\phi_{K_n,A}^2$ in \eqref{eq:phi2} are supersolutions for $\cL_{D,d}^{K_n,\sigma_a}$ and satisfy (up to suitable constants multiplying them)
\[\cL_{D,d}^{K_n,\sigma_a}[\phi_{K_n}^1]\geq K_n^2 \qquad\text{ and }\qquad \cL_{D,d}^{K_n,\sigma_a}[\phi_{K_n,A}^2]\geq e^{-\frac{d(x)}{AK_n}}.\]
Indeed, the construction of the supersolutions in \cite{DV25} can be carried out in the same way also in higher dimensions $d\geq 3$.\\

Therefore, assuming $f\in H^{s_1}(\partial D)$ for $s_1=\frac{9}{2}+\lfloor\frac{d}{2}\rfloor$ we can extend it to some $\tilde{f}\in H^{5+\lfloor\frac{d}{2}\rfloor}(D)$ such that $\tilde{f}|_{\partial D}=f$ in the sense of trace (cf. \cite{Taylor}). By elliptic regularity theory, this implies that $\rho_{0,0}$ solving \eqref{eq:leading order rho} belongs also to $H^{5+\lfloor\frac{d}{2}\rfloor}(D)$, and thus in particular $\rho_{0,0}\in H^{s_0}(D)$ since $s_0=3+\lfloor\frac{d}{2}\rfloor<s_1+\frac{1}{2}$ with $\|\rho_{0,0}\|_{H^{s_0}(K)}\leq C(D)\|f\|_{H^{s_1}(\partial D)} $. By Sobolev embedding we can conclude that $\rho_{0,0}\in C^{4,\frac{1}{2}}(D)$ as well as by Schauder theory $\psi_0\in C^{6,\frac{1}{2}}(D)$ solving \eqref{eq: leading order}. Moreover, also $c_0, c$ solving \eqref{eq:def.c} belong to $ C^{3,\frac{1}{2}}(D)$. Thus, the estimates of the norms of $\psi_0$, $c$ and $c_0$ holds as in \eqref{eq:norm.est.1}, \eqref{eq:norm.est.2}, \eqref{eq:norm.est.3}, \eqref{eq:norm.est.4} and \eqref{eq:norm.est.5} replacing $\|\rho_{0,0}\|_{H^2(K)}$, $\|\rho_{0,0}\|_{H^4(K)}$ and $\|f\|_{H^{5+\frac{1}{2}}(\partial D)}$ by $\|\rho_{0,0}\|_{H^{s_0-2}(K)}$, $\|\rho_{0,0}\|_{H^{s_0}(K)}$ and $\|f\|_{H^{s_1}(\partial D)}$, respectively.\\

Thus, Lemma \ref{lemma:rough estimate} and Lemma \ref{lemma:rough estimateR0} hold for any dimension $d\geq 3$ replacing $\|\rho_{0,0}\|_{H^4(K)}$ by $\|\rho_{0,0}\|_{H^{s_0}(K)}$ and $\|f\|_{H^{5+\frac{1}{2}}(\partial D)}$  by $\|f\|_{H^{s_1}(\partial D)}$. The same adaptation can be done for $d\geq 3$ for Lemma \ref{lem:estimate R01} and Proposition \ref{prop:R01}. In particular, all the estimates and geometrical arguments in the proof of Lemma \ref{lem:estimate R01} work also for higher dimensions once the kernel $E^{K_n,0}$ is replaced by $E^{K_n,0}_d$. Moreover, Proposition \ref{prop:R01} is just an application of the maximum principle.\\

Notice that all constructions and results about the one-dimensional boundary layer problem apply for the radiative transfer problem in any dimension $d\geq 3$. Thus, Proposition \ref{prop:collection} and the auxiliary Lemmata in Subsec.~\ref{subsec:proof_proposition}  
hold for any $d\geq 3$.\\

Finally, we can prove Theorem \ref{thm:elena_juan_estimates} for any dimension $d\geq 3$.
\begin{proof}[Proof of Theorem \ref{thm:elena_juan_estimates} for $d\geq3$]
    Repeat the proof as for $d=3$ replacing $\|\rho_{0,0}\|_{H^2(K)}$, $\|\rho_{0,0}\|_{H^4(K)}$ and $\|f\|_{H^{5+\frac{1}{2}}(\partial D)}$ by $\|\rho_{0,0}\|_{H^{s_0-2}(K)}$, $\|\rho_{0,0}\|_{H^{s_0}(K)}$ and $\|f\|_{H^{s_1}(\partial D)}$, respectively.
\end{proof}
\begin{rmk}
    We expect the results of Theorem \ref{thm:elena_juan_estimates} to hold also in the case of spacial dimension $d=2$. Nevertheless, the estimates and the supersolutions that we would need in order to prove such a statement are more involved due to the appearance of logarithmic terms in the Green's functions of the Laplace operator in two dimensions. For this reason, we have not included this case in our discussion.
\end{rmk}

\subsection{Derivation of the a priori estimates for the albedo operator}
\label{sub:albedo_apriori_estimates}

Using Theorem \ref{thm:elena_juan_estimates}, we are now in a position to deduce the main a priori bounds for the albedo operator. These will play a crucial role in the comparison strategy in Sec.~\ref{sec:instability}.

\begin{prop}\label{prop:estimates}
Let $K\Subset D$, let $M>0$ and let $\sigma_a\in \cK$, where $\cK$ is defined as in \eqref{eq:definition_K}. Then there exists a constant $C$, depending on $D,K,M$, and $s_0,s_1$ as in Theorem~\ref{thm:elena_juan_estimates} such that, for $K_n\in(0,1]$,
\begin{align*}
    \|(\Lambda_{\sigma_a}-\Lambda_0)f\|_{H^{-1/2}(\partial D)} \leq
    C( \|\rho_{0,0}\|_{H^{s_0}(K)}+K_n \|f\|_{H^{s_1}(\partial D)} ),\quad
    \forall f\in H^{s_1}(\partial D).
\end{align*}
\end{prop}


\begin{proof}
We refer to Sec.~\ref{sub:diffusion} for the definition of the terms in the diffusion approximation appearing throughout the proof. Given $g\in H^{1/2}(\partial D)$, we consider its harmonic extension $\tilde{g}\in H^1(D)$. We then apply the formula from Lemma~\ref{lem:albedo_formula} to $(\Lambda_{\sigma_a}-\Lambda_0)$ to conclude that
\begin{align*}
    \langle (\Lambda_{\sigma_a}-\Lambda_0)f,g \rangle = -\int_{D} \sigma_a \langle u_a \rangle \tilde{g} + \frac{1}{K_n} \int_{D} \langle v (u_a-u_0)\rangle\cdot\nabla{\tilde{g}} = (I)+(II),
\end{align*}
where $u_0$ solves \eqref{eq:ua_def} with $\sigma_a=0$, $F=0$ and $G=f$. We now estimate the two terms.

We write term $(I)$ as
\begin{align*}
    -(I) = \int_D \sigma_a \langle u_a \rangle\tilde{g} &= 
    \int_D \sigma_a \psi_{a,0} \tilde{g} + 
    K_n\int_D \sigma_a \langle \psi_{a,1} \rangle \tilde{g} + K_n^2 \int_{D} \sigma_a \langle R_a \rangle \tilde{g} \\ &=
    (J_1) + (J_2) + (J_3).
\end{align*}
For the term ($J_1$), we have that
\begin{align*}
    |(J_1)| &\leq \|\sigma_a\psi_{a,0}\|_{L^2(D)} \|\tilde{g}\|_{L^2(D)} \\ &\leq 
    \lc \|\sigma_a \psi_0\|_{L^2(D)} +
    \|\sigma_a \psi_{0,0}\|_{L^2(D)}
    \rc \|g\|_{H^{1/2}(\partial D)}.
\end{align*}
Using that $\psi_0$ satisfies \eqref{eq: leading order} and that $\psi_{0,0}=\rho_{0,0}$, we conclude that
\begin{align*}
    |(J_1)| \lesssim
    \|\sigma_a \rho_{0,0}\|_{L^2(D)} \|g\|_{H^{1/2}(\partial D)}\lesssim \|\rho_{0,0}\|_{L^2(K)} \|g\|_{H^{1/2}(\partial D)}.
\end{align*}
For the term ($J_2$), we have that $\langle\psi_{a,1}\rangle = c_a$, and therefore
\begin{align*}
    |(J_2)| \lesssim
    K_n \|c_a\|_{L^2(D)}
    \|\tilde{g}\|_{L^2(D)} \lesssim
    K_n \|c_a\|_{L^2(D)}
    \|g\|_{H^{1/2}(\partial D)}.
\end{align*}
Moreover, by \eqref{eq:def.c}, $c_a= c_a(x)= c(x)+c_0(x)$ satisfies
\begin{align*}
    \begin{cases}
        -C_d\Delta{c_a} + \sigma_a c_a = 0 &
        \text{in $D$},\\
        c_a|_{\partial D} = \frac{W_1}{W_2} \partial_n \psi_{a,0},
    \end{cases}
\end{align*}
and therefore
\begin{align*}
    |(J_2)| &\lesssim
    K_n \|\partial_n \psi_{a,0}\|_{H^{1/2}(\partial D)} \|g\|_{H^{1/2}(\partial D)} \\ &\lesssim
    K_n \|\psi_{a,0}\|_{H^{2}(D)} \|g\|_{H^{1/2}(\partial D)} \\ &\lesssim
    K_n \|f\|_{H^{s_1}(\partial D)}
    \|g\|_{H^{1/2}(\partial D)}.
\end{align*}
For the term $(J_3)$, using \eqref{eq:elenajuan_first_estimate}, we deduce that
\begin{align*}
    |(J_3)| \lesssim
    \lc \|\rho_{0,0}\|_{H^{s_0}(K)}+
    K_n\|f\|_{H^{s_1}(D)} \rc \|g\|_{H^{1/2}(\partial D)}.
\end{align*}
This concludes the estimate for term (I).

We now estimate term (II). We use the first order diffusion approximation for $u=u_a-u_0$ from Sec.~\ref{sub:diffusion} to get that
\begin{align*}
    (II) &= \frac{1}{K_n}\int_D\langle vu \rangle\cdot\nabla{\tilde{g}} \\ &=
    \frac{1}{K_n} \int_D \langle v\psi_{0} \rangle\cdot\nabla{\tilde{g}} +
    \int_D \langle v\psi_{1} \rangle\cdot\nabla{\tilde{g}} +
    K_n\int_D \langle vR \rangle\cdot\nabla{\tilde{g}} \\ &=
    (J_4)+(J_5)+(J_6).
\end{align*}
The term ($J_4$) vanishes, as $\psi_0=\psi_{a,0}-\psi_{0,0}$ is isotropic. For term ($J_5$), we get that $\psi_1=-v\cdot\nabla{\psi_0}+c$ where $c=c(x)$ is isotropic, and therefore
\begin{align*}
    \langle v\psi_1 \rangle = -C_d \nabla{\psi_0}.
\end{align*}
Using that $\psi_0$ satisfies \eqref{eq: leading order}, we conclude that
\begin{align*}
    |(J_5)| &= C_d\bigg| \int_D \nabla{\psi_0}\cdot\nabla{\tilde{g}} \bigg|
    \\ &\lesssim
    \|\nabla{\psi_0}\|_{L^2(D)} \|\nabla{\tilde{g}}\|_{L^2(D)} \\ &\lesssim
    \|\sigma_a \rho_{0,0}\|_{L^2(D)} \|g\|_{H^{1/2}(\partial D)} \\ &\lesssim
    \|\rho_{0,0}\|_{L^2(K)} \|g\|_{H^{1/2}(\partial D)},
\end{align*}
which concludes the estimates for term ($J_5$). Finally, for term ($J_6$) we use \eqref{eq:elenajuan_second_estimate} to conclude that
\begin{align*}
    |(J_6)| &\leq K_n \int_D |\langle vR \rangle| |\nabla{\tilde{g}}| \\ &\lesssim
    K_n \|\langle vR\rangle\|_{L^{\infty}(D)} \|\nabla{\tilde{g}}\|_{L^2(D)} \\ &\lesssim
    ( \|\rho_{0,0}\|_{H^{s_0}(K)}+K_n\|f\|_{H^{s_1}(\partial D)} )\|g\|_{H^{1/2}(\partial D)}.
\end{align*}
This concludes the proof.
\end{proof}
\subsection{Conclusion of the proof of Proposition \ref{prop:collection} -- proof of the estimates \eqref{eq:prop.collection2} and \eqref{eq:prop.collection3}}
\label{subsec:proof_proposition}
In this subsection we aim to complete the proof of Proposition \ref{prop:collection}, in particular we will show \eqref{eq:prop.collection2} and \eqref{eq:prop.collection3}. This requires several auxiliary Lemmata for the solution $u$ to the equation
\begin{equation}\label{eq:u1dim.no.p}
    \mathcal{L}[u](y)=S(y)
\end{equation}
where $\cL$ is defined in \eqref{eq:def.L} and with $S(y)$ exponentially decaying and with $|u(y)|\leq C e^{-\frac{y}{2}}$. 
Notice that \eqref{eq:u1dim.no.p} corresponds to equation \eqref{eq:u1dim} neglecting the dependence on the point $p\in\partial D$. Indeed, it is enough to prove \eqref{eq:prop.collection2} and \eqref{eq:prop.collection3} for the solution $u$ to \eqref{eq:u1dim.no.p}.
\begin{lem}\label{lem:existence.vc}
Let $\eps>0$, $\psi_\eps\in C^\infty(\bR_+)$ such that $0\leq \psi_\eps\leq 1$ and \[ \psi_\eps(y)=\begin{cases}
   0&\text{ if }y\leq \frac{\eps}{2},\\1&\text{ if }y\geq \eps.
\end{cases}\]
    Let also $c\in\bR$. Then there exists a unique solution $v_\eps^c\in C(\bR_+)$ solving
    \[\cL[v_\eps^c](y)=\psi_\eps(y)\left(S'(y)+cE(y)\right)=:S_\eps(y),\]
    where $E$ is the normalized exponential integral. Moreover, for $\alpha\in(0,1)$ small enough and strcitly positive there exist $ C_0,\ C>0$ (independent of $\epsilon$) such that
    \begin{equation}\label{eq:lem.exist.vc.1}
        \left|v_\eps^c(y)\right|\leq C_0+Cy^{-\alpha}.
    \end{equation}
    Hence, there exists $v^c\in C(\bR)$ with $\left|v^c(y)\right|\leq C_0+Cy^{-\alpha}$ such that $\cL[v^c](y)=S'(y)+cE(y)$.
\end{lem}
\begin{proof}
    Since by definition
    \[\left|S_\eps(y)\right|\leq \psi_\eps(y)\frac{1}{2}\int_0^1 \frac{1}{x}e^{-\frac{y}{x}}\ dx+\psi_\eps(y)|c|E(y)\leq \psi_\eps(y)(1+|c|)E(y)\leq \ln\left(1+\frac{2}{\eps}\right)\frac{(1+|c|)}{2}e^{-y}, \]
    Theorem 3.1 in \cite{DV25} implies the existence of a unique bounded and continuous solution $v_\eps^c$. Moreover, the results (i)-(iii) of Proposition \ref{prop:collection} hold also in this case in which $S_\eps$ is the source. Thus, as $y \rightarrow \infty$ a limit exists and  $v_\eps^c(y)\to v_\eps^c(\infty)$ exponentially fast. Next we will show that $v_\eps^c$ is uniformly bounded by the locally integrable function given in \eqref{eq:lem.exist.vc.1}. 

    First of all we show the following claim: There exists $L_0>3$ such that $2e^2 E(y+1)-E(y-1)>0$ for all $y>L_0$. In order to prove this claim we notice that the following limit holds
    \[\lim\limits_{y\to\infty} \frac{E(y-1)}{E(y+1)}=e^2.\] This can be proved using that $\frac{e^{-x}}{4}\ln\left(1+\frac{2}{x}\right)\leq E(x)\leq \frac{e^{-x}}{2}\ln\left(1+\frac{1}{x}\right)$. There exists thus some $L_0>3$ such that $\frac{E(y-1)}{E(y+1)}>e^{2}+1=\kappa$ for all $y>L_0$. Moreover, since the exponential integral converges to zero we see that the function $f(y)=\kappa E(y+1)-E(y-1)$ satisfies $f(L_0)>0$ as well as $\lim\limits_{y\to\infty} f(y)=0$. Finally, we compute
    \[f'(y)=-2e^2\frac{e^{-(y+1)}}{y+1}+\frac{e^{-y+1}}{y-1}=e^{y-1}\frac{3-y}{(y+1)(y-1)}<0\]
    for $y>L_0>3$.

    Let us now define $m_1=\inf\limits_{[0,1]}\int_y^\infty E(\eta)\ d\eta$ and $m_{L_0}=\inf\limits_{[1,L_0]}\int_y^\infty E(\eta)\ d\eta$. For $\alpha\in\left(0,\frac{1}{4}\right)$ we define
    \begin{equation}\label{eq:C0}
        C_0=C_0(\alpha):=\frac{\Arrowvert E\Arrowvert_{L^2(\bR)}^{\frac{1}{2}}}{m_1\sqrt{1-2\alpha}}+\frac{\Arrowvert E\Arrowvert_{L^2(\bR)}^{\frac{1}{2}}}{m_{L_0}\sqrt{1-2\alpha}}+\frac{2e^2}{1-\alpha}.
    \end{equation}
    The continuous function $w(y)=C_0+y^{-\alpha}\mathbb{1}_{\{y<1\}}+\mathbb{1}_{\{y\geq 1\}}$ satisfies
    \[\cL(w)\geq \begin{cases}
        y^{-\alpha}\int_y^\infty E(\eta)\ d\eta & \text{ if }\ y<1,\\ 0&\text{ if }\ y\geq 1.
    \end{cases}\]
    Indeed, for $y\in[0,1)$ we have, using H\"older's inequality and \eqref{eq:C0}
    \[\begin{split}
        \cL(w)(y)=& C_0\int_y^\infty E(\eta)\ d\eta+y^{-\alpha} \int_{-\infty}^0 E(\eta-y)\ d\eta+\int_0^1 E(\eta-y)\left(y^{-\alpha}-\eta^{-\alpha}\right)\ d\eta\\&+\left(y^{-\alpha}-1\right)\int_1^\infty E(\eta-y) d \eta\\\geq & C_0\int_y^\infty E(\eta)\ d\eta+y^{-\alpha} \int_{y}^\infty E(\eta)\ d\eta-\int_0^yE(\eta-y)\left(\eta^{-\alpha}-y^{-\alpha}\right)\ d\eta\\
        \geq& y^{-\alpha} \int_{y}^\infty E(\eta)\ d\eta+C_0\int_y^\infty E(\eta)\ d\eta-\frac{\Arrowvert E\Arrowvert_{L^2(\bR)}^{\frac{1}{2}}}{\sqrt{1-2\alpha}}y^{\frac{1-2\alpha}{2}}\\
        \geq & y^{-\alpha} \int_{y}^\infty E(\eta)\ d\eta.
    \end{split}\]
    Moreover, in a similar way we compute for $y\in[1,L_0]$ using \eqref{eq:C0}
     \[\begin{split}
        \cL(w)(y)=& C_0\int_y^\infty E(\eta)\ d\eta+ \int_{-\infty}^1 E(\eta-y)\ d\eta-\int_0^1 E(\eta-y)\eta^{-\alpha}\ d\eta\\
        \geq&  \int_{-\infty}^1 E(\eta-y)\ d\eta+C_0\int_y^\infty E(\eta)\ d\eta-\frac{\Arrowvert E\Arrowvert_{L^2(\bR)}^{\frac{1}{2}}}{\sqrt{1-2\alpha}}\\
        \geq & \int_{-\infty}^1 E(\eta-y)\ d\eta>0.
    \end{split}\]
    Finally, for $y>L_0$ using that $E$ is even and monotonously decreasing in $\bR_+$ we estimate using \eqref{eq:C0}
     \[\begin{split}
        \cL(w)(y)=& C_0\int_y^\infty E(\eta)\ d\eta+ \int_{-\infty}^1 E(\eta-y)\ d\eta-\int_0^1 E(\eta-y)\eta^{-\alpha}\ d\eta\\
        \geq&  \int_{-\infty}^1 E(\eta-y)\ d\eta+C_0\int_{-\infty}^0 E(\eta-y)\ d\eta-\int_0^1 E(\eta-y)\eta^{-\alpha}\ d\eta\\
        \geq & \int_{-\infty}^1 E(\eta-y)\ d\eta+C_0\int_{-1}^0 E(\eta-y)\ d\eta-E(y-1)\int_0^1 \eta^{-\alpha}\ d\eta\\
        \geq & \int_{-\infty}^1 E(\eta-y)\ d\eta+C_0E(y+1)-\frac{E(y-1)}{1-\alpha}\ d\eta\geq 0.
    \end{split}\]
    We will see then that for two suitable constants $C_1,\ A>0$ the function $C_1+Aw$ satisfies $\cL(C_1+Aw)\geq S_\eps(y)$ for all $\eps>0$. To this end, it is enough to show that the operator acting on $C_1+Aw$ is larger than $(1+|c|)E(y)$. Since $y^{\alpha} E(y)$ is bounded if $\alpha>0$ and converges to zero if $y\to 0$, for $A=\frac{(1+|c|)}{m_1} \sup\limits_{[0,1]} y^{\alpha} E(y)$ we see for $y\in(0,1)$ 
    \[\cL(Aw)(y)-S_\eps(y)\geq A m_1 y^{-\alpha}-(1+|c|)E(y)\geq 0.\] Notice that $\alpha$ has to be strcitly positive in order to have $A<\infty$.
    Moreover, for 
    \[C_1=\frac{(|c|+1)}{m_{L_0}}\max\limits_{[1,L_0]}E(y)+e^{2e}(1+|c|)\]
    we have on the one hand for $y\in[1,L_0]$
    \[\cL(C_1)-S_\eps(y)\geq C_1 \int_y^\infty E(\eta)\ d\eta-(1-|c|)E(y)\geq C_1 m_1-(1-|c|)E(y)> 0\]
    and on the other hand for $y>L_0$ we see similarly as we did for $w$
    \[\cL(C_1)-S_\eps(y)\geq C_1 \int_{-\infty}^0 E(\eta-y)\ d\eta-(1-|c|)E(y)\geq C_1 E(y+1)-(1+|c|)E(y-1)\geq 0.\]
    Thus, by linearity $\cL(C_1+Aw)\geq S_\eps(y)$ for all $\eps>0$. 

    As we proved in Theorem 3.1 in \cite{DV25}, the function $f_\delta(y)=\delta(1+y)$ for $y\geq 0$ satisfies $\cL(f_\delta)\geq 0$. Thus, using that
    \[\cL\left[C_1+Aw+f_\delta\pm v_\eps^c\right](y)\geq 0\] and that $\lim\limits_{y\to\infty} C_1+Aw+f_\delta-v_\eps^c$ exists and it is infinity, the maximum principle of Lemma 3.1 in \cite{DV25} implies that on $\bR_+$
    \[|v_\eps^c|\leq C_1+Aw+f_\delta.\]
   Hence, letting $\delta\to 0$ we conclude that 
   \[|v_\eps^c(y)|\leq C_1+Aw(y)\leq (C_1+AC_0)+AC_0 y^{-\alpha}.\]

   In order to prove the existence of the continuous function $v^c$ as in the statement of Lemma \ref{lem:existence.vc}, we show that for a subsequence $\eps_j$ the solutions $v_{\eps_j}^c$ converge uniformly in every compact set $[a,b]\subset\bR_+$ to such a function $v^c$. First of all, by construction $S_\eps(y)\to S'(y)+cE(y)$ both pointwise in $\bR_+$ and uniformly in every compact set $[a,b]\subset (0,\infty)$. Thus, we need to consider the continuous function $f_\eps^c(y)=\int_0^\infty E(y-\eta)v_\eps^c(\eta)$. We show now that $f_\eps^c$ is a sequence of uniformly bounded H\"older continuous functions. Then, the Theorem of Arzela Ascoli and a diagonal argument will imply the uniform convergence of $f_{\eps_j}$ to some continuous function $f^c$ in every compact set $[a,b]\subset \bR_+$.

   Because of \eqref{eq:lem.exist.vc.1} the sequence $f_\eps^c$ is uniformly bounded, indeed
   \[\left|f_\eps^c(y)\right|\leq (C_0+C)\int_0^1 E(y-\eta) \eta^{-\alpha} d\eta+(C_0+C)\int_1^\infty E(\eta-y)\ d\eta\leq (C_0+C)\left(\frac{\Arrowvert E\Arrowvert_{L^2(\bR)}^\frac{1}{2}}{\sqrt{1-2\alpha}}+1\right)<\infty.\]
  
   Thus, it is enough to show that $\left|f_\eps^c(x)-f_\eps^c(y)\right|\leq C|x-y|^{\delta}$ for some $\delta\in(0,1)$ and for $0<y-x<1$. In the following we will use that
   \[\int_0^\infty|E(x-\eta)-E(y-\eta)|\ d\eta= \frac{e^{-x}-e^{-y}}{2}+2\left(1-e^{-\frac{|y-x|}{2}}\right)+2|x-y|E(|x-y|)+xE(x)-yE(y),\]
   as computed in Lemma 3.3 in \cite{DV25}, and that
   \[\begin{split}
       |xE(x)-yE(y)|\leq& x\left(E(x)-E(y)\right)+|x-y|E(y)\leq \frac{e^{-x}-e^{-y}}{2}+|x-y|E(y-x)\\\leq& \frac{|x-y|}{2}+C(\alpha)|x-y|^{2\alpha},
   \end{split}\]
   which is obtained by the definition of $E(x)-E(y)=\frac{1}{2}\int_x^y \frac{e^{-t}}{t}\ dt\leq \frac{e^{-x}-e^{-y}}{2x}$, by the monotonicity of $E$ as well as by the boundedness of the function $z^{1-2\alpha}E(z)$ for $z\in [0,1]$ and $\alpha\in\left(0,\frac{1}{4}\right)$. Hence,
   \[\begin{split}
       \left|f_\eps^c(x)-f_\eps^c(y)\right|\leq& \int_0^{\frac{|x-y|}{2}} \left|E(x-\eta)-E(y-\eta)\right| |v_\eps^c(\eta)|\ d\eta+\int_{\frac{|x-y|}{2}}^\infty \left|E(x-\eta)-E(y-\eta)\right| |v_\eps^c(\eta)|\ d\eta\\
   \leq& 2\Arrowvert E\Arrowvert_{L^2(\bR)}^{\frac{1}{2}}\left(\int_0^{\frac{|x-y|}{2}}\left(C_0+C\eta^{-\alpha}\right)^2\ d\eta\right)^{\frac{1}{2}} \\&+\left(C_0+2^\alpha C|x-y|^{-\alpha}\right)\int_0^\infty \left|E(x-\eta)-E(y-\eta)\right|\ d\eta\\
   \leq& C(E,\alpha, C_0,C)\left(|x-y|+|x-y|^{\frac{1-2\alpha}{2}}\right)\\&+C(C_0,C,\alpha, E)\left(|x-y|+|x-y|^\alpha+|x-y|^{2\alpha}+|x-y|^{1-\alpha}\right).
   \end{split}\]
  Let $\delta=\min\left\{1,\alpha,1-\alpha,2\alpha, \frac{1-2\alpha}{2}\right\}<\frac{1}{4}$ since $0<\alpha<\frac{1}{4}$. We conclude that the sequence $f_\eps^c$ is uniformly $\delta$-H\"older continuous.

  We can conclude now the proof of Lemma \ref{lem:existence.vc}. We have just proved the existence of a limit function $f^c$ and of a subsequence $\eps_j$ such that
  \[v_{\eps_j}^c(y)\to f^c(y)+S'(y)+cE(y) \text{ pointwise in $\bR_+$ and uniformly in every compact set $[a,b]\subset \bR_+$.}\]
  Thus, we set $v^c(y)=f^c(y)+S'(y)+cE(y)$. Moreover, the pointwise convergence of $v_{\eps_{j}}$ and the uniform estimate \eqref{eq:lem.exist.vc.1} imply 
  \[|v^c(y)|\leq C_0+Cy^{-\alpha}.\]
  Finally, since by \eqref{eq:lem.exist.vc.1} the sequence $E(y-\eta)v^c_{\eps_j}(\eta)$ is uniformly bounded by an integrable function and converges to $E(y-\eta)v^c(\eta)$ pointwise for $\eta\in\R_+$,
  we can conclude by Lebesgue dominated theorem that
  \[f^c(y)=\int_0^\infty E(\eta-y)v^c(\eta)\ d\eta.\] This yields 
  \[\cL(v^c)(y)=S'(y)+cE(y).\]
\end{proof}

\begin{lem}\label{lem:uniform.convergence}
    There exists a constant $C>0$ independent of $\eps$ such that 
    \[\left|v^c_\eps(y)-v^c_\eps(\infty)\right|\leq C e^{-\frac{y}{2}}.\]
\end{lem}
\begin{proof}
    This is based on the proof of Lemma 3.4 in \cite{DV25}. First of all, since $|S_\eps(y)|\leq S'(y)+cE(y)\leq (1+c)E(y)\in L^1(\bR_+)\cap L^2(\bR_+)$ and since $ v_\eps^c$ is uniformly bounded by the locally integrable function $C_0+C_1 y^{-\alpha}$ we see that if we can bound $v_\eps^c(0)$ uniformly in $\eps$, then the proof of Lemma 3.4 in \cite{DV25} implies the claim. Using that $S_\eps(0)=0$ for all $\eps$ we see that
    $|v_\eps^c(0)|=\left|\int_0^\infty E(\eta) v_\eps^c(\eta) \ d\eta\right|\leq \int_0^\infty E(\eta) \left(C_0+C_1\eta^{-\alpha}\right)\ d\eta<\infty.$
\end{proof}
\begin{lem}\label{lem:convergence.vc}
    There exists $v^c_\infty$ such that $v^c$ converges to $v_\infty^c$ as $y\to\infty$ and such that
    \begin{equation}\label{eq:converg.vc}
        \left|v^c(y)-v^c_\infty\right|\leq C e^{-\frac{y}{2}} \qquad\text{ for }y\geq 1.
    \end{equation}
\end{lem}
\begin{proof}
    As we have proved in Lemma \ref{lem:existence.vc}, there exists a subsequence $v_{\eps_j}^c$ which converges uniformly in every compact subset of $\bR_+$ to $v^c$. Moreover, by the uniform boundedness of the sequence $v_{\eps_j}^c(\infty)$, a further subsequence (which we call $v_{\eps_j}^c$ for simplicity) has the property that $v_{\eps_j}^c(\infty)\to v^c_\infty$ for some $v^c_\infty$. Notice that by Lemma \ref{lem:uniform.convergence} $\left|v^c_{\eps_j}(y)-v^c_{\eps_j}(\infty)\right|\leq C e^{-\frac{y}{2}}$. Let hence $\delta>0$. Then there exists $R_0(\delta)>0$ such that $Ce^{-\frac{R_0}{2}}<\frac{\delta}{3}$. Moreover, there exists $J_0(\delta)$ and for $y>R_0$ 
    there exists $J_1(\delta,y)$ such that $\left|v^c_{\eps_j}(y)-v^c(y)\right|<\frac{\delta}{3}$ for $j>J_1$ and  $\left|v^c_{\eps_j}(\infty)-v^c_\infty\right|<\frac{\delta}{3}$ for $j>J_0$. Thus, for all $y>R_0$
    \[\left|v^c(y)-v^c_\infty\right|<\delta.\]
    We now show \eqref{eq:converg.vc} 
    To this end, we shall consider $y>1$. Then, there exists $R>0$ such that $y\in[R,R+1]$. By the convergence of $v_{\eps_j}^c(\infty)$ we know that there exists $J_0(R)$ such that $\left|v^c_{\eps_j}(\infty)-v^c_\infty\right|<Ce^{-\frac{R+1}{2}}$ for all $j>J_0$. Moreover, the uniform convergence of $v_{\eps_j}^c$ in $[R,R+1]$ implies also that there exists $J_1(R)$ such that $\sup\limits_{[R,R+1]}\left|v^c_{\eps_j}(y)-v^c(y)\right|<Ce^{-\frac{R+1}{2}}$ for all $j>J_1$. Therefore, applying the result of Lemma \ref{lem:uniform.convergence} we conclude that
    \[\left|v^c(y)-v^c_\infty\right|<3 C e^{-\frac{y}{2}}.\]
\end{proof}

The following Lemma provides a relation between $u_{\infty}$ and the properties of the source term.
\begin{lem}\label{lem:convergence.zero}
Let $F:\bR\to\bR$ be integrable and let $xF\in L^1(\bR)$. Let $u$ be the bounded solution to
\[u(y)-\int_\bR E(y-\eta)u(\eta) \ d\eta=F(y)\]
satisfying $u(y)\equiv 0$ if $y<0$. Assume also that there exists $u_\infty\in\bR$ and a function $h$ with exponential decay such that $u=u_\infty+h(y)$.
If $\int_\bR F(y) \ dy=0$ and $\int_\bR yF(y) \ dy=0$, then $u_\infty=0$.
\end{lem}
\begin{proof}
    By assumption we know
    \begin{equation}\label{eq:lem.conv.zero.1}
        \lim\limits_{R\to\infty}  \left(\int_0^R y \ u(y) \ dy-\int_{-R}^Ry\int_\bR E(y-\eta)u(\eta) \ d\eta \ dy \right) =\lim\limits_{R\to\infty}\int_{-R}^R yF(y)\ dy=0.
    \end{equation}
    Moreover, 
    \[\int_0^R yu(y)\ dy=u_\infty\int_0^R y\ dy+\int_0^R yh(y)\ dy\] and 
    \begin{equation}
        \begin{split}
          \int_{-R}^Ry\int_\bR E(y-\eta)u(\eta)& \ d\eta \ dy= \int_0^\infty u(\eta)\int_{-R}^R yE(y-\eta)\ dy\ d\eta\\=&u_\infty\int_0^\infty\int_{-R}^R yE(y-\eta)\ dy\ d\eta+  \int_0^\infty h(\eta)\int_{-R}^R yE(y-\eta)\ dy\ d\eta\\
          =&u_\infty\int_0^\infty\int_{-R}^R (y-\eta)E(y-\eta)\ dy\ d\eta+u_\infty\int_0^\infty\eta\int_{-R}^RE(y-\eta)\ dy\ d\eta\\&+  \int_0^\infty h(\eta)\int_{-R}^R (y-\eta)E(y-\eta)\ dy\ d\eta+\int_0^\infty \eta h(\eta)\int_{-R}^R E(y-\eta)\ dy\ d\eta,\\
        \end{split}
    \end{equation}
    where the first equality is due to the fact that $\int_\bR E(y-\eta)u(\eta)\ d\eta$ is bounded and that the integral on the variable $y$ is over a compact interval.
    The exponential decay of $h$, the symmetry of the exponential integral as well as the fact that $\int_\bR E(z)\ dz=1$ imply that
    \[ \int_0^\infty h(\eta)\int_{-R}^R (y-\eta)E(y-\eta)\ dy\ d\eta\to 0 \;\text{ and }\;\int_0^\infty \eta h(\eta)\int_{-R}^R E(y-\eta)\ dy\ d\eta\to \int_0^\infty \eta h(\eta)\ d\eta\]
    as $R\to\infty$. Thus, equation \eqref{eq:lem.conv.zero.1} and the properties of $E$ imply that
    \begin{equation}\label{eq:lem.conv.zero.2}
        \begin{split}
          0=&\lim\limits_{R\to\infty}u_\infty\left[\int_0^R y\ dy-\int_0^\infty\int_{-R}^R (y-\eta)E(y-\eta)\ dy\ d\eta-\int_0^\infty\eta\int_{-R}^RE(y-\eta)\ dy\ d\eta\right]\\
          =&\lim\limits_{R\to\infty}u_\infty\left[\int_0^R y\left(\int_{-\infty}^{-R} E(\eta-y)\ d\eta+\int_R^\infty E(\eta-y)\ d\eta\right)\ dy\right.\\&\qquad\left.-\int_0^\infty\int_{-R}^R (y-\eta)E(y-\eta)\ dy\ d\eta-\int_R^\infty\eta\int_{-R}^RE(y-\eta)\ dy\ d\eta\right].\\
        \end{split}
    \end{equation}
    We recall that as seen in Proposition 2.1 in \cite{DV25} the integrals
    \[0\leq \int_{|x|}^\infty E(z) \ dz\leq \frac{e^{-|x|}}{2} \qquad\text{ and }\qquad 0\leq\int_{|x|}^\infty zE(z) \ dz\leq \frac{(|x|+1)e^{-|x|}}{4} \] have exponential decay and satisfy $\int_0^\infty E(z)=\frac{1}{2}$ and $\int_0^\infty zE(z)=\frac{1}{4}$. Moreover, it can be easily seen that $\int_0^\infty z^2E(z)\ dz=\frac{1}{3}$. Thus, with suitable change of coordinates we compute
    \begin{equation}
      \int_0^R y\int_{-\infty}^{-R} E(\eta-y)\ d\eta\ dy=\int_0^R y\int_{R+y}^\infty E(\eta)\ d\eta\ dy=\int_R^{2R} (y-R)\int_y^\infty E(z)\ dz\ dy \to 0
    \end{equation}
    as $R\to\infty$,
     \begin{equation}
     \begin{split}
          \int_0^R y\int_{R}^{\infty} E(\eta-y)\ d\eta\ dy=&\int_0^R y\int_{R-y}^\infty E(\eta)\ d\eta\ dy=\int_0^{R} (R-y)\int_y^\infty E(z)\ dz\ dy\\
          =&\int_0^\infty (R-y)\int_y^\infty E(z)\ dz\ dy+\int_R^\infty  (y-R)\int_y E(z)\ dz\ dy\\
          =& \int_0^\infty E(z)\left(Rz-\frac{z^2}{2}\right)\ dz+\int_R^\infty  (y-R)\int_y E(z)\ dz\ dy\\
          =&\frac{R}{4}-\frac{1}{6}+\int_R^\infty  (y-R)\int_y E(z)\ dz\ dy,
     \end{split}
    \end{equation}
    where the last integral converges to zero by the exponential decay of the exponential integral.
    Moreover, in a similar way we compute
     \begin{equation}
     \begin{split}
         \int_R^\infty\eta\int_{-R}^RE(y-\eta)\ dy\ d\eta=&\int_R^\infty\eta\int_{-(R+\eta)}^{-(\eta-R)}E(z)\ dz\ d\eta\\ =&\int_R^\infty\eta \left[\int_{\eta-R} E(z)\ dz-\int_{R+\eta}^\infty E(z)\ dz\right]\ d\eta\\
         =&\int_0^\infty (\eta+R)\int_\eta^\infty E(z) \ dz\ d\eta-\int_{2R}^\infty (\eta-R)\int_\eta^\infty E(z)\ dz\ d\eta\\
         =&\frac{1}{6}+\frac{R}{4}-\int_{2R}^\infty (\eta-R)\int_\eta^\infty E(z)\ dz\ d\eta
     \end{split}
    \end{equation}
    and 
    \begin{equation}
     \begin{split}
        \int_0^\infty\int_{-R}^R (y-\eta)E(y-\eta)\ dy\ d\eta=&\int_0^R\int_{-(R+\eta)}^{-(R-\eta)}zE(z)\ dz\ d\eta+\int_R^\infty \int_{-(R+\eta)}^{R-\eta} zE(z)\ dz\ d\eta\\
        =&\int_0^R\left[\int_{(R+\eta)}^{\infty}zE(z)\ dz-\int_{(R-\eta)}^{\infty}zE(z)\ dz\right] d\eta\\+&\int_R^\infty \left[\int_{R+\eta}^{\infty} zE(z)\ dz-\int_{\eta-R}^{\infty} zE(z)\ dz\right] d\eta\\
        =&2\int_R^\infty \int_\eta^\infty zE(z)\ dz\ d\eta-2\int_0^\infty \int_\eta^\infty zE(z) \ dz\ d\eta\\
        =&-\frac{2}{3}+2\int_R^\infty \int_\eta^\infty zE(z)\ dz\ d\eta.
     \end{split}
    \end{equation}
    Thus, from  \eqref{eq:lem.conv.zero.2} we obtain that
    \[0=\lim\limits_{R\to\infty} u_\infty\left[2\int_R^{\infty} (y-R)\int_y^\infty E(z)\ dz\ dy-2\int_R^\infty \int_\eta^\infty zE(z)\ dz\ d\eta-\frac{2}{6}+\frac{2}{3}\right]=\frac{u_\infty}{3}.\]
    Hence, we conclude that $u_\infty =0$.
\end{proof}
\begin{lem}\label{lem:vc.zero}
    There exists $c_0\in\bR$ such that $v^{c_0}(y)\to 0$ as $y\to\infty$.
\end{lem}
\begin{proof}
    We know that $v^c$ solves
    \[v^c(y)-\int_\bR E(y-\eta)v^c(\eta)\ d\eta=\mathbb{1}_{y\geq 0}\left(S'(y)+cE(y)\right)-\mathbb{1}_{y< 0}\int_0^\infty E(y-\eta)v^c(\eta)\ d\eta=:F(y,c)\]
and $v^c\equiv 0$ on $y<0$. It is not difficult to see that by definition $\int_\bR F(y,c)\ dy=0$ for all $c\in\bR$. Moreover, Lemma \ref{lem:existence.vc} and the properties of the exponential integral imply that $F\in L^1(\bR)$ as well as $yF(y,c)\in L^1(\bR)$ for all $c\in\bR$. Let us define $\psi(c)=\int_\bR yF(y,c)\ dy$. We claim that $\psi$ is Lipschitz continuous. Indeed, let $c_1<c_2$, then by linearity $w=\frac{v^{c_2}-v^{c_1}}{c_2-c_1}$ solves $\cL(w)(y)=E(y)$ and repeating the maximum principle argument used in Lemma \ref{lem:existence.vc} we can conclude that $|w(y)|\leq C_0+C_1y^{-\alpha}$ for constants $C_0,\ C_1$ independent of $c_1,\ c_2$. Thus, $ \left|v^{c_1}(y)-v^{c_2}(y)\right|\leq |c_1-c_2| \left(C_0+C_1y^{-\alpha}\right)$. This implies
\[\left|\psi(c_1)-\psi(c_2)\right|\leq |c_1-c_2|\left(\int_0^\infty yE(y)\ dy+\left|\int_{-\infty}^0 y\int_0^\infty E(y-\eta) \left(C_0+C_1\eta^{-\alpha}\right)\ d\eta\ dy\right|\right).\]
Moreover, there exists $c_-<0<c_+$ such that $S'(y)+c_+E(y)>0$ and $S'(y)+c_-E(y)<0$. Thus, Lemma 3.1 in \cite{DV25} imply that $v^{c_-}\leq0$ and that $v^{c_+}\geq0$. This yield that $\psi(c_-)<0<\psi(c_+)$. Hence, there exists $c_0\in(c_-,c_+)$ such that $\psi(c_0)=0$. Since $v^{c_0}$ satisfies the assumption of Lemma \ref{lem:convergence.zero} we can conclude that $v^{c_0}(y)\to 0$ as $y\to\infty$ with exponential rate.
\end{proof}
\begin{lem}\label{lem:u.derivative}
  Let $u$ be the solution to \eqref{eq:u1dim.no.p}. Then $u$ has an integrable derivative which satisfies \[\left| u'(y)\right|\leq C\frac{e^{-\frac{y}{4}}}{y^{\alpha}}.\]
\end{lem}
\begin{proof}
    Let $w=-\int_y^\infty v^{c_0}(\eta)\ d\eta$ for $v^{c_0}$ as in Lemma \ref{lem:vc.zero}. We will show that $u=u(0)-w(0)+w(y)$. This implies the lemma, since $|v^{c_0}(y)|\leq C y^{-\alpha}= C e^{\frac{1}{4}} y^{-\alpha}e^{-\frac{y}{4}}$ for $y\leq 1$ as well as $|v^{c_0}(y)|\leq C e^{-\frac{y}{2}}\leq C y^{-\alpha} e^{-\frac{y}{4}}$ for $y>1$. In order to simplify the notation we denote $v=v^{c_0}$.

    On the one hand we have by linearity that $\cL(u+A)=S(y)+A\int_{y}^\infty E(\eta)\ d\eta$. On the other hand we compute
    \[v(y)-\int_0^\infty E(y-\eta)v(\eta)\ d\eta=S'(y)+c_0 E(y).\]
    Thus, integrating over the variable $ y$, using $w'=v$ and integrating by parts we compute
    \[\begin{split}
        \int_y^\infty& v(\xi)\ d\xi-\int_y^\infty \int_0^\infty E(\xi-\eta)v(\eta)\ d\eta\ d\xi\\=&-w(y)+\int_y^\infty E(\xi)w(0)-\int_y^\infty \int_0^\infty \partial_{\xi}E(\xi-\eta)w(\eta)\ d\eta\ d\xi\\
        =&-w(y)+w(0)\int_y^\infty E(z)\ dz+\int_0^\infty E(y-\eta)w(\eta)\ d\eta\\
        =&-\cL[w](y)+w(0)\int_y^\infty E(z)\ dz.
    \end{split}\]
    Notice that since all the considered integrands are integrable we can change the order of integration. Moreover,
    \[\int_y^\infty \left(S'(z)+c_0E(z)\right)\ dz=-S(y)+c_0\int_y^\infty E(z)\ dz.\]
    Thus, we conclude
    \[\cL[w](y)=S(y)+(w(0)-c_0)\int_y^\infty E(z)\ dz,\]
    which implies 
    \[\cL[u-w+(w(0)-c_0)](y)=0.\]
    Theorem 3.1 in \cite{DV25} implies that $u=w-(w(0)-c_0)$ and therefore also that $c_0=u(0)$.
\end{proof}

We can finally prove \eqref{eq:prop.collection2} of Proposition \ref{prop:collection}.
\begin{proof}[Proof of \eqref{eq:prop.collection2} and \eqref{eq:prop.collection3}]
The result of Lemma \ref{lem:u.derivative} implies
\begin{align*}
    \left| u(y_1;p)-u(y_2;p)\right|
&\leq C(\alpha)\left|\int_{y_1}^{y_2}\frac{e^{-\frac{t}{4}}}{t^\alpha}\ dt\right|\leq C(\alpha) e^{-\frac{\min\{y_1,y_2\}}{4}}\left|y_1^{1-\alpha}-y_1^{1-\alpha}\right|\\
&\leq  C(\alpha) e^{-\frac{\min\{y_1,y_2\}}{4}}|y_1-y_2|^{1-\alpha}.
\end{align*}
Moreover, if $y_1<1\leq y_2$ we compute
\[\int_{y_1}^{y_2} \frac{e^{-s}}{s^\alpha}ds\leq \begin{cases}
    \frac{y_1}{1-\alpha}e^{-y_1}& 2y_1\leq y_2\\
    2^\alpha \int_{y_1}^{y_2}e^{-s}ds & y_1>\frac{1}{2}y_2\geq \frac{1}{2}
\end{cases}\leq C(\alpha)|y_1-y_2|e^{-y_1}.\]
This concludes the proof of the estimate \eqref{eq:prop.collection3}.

    \end{proof}

\section{The abstract instability framework and proof of Theorem~\ref{thm:main}}
\label{sec:instability}
In this section, we provide the theoretical background for our abstract instability framework and finally provide a proof of Theorem~\ref{thm:main}. It builds on a similar strategy as outlined in \cite{KRS21} which in turn relies on central ideas from \cite{m01, dr03} and \cite{kt59}, see also \cite{ET96}. As we are considering a critical transition between two regimes, in contrast to \cite{KRS21}, in the present text, we have opted to directly work with covering numbers instead of the concept of entropy numbers. This allows us to consider ``continuous'' and not only ``discrete'' information. Many of the auxiliary results will however be parallel to the arguments from \cite{KRS21}. We also refer to \cite{KS25} for a critical stability transition which however relies on the availability of singular value bases of the operator under consideration which simultaneously define the compact set under consideration.

We begin by recalling the central concept of the covering number function \cite{kt59}.

\noindent
\begin{defi}
Let $X$ be a metric space and let $K\subset X$ be a compact subset. A finite set $\{x_i\}_{i=1}^n\subset K$ is called a \textit{$t$-net for $K$} if
\begin{align*}
    K\subset \bigcup_{i=1}^n B_{t}(x_i),
\end{align*}
where $B_t(x)\coloneqq\{z\in X\colon d_X(z,x)< t\}$. The \textit{covering number function} $\cN$ is defined as
\begin{align*}
    \cN(K,t,d_X) \coloneqq
    \min\bigg\{n\in\bN\colon\ \exists \{x_i\}_{i=1}^n\text{ is a $t$-net for $K$} \bigg\}.
\end{align*}
The function
\begin{align*}
    \log\cN(K,t,d_X)
\end{align*}
is referred to as \textit{metric entropy}.
\end{defi}
For later convenience, we prove some useful properties of covering numbers.
\begin{prop}
Let $X$ be a metric space. The following properties hold true.
\begin{enumerate}
    \item (Monotonicity) If $K_1\subset K_2\subset X$, then
    \begin{align*}
        \cN(K_1,t,d_X) \leq \cN(K_2,t,d_X).
    \end{align*}
    \item (Scaling) If $X$ is a normed space, $K\subset X$ and $\alpha>0$, then
    \begin{align*}
        \cN(\alpha K,t,\|\cdot\|_X) = \cN(K,t/\alpha,\|\cdot\|_X).
    \end{align*}
    \item (Neighbourhood) If $K\subset X$, then
    \begin{align*}
        \cN(K+B_{t/3},t,d_X) \leq \cN(K,t/3,d_X).
    \end{align*}
    \item (Volume comparison) If $X=\bC^m$ and $K\subset X$, then
    \begin{align*}
        \frac{\mathrm{vol}_m(K)}{\mathrm{vol}_m(B_t)}\leq \cN(K,t,\|\cdot\|_2) \leq \frac{\mathrm{vol}_m(K+B_{t/2})}{\mathrm{vol}_m(B_{t/2})}.
    \end{align*}
    \item (Bounded linear operators) If $X,Y$ are normed spaces, $K\subset X$ and $S\in\cB(X,Y)$, then
    \begin{align*}
        \cN(S(K),t,\|\cdot\|_Y) \leq \|S\|\, \cN(K,t,\|\cdot\|_X).
    \end{align*}
\end{enumerate}
\end{prop}
\begin{proof}
Properties (1), (2) and (5) are easy to check. We now prove property (3): Let $\{x_i\}_{i=1}^n$ be a $(t/3)$-net for $K$. Then it is easy to check that it is also a $t$-net for $K+B_{t/3}$, which concludes the proof of (3). 

Finally, we prove property (4). For the lower bound, we consider a $t$-net $\{x_i\}_{i=1}^n\subset K$; comparing the volumes of $K$ and $\cup_{i=1}^n B_t(x_i)$, we get that
\begin{align*}
    \mathrm{vol}_m(K) \leq
    \mathrm{vol}_m\lc\cup_{i=1}^n B_t(x_i)\rc \leq 
    n\, \mathrm{vol}_m\lc B_t\rc,
\end{align*}
which proves the lower bound. For the upper bound, we consider a maximal $t$-discrete set $\{x_i\}_{i=1}^n\subset K$, namely a set that is maximal among all finite subsets of $K$ such that
\begin{align*}
    \|x_i-x_j\|_2\geq t,\quad i\neq j.
\end{align*}
By maximality, it is also a $t$-net. Moreover, the balls $B_{t/2}(x_i)$ are pairwise disjoint, and therefore
\begin{align*}
    n\,\mathrm{vol}_m(B_{t/2}) =
    \mathrm{vol}_m\lc \cup_{i=1}^n B_{t/2}(x_i) \rc \leq
    \mathrm{vol}_m(K+B_{t/2}),
\end{align*}
which concludes the proof of the upper bound.
\end{proof}

Building on this, in the following subsection, we will relate mapping properties of possibly nonlinear maps by relating these to a linear comparison operator whose singular values we will estimate, similarly as in \cite{KRS21}. 

\subsection{Metric entropy and singular values}
Let $X,Y$ be Hilbert spaces and let $T\in\cK(X,Y)$ be a compact operator. The goal of this section is to provide an upper bound for
\begin{align*}
    \log\cN(T(B_1),t,d_Y),
\end{align*}
where $B_1=B_1(0)$, in terms of the singular values of $T$ (see Theorem~\ref{thm:entropy_svalues}). We recall the Singular Value Decomposition (SVD) of $T$.
\begin{thm}[Singular Value Decomposition]
Let $T\in\cK(Z,X)$ be a compact operator between Hilbert spaces. Then there exist
\begin{enumerate}
    \item an orthonormal system $(\phi_k)_k\subset Z$,
    \item an orthonormal system $(\psi_k)_k\subset X$,
    \item a non-increasing sequence of positive values $(\sigma_k)_k$ with $\sigma_k\rightarrow 0$ as $k \rightarrow \infty$,
\end{enumerate}
such that the following representation holds:
\begin{align*}
    Tz = \sum_{k} \sigma_k (z,\phi_k) \psi_k.
\end{align*}
The sequence $\sigma_k=\sigma_k(T)$ is uniquely determined by $T$, and elements of the sequence are called \textit{singular values} of $T$.
\end{thm}
We are now ready to prove the main result of this subsection.
\begin{thm}
\label{thm:entropy_svalues}
Let $T\in\cK(Z,X)$ be a compact operator between Hilbert spaces. Given $t>0$, we define
\begin{align*}
    m_1 = \#\{\sigma_k\geq 2t\},\quad
    m_2 = \#\{\sigma_k\geq t/3\}.
\end{align*}
Then the following holds:
\begin{align*}
    \sum_{k=1}^{m_1} \log\lc \frac{\sigma_k}{t} \rc \leq
    \log\cN(T(B_1),t,\|\cdot\|_X) \leq m_2\log(12)+ \sum_{k=1}^{m_2} \log\lc \frac{\sigma_k}{t}\rc.
\end{align*}
In particular, we get that
\begin{align*}
    c\, m_1 \leq \log\cN(T(B_1),t,\|\cdot\|_X)
    \leq C\, m_2\lc 1+\log\lc \frac{\|T\|}{t} \rc \rc
\end{align*}
for some universal constants $c,C>0$.
\end{thm}
\begin{proof}
Let $K=T(B_1)$. For each $m$, we define
\begin{align*}
    E_m \coloneqq \bigg\{\sum_{k=1}^m \sigma_k x_k\psi_k\colon\ \sum_{k=1}^m|x_k|^2\leq 1\bigg\}.
\end{align*}
If we identify
\begin{align*}
    \Span\{\psi_k\colon k\leq m\} \cong \bC^m,
\end{align*}
then $E_m$ can be seen as an ellipsoid in $\bC^m$. We get that
\begin{align*}
    E_{m_1} \subset K \subset E_{m_2}+B_{t/3},
\end{align*}
where the second inclusion follows from the representation
\begin{align*}
    Tz = \sum_{k=1}^{m_2} \sigma_k (z,\phi_k) \psi_k +
    \sum_{k>m_2} \sigma_k (z,\phi_k)\psi_k,
\end{align*}
and from the following estimate for the norm of the second term
\begin{align*}
    \lc\sum_{k>m_2} \sigma_k^2 |(z,\phi_k)|^2\rc^{1/2} <
    t/3 \lc\sum_{k} |(z,\phi_k)|^2\rc^{1/2} \leq t/3.
\end{align*}
By the volume comparison property, we get that
\begin{align*}
    \prod_{k=1}^{m_1} \lc\frac{\sigma_k}{t}\rc = \frac{\mathrm{vol}_{m_1}(E_{m_1})}{\mathrm{vol}_{m_1}(B_{t})} \leq \cN(E_{m_1},t,\|\cdot\|_X) \leq 
    \cN(K,t,\|\cdot\|_X),
\end{align*}
which proves the lower bound. By the neighbourhood and volume comparison properties, we get that
\begin{align*}
    \cN(K,t,\|\cdot\|_X) &\leq
    \cN(E_{m_2}+B_{t/3},t,\|\cdot\|_X) \\ &\leq
    \cN(E_{m_2},t/3,\|\cdot\|_X) \\ &\leq
    \frac{\mathrm{vol}_{m_2}(E_{m_2}+B_{t/6})}{\mathrm{vol}_{m_2}(B_{t/6})}.
\end{align*}
Using the fact that $\sigma_k\geq 2(t/6)$ for every $k=1,\dots,m_2$, we conclude that $E_{m_2}+B_{t/6}\subset 2E_{m_2}$, and therefore
\begin{align*}
    \frac{\mathrm{vol}_{m_2}(E_{m_2}+B_{t/6})}{\mathrm{vol}_{m_2}(B_{t/6})} \leq
    2^{m_2} \frac{\mathrm{vol}_{m_2}(E_{m_2})}{\mathrm{vol}_{m_2}(B_{t/6})} \leq
    12^{m_2} \prod_{k=1}^{m_2} \lc \frac{\sigma_k}{t} \rc.
\end{align*}
This concludes the proof.
\end{proof}
\begin{example}[Polynomial decay]
\label{ex:polynomial}
Suppose that
\begin{align*}
    \sigma_k = k^{-\nu}
\end{align*}
for some $\nu>0$. Then $\sigma_k\geq t$ if and only if
\begin{align*}
    k\leq t^{-1/\nu},
\end{align*}
and Theorem~\ref{thm:entropy_svalues} implies that, for small $t$,
\begin{align*}
    c t^{-1/\nu} \leq \log\cN(K,t,\|\cdot\|_X)
    \leq C t^{-1/\nu} \log\lc \frac{\|T\|}{t} \rc.
\end{align*}
\end{example}
\begin{example}[Exponential decay]
Suppose that
\begin{align*}
    \sigma_k = \exp(-k^{\mu})
\end{align*}
for some $\mu>0$. Then $\sigma_k\geq t$ if and only if
\begin{align*}
    k \leq | \log(t) |^{1/\mu},
\end{align*}
and Theorem~\ref{thm:entropy_svalues} implies that, for small $t$,
\begin{align*}
    c|\log(ct)|^{1/\mu} \leq \log\cN(K,t,\|\cdot\|_X)
    &\leq C |\log(Ct)|^{1/\mu} \log\lc \frac{\|T\|}{t} \rc \\ &\leq
    C |\log(Ct)|^{\frac{1+\mu}{\mu}}.
\end{align*}
\end{example}

For later use, we also record a well-known lemma on the relation between singular value bounds for sums of operators.

\begin{lem}
\label{lem:sum}
Let $X,Y$ and $Z_1,Z_2 \subset Y$ be Hilbert spaces.
Let $A: X \rightarrow Y$ and $A_1:X \rightarrow Z_1$, $A_2: X \rightarrow Z_2$ with singular values $(\sigma_k(A))_{k}$, $(\sigma_k(A_1))_{k}$ and $(\sigma_k(A_2))_k$ be compact operators such that 
\begin{align*}
\|A f\|_{Y} \leq C(\|A_1 f\|_{Z_1} + \|A_2 f\|_{Z_2}) \mbox{ for all } f \in X.
\end{align*}
Then, $\sigma_{i+j-1}(A) \leq \sigma_i(A_1) + \sigma_j(A_2)$.
\end{lem}

\begin{proof}
The claim follows from the Courant minimax principle. Indeed,
\begin{align*}
\sigma_{i+j-1}(A) = \min\limits_{S} \max_{w\perp S, \ \|w\|_{X}=1} \|A w\|_{X},
\end{align*}
where $S \subset X$ is any subspace of dimension $\dim(S) = i+j-2$.
Again by the minimax principle, there are subspaces $U,V \subset Y$ of dimensions $i-1, j-1$, respectively, such that for $u\perp U, v \perp V$
\begin{align*}
 \|A_1 u\|_{Z_1} \leq \sigma_i(A_1) \|u\|_{X}, \ \|A_2 v\|_{Z_2} \leq  \sigma_j(A_2) \|v\|_{X}.
\end{align*}
As a consequence, for a subspace $S \subset X$ of dimension $\dim(S) = i+j-2$ containing $U,V$, we have
\begin{align*}
\sigma_{i+j-1}(A) & = \min\limits_{S} \max_{w\perp S, \ \|w\|_{X}=1} \|A w\|_{X} \leq C \min\limits_{S} \max_{w\perp S, \ \|w\|_{X}=1} (\|A_1 w\|_{Z_1} + \|A_2 w\|_{Z_2})\\
&\leq C(\sigma_i(A_1)) + \sigma_j(A_2)).
\end{align*}
\end{proof}

\subsection{The comparison strategy}
In this section, all the spaces considered will be Hilbert spaces. Consider the space $Y=\cB(W,W')$ and the subspace of Hilbert-Schmidt operators $Y_2=\cB_2(W,W')$. Let $K_Y\subset Y$ be a set of operators such that
\begin{equation}
\label{eq:lambda_ass}
    \max(\|\Lambda f\|_{W'},\|\Lambda' f\|_{W'}) \leq \|Af\|_{Z},\quad \Lambda\in K_Y,\ f\in W,
\end{equation}
where $A\in\cK(W,Z)$ is a compact operator. The goal of this section is to provide, under suitable assumptions on $A$, an upper bound for
\begin{align*}
    \log\cN(K_Y,t,d_Y)
\end{align*}
in terms of the singular values $(\tau_k)_k$ of $A$ (see Theorem~\ref{thm:comparison}). We will do so by showing that $K_Y$ is contained in $\iota (B_M)$ for some compact operator $\iota \in\cK(X,Y_2)$ with $X$ a suitable Hilbert space and for some $M>0$, and by relating the singular values $(\sigma_k)_k$ of $\iota$ with the singular values $(\tau_k)_k$ of $A$; Theorem~\ref{thm:entropy_svalues} will then imply the desired estimate.

As in the previous section, we consider the Singular Value Decomposition of $A$, given by an orthonormal system $(\phi_k)_k$ of $W$, an orthonormal system $(\psi_k)_k$ of $Z$ and the singular values $(\tau_k)_k$; we have that
\begin{align*}
    Af = \sum_k \tau_k f_k \psi_k,\quad 
    f_k\coloneqq (z,\phi_k)_{Z}.
\end{align*}
\begin{thm}
\label{thm:comparison}
Let $W,Z$ be Hilbert spaces, let $Y=\cB(W,W')$, let $A\in\cK(W,Z)$ and let $(\tau_k)_k$ be the singular values of $A$. Let $K_Y\subset Y$ be a subset such that
\begin{align*}
    \max(\|\Lambda f\|_{W'},\|\Lambda'f\|_{W'}) \leq \|Af\|_Z,\quad \Lambda\in K_Y,\ f\in W.
\end{align*}
If $C_{\tau}\coloneqq \sum_k\tau_k^{1/2}<+\infty$, then $K_Y$ is contained in $\iota(B_{C_{\tau}})$, where $X$ is a suitable Hilbert space and $\iota\colon X\rightarrow Y_2$ is an embedding with singular values given by $\sigma_{k,j}=\tau_k^{1/4}\tau_j^{1/4}$. In particular, the following estimate holds:
\begin{align*}
    \log\cN(K_Y,t,d_Y) \leq Cm\lc 1+\log\lc \frac{C_{\tau}\|\iota\|}{t} \rc \rc,
\end{align*}
where $C>0$ is a universal constant and
\begin{align*}
    m = \#\{\tau_k^{1/4}\tau_j^{1/4}\geq \tfrac{t}{3C_{\tau}}\}.
\end{align*}
\end{thm}
\begin{proof}
Let $(\phi_k)_k$ be the orthonormal system from the SVD of $A$. Then
\begin{align*}
    &|\langle \Lambda\phi_k,\phi_j \rangle| \leq \|\Lambda\phi_k\|_{W'} \|\phi_j\|_{W} \leq \|A\phi_k\|_Z = \tau_k,\\
    &|\langle \Lambda\phi_k,\phi_j \rangle| =
    |\langle \phi_k,\Lambda'\phi_j \rangle| \leq
    \|\phi_k\|_W \|\Lambda'\phi_j\|_{W'} \leq \|A\phi_j\|_Z =\tau_j,
\end{align*}
which implies that
\begin{equation}
\label{eq:matrix_entries}
    |\langle \Lambda\phi_k,\phi_j \rangle|^2 \leq
    \tau_k\tau_j.
\end{equation}
Recall that the norm on $Y_2=\cB_2(W,W')$ is defined by
\begin{align*}
    \|T\|_{Y_2}^2 \coloneqq
    \sum_{k,j} |\langle T\phi_k,\phi_j \rangle|^2.
\end{align*}
Moreover, $Y_2$ can be canonically identified with $W'\otimes W'$. An orthonormal basis of $W'\otimes W'$ is given by $(\phi_k'\otimes\phi_j')$, as the following holds in the identification:
\begin{align*}
    (T,\phi_k'\otimes\phi_j')_{W'\otimes W'}= \langle T\phi_k,\phi_j \rangle.
\end{align*}
Then the norm can be written as
\begin{align*}
    \|T\|_{W'\otimes W'}^2 = \sum_{k,j} |(T,\phi_k'\otimes\phi_j')_{W'\otimes W'}|^2.
\end{align*}
We now define the subspace $X\subset W'\otimes W'$ endowed with the norm
\begin{align*}
    \|T\|_X^2 = \sum_{k,j} \tau_k^{-1/2} \tau_j^{-1/2} |(T,\phi_k'\otimes\phi_j')_{W'\otimes W'}|^2.
\end{align*}
This defines a Hilbert space structure on $X$. As in the statement of the theorem, we then consider the embedding $\iota: X \rightarrow Y_2$. From \eqref{eq:matrix_entries}, we conclude that $K_Y\subset \iota(B_{C_{\tau}})$. Moreover, we have that $(\tau_k^{1/4}\tau_j^{1/4}\phi_k'\otimes\phi_j')_{k,j}$ is an orthonormal basis of $X$, and therefore
\begin{align*}
    \iota(T) &= \sum_{k,j} (T,\phi_k'\otimes\phi_j')_{W'\otimes W'}\, (\phi_k'\otimes\phi_j') \\ &=
    \sum_{k,j} \tau_k^{1/4}\tau_j^{1/4}(T,\tau_k^{1/4}\tau_j^{1/4}\phi_k'\otimes\phi_j')_{X} \, (\phi_k'\otimes\phi_j'),
\end{align*}
which implies that $\sigma_{k,j}=\tau_k^{1/4}\tau_j^{1/4}$. The final estimate follows from Theorem~\ref{thm:entropy_svalues}, together with the already recalled properties of covering numbers with respect to monotone inclusions and bounded linear operators, the latter being applied to the embedding $Y_2=\cB_2(W,W')\rightarrow Y=\cB(W,W')$ whose operator norm is bounded by $1$.
\end{proof}
Suppose now that we can prove that the hypotheses of Theorem~\ref{thm:comparison} are satisfied by
\begin{align*}
    \tau_k^{1/4}=f(k),
\end{align*}
where $f\colon[1,+\infty)\rightarrow (0,\eta]$ is a strictly decreasing continuous function, with
\begin{align*}
    f(1)=\eta\leq 1,\quad f(+\infty)=0.
\end{align*}
We denote by $g\colon (0,\eta]\rightarrow[1,+\infty)$ the inverse of $f$, i.e., $g=f^{-1}$. Then we have that
\begin{align*}
    \sigma_{k,j} = f(k)f(j) \leq \eta f(k\lor j)\leq f(k\lor j).
\end{align*}
If $\sigma_{k,j}\geq \tfrac{t}{3M}$, then necessarily $f(k\lor j)\geq \tfrac{t}{3M}$, and therefore
\begin{align*}
    k\lor j \leq  g\lc\tfrac{t}{3M}\rc.
\end{align*}
We conclude that with the notation from Theorem \ref{thm:comparison}
\begin{align*}
    m \leq \lc g\lc \tfrac{t}{3M} \rc \rc^2.
\end{align*}
We have just proved the following result.
\begin{cor}
\label{cor:comparison}
Let $W$ be a Hilbert space, let $K_Y\subset Y=\cB(W,W')$ and let $f\colon [1,+\infty)\rightarrow (0,\eta]$ be a strictly decreasing continuous function such that
\begin{align*}
    f(1)=\eta\leq 1,\quad f(+\infty)=0.
\end{align*}
Let $A\in\cK(W,Z)$ be a compact operator, with positive singular values given by $(\tau_k)_k$, where $\tau_k^{1/4}=f(k)$, such that the following holds:
\begin{align}
\label{eq:majorant}
   \max ( \|\Lambda z\|_{Z'},\|\Lambda' z\|_{Z'} ) \leq \|Az\|_W,\quad
    \Lambda\in K_Y,\ z\in Z.
\end{align}
Suppose that
\begin{align*}
    C_{\tau} \coloneqq \sum_k \tau_k^{1/2}<+\infty.
\end{align*}
Then the following estimate holds for the space $Y$:
\begin{equation}
\label{eq:final_estimate}
    \log\cN(K_Y,t,d_Y) \leq C\lc g\big(\tfrac{t}{3C_{\tau}}\big)\rc^2 \lc 1+
    \log\lc \frac{C_{\tau}}{t} \rc \rc,
\end{equation}
where $C>0$ is a universal constant and $g\coloneqq f^{-1}$.
\end{cor}

\begin{rmk}
\label{rmk:estimate}
We remark that in order to deduce an upper bound of the form \eqref{eq:final_estimate} it suffices to have an upper bound for the singular values $(\tau_k)_{k}$ at our disposal. Indeed, from this upper bound for $(\tau_k)_{k}$ we can construct an operator $\tilde{A}$, defined on the singular vectors $(\phi_k)_k$ of $A$ by $\tilde{A}(\phi_k):= \tau_k \psi_k$ such that the equality $\tau_k^{\frac{1}{4}}=f(k)$ holds for the singular values of $\tilde{A}$ and the estimate \eqref{eq:majorant} from Corollary  \ref{cor:comparison} remains valid.
\end{rmk}

\subsection{Covering number estimates for spaces of non-negative functions}
The constraints that defines the set $\cS$ of admissible absorption coefficients (see \eqref{eq:definition_K}) include smoothness and non-negativity conditions; in particular, the latter is needed in order for the radiative transfer equation to be well-posed. As argued in Sec.~\ref{sub:outline_proof}, we want to deduce a lower bound for the covering numbers of $\cS$. In this section, we show that imposing a non-negativity constraint on Sobolev spaces does not change the scaling of the entropy numbers for the corresponding embeddings.

\begin{prop}
\label{prop:non-negativity}
Let $K\subset\bR^{d}$ be a bounded domain, let $M>0$, let $\gamma>0$ and consider $\cS$ defined as in \eqref{eq:definition_S}. Then there exists a constant $c>0$ such that, for small $t>0$,
\begin{align*}
    \log\cN(\cS,t,\|\cdot\|_{L^2}) \geq ct^{-d/\gamma}.
\end{align*}
\end{prop}
\begin{proof}
We can suppose without loss of generality that $K$ is a cube $Q$, as $\cS$ contains the set
\begin{align*}
    \{\sigma_a\in H^{\gamma}_0(Q)\colon\ \sigma_a\geq 0,\ \|\sigma_a\|_{H^{\gamma}(Q)}\leq M\}
\end{align*}
for a suitable small cube $Q$. For simplicity, we will suppose that $Q=[0,1]^d$ and that $M=1$.

We consider a fixed smooth non-negative function $\Psi\in H^{\gamma}_0(Q)$ with $\|\Psi\|_{H^{\gamma}(Q)}=1$. For every $N\in\bN$, we can rescale it as
\begin{align*}
    \Psi_N = 2^{-N(\gamma-d/2)} \Psi(2^N \cdot)
\end{align*}
and consider the translates
\begin{align*}
    \Psi_{k,N} = \tau_{k/2^N}\Psi_N,\quad k\in\{0,\dots,2^{N}-1\}^d.
\end{align*}
We then consider the spaces
\begin{align*}
    \cM_N = \bigg\{ \sum_k a_k \psi_{k,N}\colon\ 
    a_k\geq 0,\ \sum_k|a_k|^2\leq 1\bigg\}.
\end{align*}
Let $\tilde{\gamma}\in\bN$. By the choice of the rescaling, we get that, for all $\sigma_{a}=\sum_k a_k \Psi_{k,N}\in \cM_N$,
\begin{align*}
    \|\sigma_{a}\|_{H^{\tilde{\gamma}}(Q)}^2 =
    \sum_{k} |a_k|^2 \|\Psi_{k,N}\|_{H^{\tilde{\gamma}}(Q)}^2 \lesssim 2^{2N(\tilde{\gamma}-\gamma)} \sum_{k} |a_k|^2.
\end{align*}
This implies that the map $\ell^2\ni (a_k)_k\mapsto \sum_k a_k \Psi_{k,N}\in H^{\tilde{\gamma}}(Q)$ has operator norm bounded by $C2^{N(\tilde{\gamma}-\gamma)}$. By interpolation, it follows that
\begin{align*}
    \|\sigma_a\|_{H^{\gamma}(Q)}^2 \lesssim \sum_{k} |a_k|^2,
\end{align*}
and therefore, possibly up to multiplying the set by a constant, we have that $\cM_N\subset\cS$. Moreover, we get that, for $\sigma_{a,j}=\sum_{k}a_k^j\Psi_{k,N}\in \cM_N$ ($j=1,2$),
\begin{align*}
    \|\sigma_{a,1}-\sigma_{a,2}\|_{L^2(Q)}^2 = 
    \sum_{k} |a_k^1-a_k^2|^2 \|\Psi_{k,N}\|_{L^2(Q)}^2 =
    2^{-2N\gamma} \sum_{k}|a_k^1-a_k^2|^2. 
\end{align*}
This implies that
\begin{align*}
    \cN(\cS,t,\|\cdot\|_{L^2}) \geq
    \cN(B^{2^{dN},+}_1,t,2^{-N\gamma}|\cdot|_2) = 
    \cN(B^{2^{dN},+}_1,2^{N\gamma}t,|\cdot|_2),
\end{align*}
where
\begin{align*}
    B^{2^{dN},+}_1 = \{a\in\bR^{2^{dN}}\colon\ a_k\geq 0,\ \sum_{k}|a_k|^2\leq 1\}.
\end{align*}
We now have that
\begin{align*}
    \cN(B^{2^{dN},+}_1,2^{N\gamma}t,|\cdot|_2) &\geq
    \frac{\vol_{2^{dN}}(B^{2^{dN},+}_1)}{\vol_{2^{dN}}(B^{2^{dN}}_{2^{N\gamma}t})} \\ &=
    2^{-2^{dN}}\frac{\vol_{2^{dN}}(B^{2^{dN}}_1)}{\vol_{2^{dN}}(B^{2^{dN}}_{2^{N\gamma}t})} =
    2^{-(1+N\gamma)2^{dN}} t^{-2^{dN}}.
\end{align*}
We infer that, for every $N\geq 1$,
\begin{align*}
    \log_{2}\cN(\cS,t,\|\cdot\|_{L^2}) \geq
    2^{dN} \lc\log_{2}(1/t)-(1+N\gamma)\rc.
\end{align*}
Choosing $N=\frac{1}{\gamma}\log_2(1/t)-\frac{2}{\gamma}$, we get that
\begin{align*}
    2^{dN} \lc\log_{2}(1/t)-(1+N\gamma)\rc = 2^{-2d/\gamma} t^{-d/\gamma}
\end{align*}
which concludes the proof.
\end{proof}

\subsection{Abstract instability result and proof of Theorem \ref{thm:main}}
Building on the above arguments, we finally obtain the desired abstract instability estimate.
\begin{thm}
\label{thm:final_abstract_instability}
Let $X$ be a metric space, let $\cS\subset X$ be a compact set, let $W$ be a Hilbert space, and let
\begin{align*}
    \Lambda\colon X\rightarrow \cB(W,W')
\end{align*}
be a continuous map. Suppose that the following conditions are satisfied:
\begin{enumerate}
    \item there exist $c,\alpha>0$ such that
    \begin{equation}
    \label{eq:lower_bound}
        \log\cN(\cS,t,d_X) \geq ct^{-\alpha},
    \end{equation}
    \item there exist a strictly decreasing continuous function $f\colon[1,+\infty)\rightarrow(0,\eta]$, with $f(1)=\eta\leq 1$ and $f(+\infty)=0$, and a compact operator $A\in\cK(W,Z)$, with singular values given by $(\tau_k)_k$ where
    \begin{align*}
        \tau_k^{1/4}=f(k),
    \end{align*}
    such that
    \begin{align*}
        C_{\tau} \coloneqq \sum_{k} \tau_k^{1/2}<+\infty
    \end{align*}
    and such that the following holds:
    \begin{align*}
        \max( \|\Lambda(x)f\|_{W'},\|\lc\Lambda(x)\rc'f\|_{W'})\leq \|Af\|_Z,\quad x\in \cS,\ f\in W.
    \end{align*}
\end{enumerate}
If
\begin{equation}
\label{eq:cond_stab}
    \|x_1-x_2\|_X \leq \omega(\|\Lambda(x_1)-\Lambda(x_2)\|),\quad x_1,x_2\in \cS
\end{equation}
for some modulus of continuity $\omega$, then
\begin{equation}
\label{eq:final_estimate2}
    \omega(t) \geq
    \lc C \lc g\big(\tfrac{t}{3C_{\tau}}\big)\rc^2
    \lc 1+
    \log\lc \frac{C_{\tau}}{t} \rc \rc \rc^{-1/\alpha},
\end{equation}
where $C>0$ depends only on $c$ in \eqref{eq:lower_bound} and $g\coloneqq f^{-1}$.
\end{thm}
\begin{proof}
Condition \eqref{eq:cond_stab} implies that
\begin{align*}
    \log\cN(\cS,\omega(t),d_X) \leq \log\cN(\Lambda(\cS),t,d_Y).
\end{align*}
A lower bound for the left hand side is provided by \eqref{eq:lower_bound}, namely
\begin{align*}
    c\lc\omega(t)\rc^{-\alpha} \leq \log\cN(\cS,\omega(t),d_X).
\end{align*} 
An upper bound for the right hand side is provided by Corollary~\ref{cor:comparison}, namely
\begin{align*}
    \log\cN(K_Y,t,d_Y) \leq C\lc g\big(\tfrac{t}{3C_{\tau}}\big)\rc^2 \lc 1+
    \log\lc \frac{C_{\tau}}{t} \rc \rc.
\end{align*} 
This concludes the proof.
\end{proof}

With Theorem~\ref{thm:final_abstract_instability} and the diffusion approximation estimates from Sec.~\ref{sub:albedo} in hand, we can finally turn to the proof of Theorem~\ref{thm:main}.

\begin{proof}[Proof of Theorem~\ref{thm:main}, radiative transfer in diffusive regime]
First, Proposition~\ref{prop:non-negativity} implies that $\cS$ satisfies \eqref{eq:lower_bound} with $\alpha=d/\gamma$. Moreover, from Proposition~\ref{prop:estimates} we conclude that, for $\sigma_a\in\cS\subset\cK$,
\begin{align*}
    \|(\Lambda_{\sigma_a}-\Lambda_0) f\|_{H^{-s}(\partial D)} \leq C( \|\rho_{0,0}\|_{H^{s_0}(K)}+K_n\|f\|_{H^{s_1}(\partial D)} ),
\end{align*}
where $\rho_{0,0}$ is defined in Sec.~\ref{sub:diffusion}. In addition, Proposition~\ref{prop:dual_operator} implies that $(\Lambda_{\sigma_a}-\Lambda_0)'$ satisfies the same bounds, as \eqref{eq:backward_rte} reduces to \eqref{eq:ua_def} with $F=0$ and $G=g$ via the change of variable $(x,v)\mapsto (x,-v)$.

Arguing as in \cite[Theorem 5.7]{KRS21}, we can conclude that the map
\begin{align*}
   A_1: H^{s}(\partial D)\ni f\mapsto \rho_{0,0}|_K \in H^{s_0}(K)
\end{align*}
has singular values upper bounded by $C\exp(-ck^{\mu})$ for some $C,c$ and for $\mu=1/d$. On the other hand, by classical entropy estimates for embeddings between Sobolev spaces (see, for instance, \cite{ET96}), we conclude that the map
\begin{align*}
   A_2: H^s(\partial D)\ni f\mapsto f\in H^{s_1}(\partial D)
\end{align*}
has singular values upper bounded by $Ck^{-\nu}$ for some $C>0$ and for $\nu=(s-s_1)/d>2$. Using Lemma \ref{lem:sum} we can then conclude that \eqref{eq:lambda_ass} is satisfied by the singular values $(\tau_k)_k$ of $\Lambda_{\sigma_a}-\Lambda_0$ with
\begin{align*}
    \tau_k   \leq \sigma_{\frac{k}{2}}(A_1) + \sigma_{\frac{k}{2}}(A_2) \leq    C\max\{\exp(-ck^{\mu}),K_n k^{-\nu}\},
\end{align*}
where $\mu>0$, $\nu>2$ and $K_n\in(0,1]$. Notice that $\tau_k$ is $\ell^{1/2}$-summable; moreover, using Remark \ref{rmk:estimate}, after possibly a small modification as outlined in the remark, it is possible to write
\begin{align*}
    \tau_k^{1/4} = f(k)
\end{align*}
where $f$ is defined by
\begin{align*}
    f(s) = C\max\{\exp(-cs^{\mu}/4),K_n^{1/4}s^{-\nu/4}\}.
\end{align*}
The inverse of $f$ is given by
\begin{align*}
    g(t) = C_1\max\{|\log(C_2 t)|^{1/\mu},K_n^{1/\nu} t^{-4/\nu}\}.
\end{align*}
Hence, by Corollary~\ref{cor:comparison}, we conclude that \eqref{eq:final_estimate2} holds for such $g$, and therefore
\begin{align*}
    \omega(t) \geq C_3 \min\{|\log(C_4 t)|^{-2\gamma/ d\mu},K_n^{-\gamma/d\nu}t^{8\gamma/d \nu}\}\lc 1+\log\lc \frac{C_{\tau}}{t} \rc \rc^{-\gamma/d}.
\end{align*}
Substituting the values for $\mu$ and $\nu$ in the estimates concludes the proof of the theorem.
\end{proof}

\section*{Acknowledgement}
All authors gratefully acknowledge support by the Deutsche Forschungsgemeinschaft (DFG) through CRC 1720 ``Analysis of criticality: from complex phenomena to models and estimates'', 539309657. A.F. and A.R. also gratefully acknowledge support by the Deutsche Forschungsgemeinschaft (DFG) through the Leibniz prize.

\bibliographystyle{alpha}
\bibliography{citationsFT}

\end{document}